\documentclass[11pt,leqno]{amsart}
\usepackage{epsfig}
\usepackage{amssymb}
\usepackage{amscd}
\usepackage{ifthen}
\usepackage{enumitem}
\setenumerate[1]{label=(\alph*)}
\setenumerate[2]{label=(\roman*)}

\usepackage{amsthm}
\usepackage{newlfont}
\usepackage{amsmath}
\usepackage{amsfonts}
\usepackage[mathscr]{euscript}
\usepackage[ngerman,french,english]{babel}
\usepackage{mathrsfs}

\usepackage{mathtools}

\usepackage{fancybox}
\usepackage{nicefrac}
\usepackage{pifont}
\usepackage[all]{xy}
\usepackage{stmaryrd}
\usepackage{srcltx}
\usepackage{ifthen}
\usepackage{ulem}
\usepackage{accents}
\usepackage{leftidx}

\usepackage{hyperref}
\hypersetup{colorlinks=true, anchorcolor=red, citecolor= blue,
  filecolor=magenta, linkcolor=red, menucolor=red, urlcolor=blue}

\newtheorem{theorem}{Theorem}[section]
\newtheorem{lemma}[theorem]{Lemma}
\newtheorem{definition}[theorem]{Definition}
\newtheorem{proposition}[theorem]{Proposition}

\newtheorem{corollary}[theorem]{Corollary}
\newtheorem{remark}[theorem]{Remark}

\newtheorem{example}[theorem]{Example}

\newtheorem{question}[theorem]{Question}
\numberwithin{equation}{section}

\newcommand{\bR}{\mathbf{R}}
\newcommand{\bL}{\mathbf{L}}

\newcommand{\Tot}{\operatorname{Tot}}
\newcommand{\im}{\operatorname{im}}

\newcommand{\colim}{\operatorname{colim}}

\newcommand{\Sym}{\operatorname{Sym}}

\newcommand{\Spec}{\operatorname{Spec}}

\newcommand{\id}{\operatorname{id}}

\newcommand{\dg}{{\operatorname{dg}}}

\newcommand{\Proj}{\operatorname{Proj}}

\DeclareMathAlphabet{\mathpzc}{OT1}{pzc}{m}{it}

\newcommand{\matfak}[4]{{\xymatrix@C3ex{{{#1}} \ar@<0.3ex>[r]^-{{#2}} & {{#3}} \ar@<0.3ex>[l]^-{{#4}}}}}

\newcommand{\op}{\operatorname}

\newcommand{\ra}{\rightarrow}

\newcommand{\la}{\leftarrow}

\newcommand{\sra}{\twoheadrightarrow}
\newcommand{\hra}{\hookrightarrow}

\newcommand{\xra}[1]{\xrightarrow{#1}}
\newcommand{\xla}[1]{\xleftarrow{#1}}

\newcommand{\sira}{\xra{\sim}}
\newcommand{\xsira}[1]{\xrightarrow[\sim]{#1}}

\newcommand{\sila}{\overset{\sim}{\leftarrow}}

\newcommand{\xsra}[1]{\overset{#1}{\twoheadrightarrow}}
\newcommand{\xhra}[1]{\overset{#1}{\hookrightarrow}}

\newcommand{\mar}{\ar@{|->}}
\newcommand{\sar}{\ar@{->>}}
\newcommand{\iar}{\ar@{^{(}->}}
\newcommand{\gar}{\ar@{=}}
\newcommand{\gleichar}{\ar@{}|{=}}
\newcommand{\congar}{\ar@{}|{\cong}}

\newcommand{\Bl}[1]{{\mathbb{#1}}}
\newcommand{\DZ}{\Bl{Z}}
\newcommand{\DN}{\Bl{N}}

\newcommand{\DA}{{\Bl{A}}}
\newcommand{\DP}{{\Bl{P}}}

\newcommand{\DL}{\Bl{L}}

\newcommand{\Hom}{\op{Hom}}

\newcommand{\Mor}{\op{Mor}}

\newcommand{\Sh}{{\op{Sh}}}

\newcommand{\Qcoh}{{\op{Qcoh}}}
\newcommand{\Coh}{{\op{Coh}}}
\newcommand{\mfPerf}{{{\mathfrak P}{\mathfrak e}{\mathfrak r}{\mathfrak f}}}
\newcommand{\Sg}{{\op{Sg}}}

\newcommand{\MF}{\op{MF}}
\newcommand{\InjSh}{{\op{InjSh}}}
\newcommand{\InjQcoh}{{\op{InjQcoh}}}

\newcommand{\Locfree}{\op{Locfree}}
\newcommand{\FlatQcoh}{\op{FlatQcoh}}

\newcommand{\Acycl}{\op{Acycl}}
\newcommand{\Ex}{\op{Ex}}

\newcommand{\Cechmor}{{\mathrm{\check{C}mor}}}
\newcommand{\Cechobj}{{\mathrm{\check{C}obj}}}

\newcommand{\alt}{{\op{alt}}}
\newcommand{\ord}{{\op{ord}}}

\newcommand{\DSh}{\op{DSh}}
\newcommand{\DQcoh}{\op{DQcoh}}
\newcommand{\DCoh}{\op{DCoh}}
\newcommand{\bfMF}{\mathbf{MF}}

\newcommand{\DLocfree}{\op{DLocfree}}
\newcommand{\DFlatQcoh}{\op{DFlatQcoh}}

\newcommand{\AcyclMF}{\op{AcyclMF}}
\newcommand{\AcyclCoh}{\op{AcyclCoh}}
\newcommand{\AcyclQcoh}{\op{AcyclQcoh}}

\newcommand{\co}{{\op{co}}}

\newcommand{\sheafHom}{{\mkern3mu\mathcal{H}{om}\mkern3mu}}

\newcommand{\End}{\op{End}}

\newcommand{\pr}{\op{pr}}

\newcommand{\ol}[1]{{\overline{#1}}}
\newcommand{\ul}[1]{{\underline{#1}}}

\newcommand{\charakteristik}{\op{char}}

\newcommand{\Kern}{\op{ker}}

\newcommand{\Bild}{\op{im}}

\newcommand{\coker}{\op{cok}}
\newcommand{\cokern}{\op{cok}}

\newcommand{\cek}{\vee}

\newcommand{\inv}{^{-1}}
\newcommand{\can}{\op{can}}

\newcommand{\primp}{\mathfrak{p}}

\newcommand{\transposed}{{\operatorname{t}}}
\newcommand{\tp}{{\transposed}}

\newcommand{\sat}{\op{sat}}

\newcommand{\tzmat}[4]{{\left[\begin{smallmatrix} {#1} & {#2} \\ {#3} & {#4} \end{smallmatrix}\right]}}

\newcommand{\tildew}[1]{\widetilde{#1}}

\newcommand{\Var}{{\op{Var}}}

\newcommand{\comp}{\circ}

\newcommand{\opp}{{\op{op}}}

\newcommand{\gr}{\op{gr}}

\newcommand{\free}{\op{free}}

\newcommand{\mathovalbox}[1]{{\text{\ovalbox{${#1}$}}}}

\newcommand{\define}[1]{{\textbf{#1}}}

\numberwithin{equation}{section}

\newcommand{\zvek}[2]{{\left[\begin{smallmatrix} {#1} & {#2} \end{smallmatrix}\right]}}
\newcommand{\svek}[2]{{\left[\begin{smallmatrix} {#1} \\ {#2} \end{smallmatrix}\right]}}

\newcommand{\Cone}{\op{Cone}}

\renewcommand{\epsilon}{{\varepsilon}}

\newcommand{\strict}{\op{strict}}
\newcommand{\tria}{\op{tria}}
\newcommand{\thick}{\op{thick}}

\newcommand{\anticomm}{\circleddash}

\author{Valery A.~Lunts \and Olaf M.~Schn{\"u}rer}

\address{
  Department of Mathematics\\
  Indiana University\\
  Rawles Hall\\
  831 East 3rd Street\\
  Bloomington, IN 47405\\
  USA
}

\email{vlunts@indiana.edu} 

\email{oschnure@indiana.edu, olaf.schnuerer@math.uni-bonn.de}

\title[Matrix factorizations and semi-orthogonal
decompositions]{Matrix factorizations and semi-orthogonal
  decompositions for blowing-ups} 

\begin{document}

\begin{abstract}
  We study categories of matrix factorizations. These categories
  are defined for any regular function on a suitable regular
  scheme. Our paper has two parts.
  In the first part we develop the foundations; 
  for example we discuss derived direct and inverse image
  functors  
  and dg enhancements.
  In the second part we prove that the category of matrix
  factorizations on the 
  blowing-up of a suitable regular scheme $X$ along a
  regular closed 
  subscheme $Y$ has a semi-orthogonal decomposition into
  admissible subcategories in terms of
  matrix factorizations on $Y$ and $X.$
  This is the analog of a well-known theorem for bounded derived
  categories of coherent sheaves, and is an essential step in our
  forthcoming article 
  \cite{valery-olaf-matfak-motmeas-in-prep} which
  defines a Landau-Ginzburg motivic measure using
  categories of matrix factorizations.
  Finally we explain some applications.
\end{abstract}

\maketitle
\tableofcontents

\section{Introduction}
\label{sec:introduction}

Let $X$ be a 
separated regular Noetherian scheme of finite
Krull dimension over a field $k,$ 
for example a regular quasi-projective scheme over $k.$
Let $W \in
\Gamma(X, \mathcal{O}_X)$ be a regular function on $X.$
A matrix factorization $E$ of $W$ is a diagram
\begin{equation*}
  E =  (\matfak{E_1}{e_1}{E_0}{e_0})
\end{equation*}
of locally free sheaves of finite type (= vector bundles) 
on $X$ such that $e_0e_1=W \id_{E_1}$ and $e_1e_0=W \id_{E_0}.$
These diagrams are the objects of a differential $\DZ_2$-graded
category. Its homotopy category is a triangulated category, and
the category $\bfMF(X,W)$ of matrix factorizations of $W$ is
defined as a certain Verdier quotient of this triangulated
category, see \cite{orlov-mf-nonaffine-lg-models}.

Let $\pi \colon \tildew{X} \ra X$ be the
blowing-up of $X$ along 
a regular equi-codimensional closed
subscheme $Y.$
Consider the pullback diagram
\begin{equation*}
  \xymatrix{
    {E} \ar[r]^-{j} \ar[d]^-{p} &
    {\tildew{X}} \ar[d]^-{\pi} \\
    {Y} \ar[r]^-i &
    {X.}
  }
\end{equation*}

The usual construction of the blowing-up
endows $\tildew{X}$ with a line bundle
$\mathcal{O}_{\tildew{X}}(1).$ 
We denote its restriction to $E$ by
$\mathcal{O}_E(1).$
We denote the pullback functions of $W$ to $Y,$ $\tildew{X}$ and
$E$ by the same symbol. Then $\pi$ and $p$ induce (left derived)
inverse image 
functors $\pi^* \colon  \bfMF(X, W) \ra \bfMF(\tildew{X}, W)$
and $p^* \colon  \bfMF(Y, W) \ra \bfMF(E, W).$ Similarly, $j$ gives rise
to a (right derived) direct image functor 
$j_* \colon  \bfMF(E, W) \ra \bfMF(\tildew{X}, W)$ (strictly speaking this
functor does not land in $\bfMF(\tildew{X},W)$ but in an
equivalent bigger category). 
Now we can state our main theorem.
It is the analog of a well-known result for bounded derived
categories of coherent sheaves.

\begin{theorem}
  [see Theorem~\ref{t:semi-orthog-2-mf}]
  \label{t:semi-orthog-2-mf-intro}
  Assume that the codimension $r$ of $Y$ in $X$ is $\geq 2,$ and
  let $l \in \DZ.$ Then
  the functors
  $\pi^*  \colon \bfMF(X,W)\ra \bfMF(\tildew{X},W)$ 
  and 
  \begin{equation*}
    j_*(\mathcal{O}_E (l)\otimes p^*(-)) \colon 
    \bfMF(Y,W) \ra \bfMF(\tildew{X}, W)  
  \end{equation*}
  are full and faithful.
  Their essential images 
  $\pi^*\bfMF(X,W)$ and
  $\bfMF(Y,W)_{l}$
  in $\bfMF(\tildew{X}, W)$ are admissible
  subcategories, and we have a semi-orthogonal decomposition
  \begin{equation*}
    \bfMF(\tildew{X},W)=
    \Big\langle \bfMF(Y,W)_{-r+1}, \dots,\bfMF(Y,W)_{-1}, \pi^*
    \bfMF(X,W)\Big\rangle.
  \end{equation*}
\end{theorem}

This result is proved in the second part
(section~\ref{sec:semi-orth-decomp}) 
of this article. As a predecessor we prove 
Theorem~\ref{t:semi-orthog-1-mf} which provides semi-orthogonal
decompositions for projective space bundles.
We also discuss some applications.

In the first part 
(section~\ref{sec:categories-curved-dg})
we discuss general results on
categories of matrix factorizations.
Certainly categories of global matrix factorizations have been
around 
for a while 
\cite{lin-pomerleano, orlov-mf-nonaffine-lg-models}
but there is no systematic treatment of the
general theory, with the exception of
\cite{positselski-coh-analogues-matrix-fact-sing-cats, positselski-two-kinds}
which contains many of our results
(usually in a more general context).
Here is an outline of the main results.
First we define triangulated categories
$\DCoh(X,W)$ and $\DQcoh(X,W)$
in essentially the same way as $\bfMF(X,W)$
by using coherent (resp.\ quasi-coherent) sheaves instead of
vector bundles.
There are natural functors 
\begin{equation*}
  \bfMF(X,W) \ra \DCoh(X,W) \ra \DQcoh(X,W).
\end{equation*}
We show that the first functor is an equivalence and the second
one is 
full and faithful (see Theorem~\ref{t:equivalences-curved-categories}).

Assume that $Y$ is another separated regular Noetherian scheme of 
finite Krull dimension over $k.$
Let $\pi \colon  Y \ra X$ be a morphism of schemes over $k.$
The usual direct and inverse image functors
$\pi_*$ and $\pi^*$ between
categories of quasi-coherent sheaves
give rise to functors 
$\bR \pi_* \colon  \DQcoh(Y,W) \ra \DQcoh(X,W)$ and 
$\bL \pi^* \colon  \DQcoh(X,W) \ra \DQcoh(Y,W).$ 
This is deduced from the general theory of
derived functors. Moreover, 
there is an adjunction $(\bL \pi^*, \bR \pi_*)$
(see Theorem~\ref{t:derived-inverse-and-direct-image}).
Similarly, we define functors 
$\bR\sheafHom(-,-)$ and $(- \otimes^\bL -).$

We then describe 
several (differential $\DZ_2$-graded) enhancements of $\bfMF(X,W)$
(and 
$\DQcoh(X,W)$) and show that they are equivalent
(see section~\ref{sec:enhancements}).
They are constructed using injective quasi-coherent sheaves,
Drinfeld dg quotient categories, and \v{C}ech resolutions,
respectively. 
Finally we show 
that the subcategory of compact objects in
$\DQcoh(X,W)$ is the Karoubi envelope of $\bfMF(X,W),$  
and that $\bfMF(X,W)$ has a classical generator
(see section~\ref{sec:compact-generators}).

In two appendices
we collect some results on admissible subcategories and
semi-orthogonal decompositions 
(appendix~\ref{sec:app:remind-admiss-subc})
and on embeddings of Verdier
quotients (appendix~\ref{sec:app:embedd-verd-quot}).

This article is part of our project to construct
motivic measures  
using categories of matrix factorizations.
We sketch our main results. They will appear in forthcoming
articles.

We now assume that $k$ is algebraically closed and of
characteristic zero.  
Denote
by $K_0(\Var_{\DA^1})$ the motivic Grothendieck group
of varieties over $\DA^1:=\DA^1_k.$ 
Given $W \colon  X \ra \DA^1$ and $V \colon Y \ra \DA^1$ we define
$W * V \colon  X \times Y \ra \DA^1$ by $(W*V)(x,y)=W(x)+V(y).$ This
operation turns $K_0(\Var_{\DA^1})$ into a commutative ring.
By a Landau-Ginzburg motivic measure we mean a morphism of
rings from 
$K_0(\Var_{\DA^1})$ to some other ring.

Given a smooth 
variety $X$ and $W \colon  X \ra \DA^1$
we define the 
category 
of singularities of $W$ as
\begin{equation*}
  \bfMF(W):=\prod_{a \in k} \textbf{MF}(X, W-a).
\end{equation*}
Only finitely many factors of this product are non-zero, and
$\bfMF(W)$ vanishes if and only if $W$ is a smooth morphism.
Let $\bfMF(W)^{\dg,\natural}$ be a suitable enhancement of 
the Karoubi envelope of $\bfMF(W).$
If $W$ is a proper morphism, $\bfMF(W)^{\dg,\natural}$ is a
saturated dg (= differential $\DZ_2$-graded) category.

We denote by $K_0(\sat_k^{\DZ_2})$ the
free abelian group generated by the
quasi-equivalence classes of saturated dg (= differential
$\DZ_2$-graded) categories with
relations coming from semi-orthogonal decompositions into
admissible subcategories on the level of homotopy categories.
The tensor product of dg categories induces a ring structure on
$K_0(\sat_k^{\DZ_2}).$ One may think of $K_0(\sat_k^{\DZ_2})$ as a Grothendieck
ring of suitable pretriangulated dg categories.
Here is the main result of
the forthcoming article 
\cite{valery-olaf-matfak-motmeas-in-prep}. 

\begin{theorem}
  \label{t:category-of-singularities-induces-ring-morphism-intro}
  There is a unique morphism
  \begin{equation}
    \label{eq:MF-mot-meas}
    \mu \colon  K_0(\Var_{\DA^1})\ra K_0(\sat_k^{\DZ_2})
  \end{equation}
  of rings (= a Landau-Ginzburg motivic measure) that maps
  $[X,W]$ to 
  the class of $\bfMF(W)^{\dg,\natural}$ whenever $X$ is a smooth
  variety and
  $W \colon  X \ra \DA^1$ is a proper morphism. 

  In particular, $\mu$ is a morphism of abelian groups and maps 
  $[X,W]$ to 
  the class of $\bfMF(W)^{\dg,\natural}$ whenever $X$ is a smooth
  (connected) variety and
  $W \colon  X \ra \DA^1$ is a projective morphism. These two properties
  determine $\mu$ uniquely.
\end{theorem}

Since $K_0(\Var_{\DA^1})$ has a presentation whose relations come from
suitable blowing-ups (see
\cite[Thm.~5.1]{bittner-euler-characteristic}),
Theorem~\ref{t:semi-orthog-2-mf-intro}
and its predecessor Theorem~\ref{t:semi-orthog-1-mf}   
essentially imply that there is a unique morphism 
$\mu \colon  K_0(\Var_{\DA^1}) \ra K_0(\sat_k^{\DZ_2})$
of abelian
groups 
sending $[X,W]$ to the class of $\bfMF(W)^{\dg,\natural}$ if $X$
is a smooth variety
and $W$ is a proper morphism.
Here we implicitly use the fact mentioned above that
$\bfMF(W)^{\dg,\natural}$ is a saturated dg category for proper
$W.$ This fact and multiplicativity of $\mu$ is established in 
\cite{valery-olaf-matfak-motmeas-in-prep}.
We also give a careful definition of $K_0(\sat_k^{\DZ_2})$ there.  

Theorem~\ref{t:category-of-singularities-induces-ring-morphism-intro}
above was motivated by 
and is a relative version of 
a result by 
A.~Bondal, M.~Larsen and the first author
(see \cite[8.2]{bondal-larsen-lunts-grothendieck-ring}): 
they construct a morphism of rings 
\begin{equation*}
  K_0(\Var_k) \ra K_0(\sat_k^\DZ)  
\end{equation*}
(= a motivic measure)
that maps the class of a smooth projective variety $X$ over $k$
to the 
class of the standard enhancement of $D^b(\Coh(X))$ by bounded
below complexes of injective sheaves with bounded coherent
cohomologies; here $K_0(\Var_k)$ is the Grothendieck group of
varieties over $k,$ and 
$K_0(\sat_k^\DZ)$ is defined similarly as 
$K_0(\sat_k^{\DZ_2})$ starting from saturated differential
$\DZ$-graded categories.

In the article \cite{olaf-folding-derived-categories-in-prep} we show that the above two motivic measures are connected by a
commutative diagram  
\begin{equation*}
  \xymatrix{
    {K_0(\Var_k)} \ar[r] \ar[d]
    & {K_0(\sat_k^{\DZ})} \ar[d] \\
    {K_0(\Var_{\DA^1})} \ar[r]^-{\mu} 
    & {K_0(\sat_k^{\DZ_2})} 
  }
\end{equation*}
of ring morphisms where the vertical morphism on the left maps
$[X]$ to $[X,0]$ and the vertical morphism on the right is
induced by folding a differential $\DZ$-graded category into a
differential $\DZ_2$-graded category (and taking its triangulated
envelope).
The upper (resp.\ lower) horizontal arrow maps 
$\DL_k:=[\DA^1]$ (resp.\ $\DL_{\DA^1}:=[\DA^1,0]$) to 1.

In the article \cite{valery-olaf-fiber-measure-in-prep}
we prove that the motivic vanishing fiber map 
\begin{equation*}
  \phi \colon  K_0(\Var_{\DA^1}) \ra \mathcal{M}^{\hat{\mu}}_k
\end{equation*}
to the equivariant Grothendieck ring 
$\mathcal{M}^{\hat{\mu}}_k$ is also a Landau-Ginzburg motivic
measure (here 
$\hat{\mu}$ is the projective limit of the group schemes $\mu_n$
of $n$-th roots of unity). We show
that it is related to the above measure
\eqref{eq:MF-mot-meas} via Euler characteristics with compact
support on one hand and Euler characteristics of periodic cyclic
homology on the other hand.

\subsection*{Acknowledgments}
\label{sec:acknowledgements}

We thank Maxim Kontsevich for sharing his insights with us and for
many useful discussions. We thank Alexander Kuznetsov for
explanations 
concerning Theorem~\ref{t:basic-semi-orthog-2}.
Furthermore we have benefited from discussions and correspondence
with Dmitri Orlov, Vladimir Drinfeld, Daniel Huybrechts, Tobias
Dyckerhoff, Bertrand To\"en and J\'anos Koll\'ar. 

The second author was
supported by a postdoctoral fellowship of 
the German Academic Exchange Service (DAAD)  
when finishing this article.
Before that he was partially supported by
the Collaborative Research Center SFB Transregio 45
and 
the priority program SPP 1388 
of the
German Science foundation (DFG). He thanks these institutions.

\section{Categories of curved dg sheaves}
\label{sec:categories-curved-dg}

As described in the introduction we discuss foundational results
on categories of matrix factorizations.
Our main references for this section were
\cite{positselski-coh-analogues-matrix-fact-sing-cats,
  positselski-two-kinds,
  orlov-mf-nonaffine-lg-models}.
Some of the ideas are also contained in
\cite{lin-pomerleano}. 

Let $k$ be a fixed field. All schemes considered are schemes over
$k.$ We say that a scheme $X$
satisfies condition~\ref{enum:srNfKd} 
if
\begin{enumerate}[label=(srNfKd)]
\item
  \label{enum:srNfKd}
  $X$ is a separated regular Noetherian scheme of finite
  Krull dimension. 
\end{enumerate}
For example, any regular quasi-projective scheme satisfies
condition~\ref{enum:srNfKd}.
Note that any coherent $\mathcal{O}_X$-module on an
\ref{enum:srNfKd}-scheme $X$ is a quotient of a locally free
$\mathcal{O}_X$-module of finite type (by theorems of Kleiman
\cite[Ex.~III.6.8]{Hart} and Auslander and Buchsbaum
\cite[Thm.~20.3]{matsumura-comm-ring}); in particular, such a
scheme satisfies condition (ELF) in
\cite{orlov-mf-nonaffine-lg-models}.

Fix a scheme $X$ satisfying condition~\ref{enum:srNfKd}. 
Let $W
\in \Gamma(X,\mathcal{O} _X)$ be a global regular function which
we consider as a morphism $W \colon  X \ra \DA^1:=\DA^1_k=\Spec k[T].$
We do not assume that the
morphism $W$ is flat, for example $W$ may be the zero function.

In this section graded means $\DZ_2$-graded (where $\DZ_2=\DZ/2\DZ$)
if not explicitly stated otherwise, and
differential graded is often abbreviated by dg.
We use lower indices when referring to the graded components of a
$\DZ_2$-graded object.

The usual notions and results for differential $\DZ$-graded
categories (quasi-equivalence, pretriangulated dg category,
(Drinfeld) dg quotient, etc.) have obvious counterparts
in the world of differential $\DZ_2$-graded categories.

\subsection{Definition of various categories}
\label{sec:defin-vari-categ}

By a sheaf on $X$ we mean an $\mathcal{O}_X$-module.
We denote by $\Sh(X)$ the category of all sheaves on $X,$
and by $\Qcoh(X)$ and $\Coh(X)$ the full subcategories of
quasi-coherent and coherent sheaves, respectively.
By $\InjSh(X)$ (resp.\ $\InjQcoh(X)$) we denote the full
subcategory of injective objects in $\Sh(X)$ (resp.\
$\Qcoh(X)$). We write $\Locfree(X)$ (resp.\ $\FlatQcoh(X)$) for
the full subcategory of $\Qcoh(X)$ consisting of locally free
sheaves (of possibly infinite rank) (resp.\ of quasi-coherent
sheaves that are flat over $\mathcal{O}_X$).

We recall some results from 
\cite[II.\S 7]{hartshorne-residues-duality}
and deduce some well-known consequences.

\begin{theorem}
  [{\cite[II.\S 7]{hartshorne-residues-duality}}]
  \label{t:injective-in-qcoh-vs-all-OX}
  (Here $X$ can be any locally Noetherian scheme.)
  \begin{enumerate}
  \item
    \label{enum:qcoh-embeds-in-inj-OX-that-qcoh}
    Every object of $\Qcoh(X)$ can be embedded in an object of
    $\InjSh(X) \cap \Qcoh(X).$ 
  \item
    \label{enum:inj-qcoh-equal-inj-OX-that-qcoh}
    The injective objects in $\Qcoh(X)$ are precisely the
    injective objects of $\Sh(X)$ that are
    quasi-coherent, $\InjQcoh(X) = \InjSh(X) \cap
    \Qcoh(X).$
  \item 
    \label{enum:inj-qcoh-restrict-to-inj-qcoh}
    If $I \in \Qcoh(X)$ is an injective object and $U \subset
    X$ is open, then $I|_U \in \Qcoh(U)$ is again injective.
  \item 
    \label{enum:inj-and-inj-qcoh-arbitrary-sums}
    Any direct sum of objects of $\InjSh(X)$ (resp.\
    $\InjQcoh(X)$) is in $\InjSh(X)$ (resp.\ $\InjQcoh(X)$).
  \end{enumerate}
\end{theorem}

\begin{proof}
  \ref{enum:qcoh-embeds-in-inj-OX-that-qcoh}:
  This is \cite[Thm.~II.7.18]{hartshorne-residues-duality}.

  \ref{enum:inj-qcoh-equal-inj-OX-that-qcoh}:
  The inclusion $\supset$ is obvious. For the inclusion
  $\subset$ let $F \in \Qcoh(X).$ Then $F \subset J$ for $J \in
  \InjSh(X) \cap \Qcoh(X)$
  by \ref{enum:qcoh-embeds-in-inj-OX-that-qcoh}. 
  If $F$ is injective
  in $\Qcoh(X),$ this inclusion splits, and hence $F$ is an
  injective object of $\Sh(X).$

  \ref{enum:inj-qcoh-restrict-to-inj-qcoh}:
  By 
  \ref{enum:inj-qcoh-equal-inj-OX-that-qcoh}, $I$ is an injective
  $\mathcal{O}_X$-module.
  Let $j \colon  U \ra X$ be the inclusion. 
  We have the adjunction $(j_!,
  j^!=j^*)$ (of functors between $\Sh(X)$ and $\Sh(U)$).
  Since $j_!$ is exact this shows that
  $j^*(I)$ is an injective $\mathcal{O}_U$-module. It is
  quasi-coherent, so we can use
  \ref{enum:inj-qcoh-equal-inj-OX-that-qcoh} again.

  \ref{enum:inj-and-inj-qcoh-arbitrary-sums}:
  The statement for $\InjSh(X)$ is precisely
  \cite[Cor.~7.9]{hartshorne-residues-duality},
  and the statement for $\InjQcoh(X)$ then follows from
  \ref{enum:inj-qcoh-equal-inj-OX-that-qcoh}
  since the inclusion $\Qcoh(X) \subset \Sh(X)$ preserves direct
  sums 
  (for Noetherian $X$ one can also use 
  \cite[Prop.~7.2]{hartshorne-residues-duality}
  and the example before that proposition).
\end{proof}

\begin{definition}
  \label{d:Wcdg-sheaves}
  The dg (differential $\DZ_2$-graded) category $\Sh(X,W)$ is
  defined as follows. Its objects are
  \define{$W$-curved dg sheaves on $X,$} i.\,e.\ diagrams 
  \begin{equation*}
    E =  (\matfak{E_1}{e_1}{E_0}{e_0})
  \end{equation*}
  in $\Sh(X)$ satisfying
  $e_{i+1} e_i =W 
  \id_{E_i},$ for $i \in \DZ_2.$
  The morphism space between two $W$-curved dg sheaves $E,$ $E'$
  is the graded module
  \begin{equation*}
    \Hom_{\Sh(X,W)} (E,E'):=\bigoplus_{l \in \DZ_2}
    \Big(\bigoplus_{i \in \DZ_2}\Hom_{\mathcal{O}_X}(E_i,E'_{i+l})\Big)
  \end{equation*}
  with differential $d(g)= e' \comp g -
  (-1)^{|g|}g \comp e$ where $g$ is homogeneous of degree $|g|.$ 

  Denote by $\Qcoh(X,W),$ $\Coh(X,W),$ 
  $\MF(X,W),$ 
  $\InjQcoh(X,W),$ 
  $\Locfree(X,W),$
  and 
  $\FlatQcoh(X,W)$
  the full dg subcategories of $\Sh(X,W)$ consisting of objects
  whose components are quasi-coherent sheaves,
  coherent sheaves,
  locally free sheaves of finite type (= vector
  bundles),
  injective quasi-coherent sheaves, 
  locally free sheaves,
  and flat quasi-coherent sheaves,
  respectively.
  Objects of $\MF(X,W)$ are called 
  \define{matrix factorizations of $W.$}
\end{definition}

The shift $[1]E$ of a $W$-curved dg sheaf $E$ as above is defined as
\begin{equation*}
  [1]E = (\matfak{E_0}{-e_0}{E_1}{-e_1}).
\end{equation*}

Given a dg category $\mathcal{C},$ the category
$Z_0(\mathcal{C})$ and the homotopy category $[\mathcal{C}]$
of $\mathcal{C}$ are defined as usual: they have
the same objects as $\mathcal{C},$ but
$\Hom_{Z_0(\mathcal{C})}(E,E')=Z_0(\Hom_\mathcal{C}(E,E'))$ and
$\Hom_{[\mathcal{C}]}(E,E')=H_0(\Hom_\mathcal{C}(E,E')).$

\begin{remark}
  The categories $Z_0(\Sh(X,W)),$ $Z_0(\Qcoh(X,W))$ and
  $Z_0(\Coh(X,W))$ are abelian categories.  A sequence in
  $Z_0(\MF(X,W)),$ 
  $Z_0(\InjQcoh(X,W)),$ 
  $Z_0(\Locfree(X,W))$
  or
  $Z_0(\FlatQcoh(X,W))$
  will be called exact if
  it is exact in the ambient abelian category $Z_0(\Qcoh(X,W)).$
\end{remark}

Let $F= (\ldots \ra F^i \xra{d^i_F} F^{i+1} \ra \ldots)$
be a complex in $Z_0(\Sh(X,W)).$
We define its totalization
$\Tot(F) =: T = (\matfak{T_1}{t_1}{T_0}{t_0}) \in \Sh(X,W)$ by
\begin{equation*}
  T_l :=\bigoplus_{i \in \DZ,\; j \in \DZ_2, \atop i+j\equiv l \mod 2} F^i_j
\end{equation*}
for $l \in \DZ_2$
and $t_l|_{F^i_j} = (d^i_F)_j + (-1)^i f^i_j,$
for $l, j \in \DZ_2$ and $i \in \DZ$ satisfying $i+j \equiv l \mod
2.$

If $g \colon  E \ra E'$ is a morphism in $Z_0(\Sh(X,W))$ we define its
cone 
$\Cone(g)$ to be the totalization of the complex $(\ldots \ra 0 \ra E
\xra{g} E' \ra 0 \ra \ldots)$ with $E'$ in degree zero.
This shows that
$\Sh(X,W)$ is a
pretriangulated
dg category, and similarly for
$\Qcoh(X,W),$ $\Coh(X,W),$
$\MF(X,W),$
$\InjQcoh(X,W),$ 
$\Locfree(X,W)$ and $\FlatQcoh(X,W).$
In particular, the homotopy categories
$[\Sh(X,W)],$
$[\Qcoh(X,W)],$ $[\Coh(X,W)],$ 
$[\MF(X,W)],$
$[\InjQcoh(X,W)],$ 
$[\Locfree(X,W)]$ and $[\FlatQcoh(X,W)]$
are triangulated\footnote{
  Our (standard) triangles and the (standard) triangles
  in \cite{orlov-mf-nonaffine-lg-models} differ in the sign
  of the last morphism. However the associated homotopy
  categories are equivalent as triangulated categories.
  For this one may use
  \cite[10.1.10.i]{KS} or the equivalence that multiplies the
  differentials $e_0,$ $e_1$ of all objects $E$ by $-1$ and is
  the identity on morphisms.
} 
categories.

\begin{remark}
  \label{rem:cohomology-ModXW-not-defined}
  Notice that one cannot define the cohomology of
  an object $E \in \Sh(X,W)$ (unless $W=0$), but we can define the
  cohomology of a complex $F$ as above. In particular, it
  makes sense to ask whether $F$ is exact.
\end{remark}

\begin{definition}
  \label{d:absolute-derived-categories}
  Denote by $\Acycl[\Sh(X,W)]$ the
  full triangulated subcategory 
  of $[\Sh(X,W)]$
  classically generated by
  the totalizations of all short exact sequences
  \begin{equation*}
    0 \ra F^1 \ra F^2\ra F^3\ra 0
  \end{equation*}
  with $F^i \in \Sh(X,W).$ (Instead of short exact sequences one
  can take all bounded exact complexes, see
  Lemma~\ref{l:totalization}.\ref{enum:totalization-bounded-exact-complex} below.)
  By definition, $\Acycl[\Sh(X,W)]$ is a thick subcategory of
  $[\Sh(X,W)],$ i.\,e.\
  a strict full triangulated subcategory closed under
  direct summands.

  Following
  \cite{positselski-two-kinds,positselski-coh-analogues-matrix-fact-sing-cats} 
  we define
  the \define{absolute derived category
    $\DSh(X,W)$ of $W$-curved dg sheaves}
  as the Verdier quotient
  \begin{equation*}
    \DSh(X,W):=[\Sh(X,W)]/\Acycl[\Sh(X,W)].
  \end{equation*}
  Similarly, we consider the full subcategories
  $\Acycl[\Qcoh(X,W)] \subset [\Qcoh(X,W)],$ 
  $\Acycl[\Coh(X,W)] \subset [\Coh(X,W)],$ 
  $\Acycl[\MF(X,W)] \subset [\MF(X,W)],$ 
  $\Acycl[\Locfree(X,W)] \subset [\Locfree(X,W)],$ 
  $\Acycl[\FlatQcoh(X,W)] \subset [\FlatQcoh(X,W)],$ 
  and the corresponding Verdier quotients
  \begin{align*}
    \DQcoh(X,W) & =[\Qcoh(X,W)]/\Acycl[\Qcoh(X,W)],\\
    \DCoh(X,W) & =[\Coh(X,W)]/\Acycl[\Coh(X,W)],\\
    \bfMF(X,W) & =[\MF(X,W)]/\Acycl[\MF(X,W)],\\
    \DLocfree(X,W) & =[\Locfree(X,W)]/\Acycl[\Locfree(X,W)],\\
    \DFlatQcoh(X,W) & =[\FlatQcoh(X,W)]/\Acycl[\FlatQcoh(X,W)].
  \end{align*}
  The triangulated category
  $\bfMF(X,W)$ is called
  the \define{category of matrix factorizations of $W.$}
\end{definition}

There is another characterization of $\Acycl[\MF(X,W)]$
given in Corollary~\ref{c:acycl-vs-locally-contractible} below.
We will be mainly interested in the category 
$\bfMF(X,W).$

\begin{remark}
  \label{rem:X-disconnected}
  Let $X_1, \dots ,X_m$ be the connected components of $X.$ Then
  \begin{equation*}
    \DSh(X,W) = \prod_{i=1}^m \DSh(X_i,W),
  \end{equation*}
  and similarly for all other categories defined above. So to study
  these categories one may assume that $X$ is connected (if needed),
  and then the map $W$ is either flat or else 
  constant (here constant means that $W(X)$
  consists of a single point in $\DA^1$ which is then
  necessarily closed; if we think of $W$ as an element of
  $\Gamma(X, \mathcal{O}_X)$ it means that $W \in k \subset \Gamma(X, \mathcal{O}_X)$). 
\end{remark}

Here is a useful lemma.

\begin{lemma}
  \label{l:totalization}
  Let 
  $\mathcal{M}$ be $\Sh(X,W),$ $\Qcoh(X,W),$ $\Coh(X,W),$ 
  $\MF(X,W),$ 
  $\Locfree(X,W),$
  or
  $\FlatQcoh(X,W).$
  \begin{enumerate}
  \item
    \label{enum:ses-gives-triangle}
    Any short exact sequence
    $0 \ra F^{-1} \xra{p} F^{0} \xra{q} F^1 \ra 0$
    in $Z_0(\mathcal{M})$ gives rise to a triangle
    $F^{-1} \xra{p} F^{0} \xra{q} F^1 \ra [1]F^{-1}$
    in
    $\op{D}\mathcal{M}$
    (where $\op{D}\MF(X,W):= \bfMF(X,W)$).
  \item 
    \label{enum:totalization-bounded-exact-complex}
    Let 
    $F= (\ldots \ra  0 \ra F^1 \xra{f^1} F^2 \xra{f^2} F^3 \xra{f^3}
    F^4 \ra \dots \ra F^n \ra 0 \ra \ldots)$
    be a bounded exact complex in $Z_0(\mathcal{M}).$ 
    Then $\Tot(F) \in \Acycl[\mathcal{M}].$
  \item
    \label{enum:brutal-truncation-and-totalization}
    If
    $F= (\ldots \ra 0 \ra P^a \ra \ldots \ra P^b \xra{v}
    I^{b+1} \ra \ldots \ra I^c \ra 0 \ra \ldots)$
    is a bounded
    complex
    in $Z_0(\mathcal{M})$ that is composed of two bounded complexes
    $P$ and $I$ as indicated,
    there is a standard triangle
    \begin{equation*}
      [1]\Tot(P) \xra{v} \Tot(I) \ra \Tot(F) \ra \Tot(P)
    \end{equation*}
    in $[\mathcal{M}].$
    If $F$ is exact, $[1]\Tot(P) \xra{v} \Tot(I)$ is an
    isomorphism in 
    $\op{D}\mathcal{M}.$
  \item
    \label{enum:totalization-of-finite-complex-components-null-homotopic}
    Let $F$ be a bounded complex in $Z_0(\mathcal{M}).$ 
    If
    each $F^i$ is isomorphic to $0$ in $[\mathcal{M}],$ then
    $\Tot(F)= 0$ in $[\mathcal{M}].$ Similarly, if each $F^i$ is
    in $\Acycl[\mathcal{M}],$ then $\Tot(F) \in
    \Acycl[\mathcal{M}].$
  \end{enumerate}
\end{lemma}

\begin{proof}
  \ref{enum:ses-gives-triangle}:
  We have standard triangles
  \begin{equation*}
    F^{-1} \xra{p} F^0 \xra{\svek 10} \Cone(p) \xra{\zvek 01}
    [1]F^{-1}, 
  \end{equation*}
  where $\Cone(p)=F^0\oplus [1]F^{-1}$ as a graded sheaf,
  and
  \begin{equation*}
    \Cone(p) \xra{\zvek q0} F^1 \xra{\svek 10} \Cone(\zvek q0)
    \xra{\zvek 01} [1]\Cone(p)
  \end{equation*}
  in $[\mathcal{M}].$
  Note that
  \begin{equation*}
    \Cone(\zvek q0) = F^1 \oplus [1] \Cone(p)=
    F^1 \oplus [1] F^0 \oplus F^{-1}
  \end{equation*}
  has differential
  $  \left[
    \begin{smallmatrix}
      f^1 & q & 0\\
      0 & -f^0 & -p\\
      0 & 0 & f^{-1}
    \end{smallmatrix}
  \right]$
  and hence is the totalization of the exact complex $0 \ra F^{-1}
  \xra{-p} F^0 \xra{q} F^1 \ra 0$ with $F^0$ in odd degree.
  This implies that $\Cone(p) \xra{\zvek q0} F^1$ becomes an
  isomorphism in 
  $\op{D}\mathcal{M}.$

  \ref{enum:totalization-bounded-exact-complex}:
  Factor $F^2 \ra F^3$ in $Z_0(\Sh(X,W))$ into an epimorphism followed by a
  monomorphism, $F^2 \xra{p} Q \xra{i} F^3,$
  and note that $Q \in \mathcal{M}$ (for example, if 
  $\mathcal{M}=\MF(X,W),$ the kernel of $F^{n-1} \sra F^n$ is in
  $\MF(X,W),$ and we can iterate this argument).
  Consider the vertical morphism of horizontal
  complexes
  \begin{equation*}
    \xymatrix{
      {R \colon } \ar[d]^u &
      {0} \ar[r] \ar[d] &
      {0} \ar[r] \ar[d] &
      {0} \ar[r] \ar[d] &
      {Q} \ar[r]^{-i} \ar[d]^1 &
      {F^3} \ar[r]^{-f^3} \ar[d] &
      {F^4} \ar[r] \ar[d] &
      {\dots} \ar[r] &
      {0} \\
      {S \colon } &
      {0} \ar[r] &
      {F^1} \ar[r]^{f^1} &
      {F^2} \ar[r]^p &
      {Q} \ar[r] &
      {0} \ar[r] &
      {0} \ar[r] &
      {\dots} \ar[r] &
      {0}
    }
  \end{equation*}
  We leave it to the reader to check that
  the mapping cone $\Cone(u)$ of this morphism is isomorphic to
  $F$ in the
  homotopy category of complexes in $Z_0(\mathcal{M}).$ 
  Hence $\Tot(\Cone(u)) \cong \Tot(F)$ in $[\mathcal{M}].$
  On the other hand we have a short exact sequence $0 \ra \Tot(S)
  \ra \Tot(\Cone(u)) \ra [1]\Tot(R) \ra 0$ in $Z_0(\mathcal{M})$
  and hence a triangle $\Tot(S) \ra \Tot(\Cone(u)) \ra [1]\Tot(R)
  \ra [1]\Tot(S)$ in $\op{D} \mathcal{M}$ 
  by
  \ref{enum:ses-gives-triangle}. 
  By induction $\Tot(S)$ and $\Tot(R)$ are in
  $\Acycl[\mathcal{M}],$ and hence $\Tot(F) \cong \Tot(\Cone(u))
  \in \Acycl[\mathcal{M}].$

  \ref{enum:brutal-truncation-and-totalization}: Obvious.
  If $F$ is exact, use
  \ref{enum:totalization-bounded-exact-complex}. 

  \ref{enum:totalization-of-finite-complex-components-null-homotopic}:
  We argue by induction on the number of $i$ with $F^i\not=0$ in
  $\mathcal{M}.$ If
  this number is $\leq 1$ the claim is obvious.
  Otherwise let $i \in \DZ$ be non-maximal with $F^i \not=0.$
  Let $w_{\leq i}F$ be the
  complex obtained from $F$ by replacing all terms in degrees $>i$
  by $0,$ and define $w_{>i}F$ similarly.
  As in \ref{enum:brutal-truncation-and-totalization}, there is a
  standard triangle
  \begin{equation*}
    [1]\Tot(w_{\leq i}F) \ra \Tot(w_{>i}F) \ra \Tot(F) \ra \Tot(w_{\leq i}F)
  \end{equation*}
  in $[\mathcal{M}].$ By induction the first two terms are
  isomorphic to zero in 
  $[\mathcal{M}]$ (resp.\ are in $\Acycl[\mathcal{M}]$),
  hence so is $\Tot(F).$
\end{proof}

\subsection{Matrix factorizations and the category of singularities}
\label{sec:matr-fact-categ}

In case the morphism $W \colon X\ra \DA^1$ is flat we recall
an important theorem proved in \cite{orlov-mf-nonaffine-lg-models}. Recall that the category of singularities
$D_\Sg(Y)$ of a Noetherian scheme $Y$
is defined as the Verdier quotient
\begin{equation*}
  D_\Sg(Y):=D^b(\Coh(Y))/\mfPerf(Y),
\end{equation*}
where 
$D^b(\Coh(Y))$ is the bounded derived category of coherent
sheaves on $Y$ and
$\mfPerf(Y)$ is the category of perfect complexes.

Let $X_0$ be the scheme-theoretical zero fiber of the morphism
$W \colon X\ra \DA^1.$ Given $E = (\matfak{E_1}{e_1}{E_0}{e_0})
\in \bfMF(X,W)$ the
cokernel of the map $e_1$ is annihilated by $W,$ hence it comes from
an object in $\Coh(X_0).$ We denote this object by $\coker e_1.$

\begin{theorem}
  [\cite{orlov-mf-nonaffine-lg-models}]
  \label{t:factorizations=singularity}
  Assume that the morphism
  $W \colon X\ra \DA^1$ is flat. Then the
  above construction induces a functor
  \begin{equation*}
    \coker  \colon \bfMF(X,W)\ra D_\Sg(X_0)
  \end{equation*}
  which is an equivalence of triangulated categories.
\end{theorem}

The above theorem is useful because it gives us two completely
different descriptions of the same triangulated category.

\subsection{Some embeddings and equivalences}
\label{sec:some-embedd-equiv}

Our next aim is to prove the equivalences and embeddings stated in
the following theorem.

\begin{theorem}
  \label{t:equivalences-curved-categories}
  \rule{1mm}{0mm}
  \begin{enumerate}
  \item
    \label{enum:injQcoh-DQcoh-equiv}
    The functor $[\InjQcoh(X,W)]\ra \DQcoh(X,W)$ is an equivalence.
  \item
    \label{enum:Dcoh-DQcoh-ff}
    The functor $\DCoh(X,W)\ra \DQcoh(X,W)$ is full and faithful.
  \item
    \label{enum:MF-Dcoh-equiv}
    The functor $\bfMF(X,W)\ra \DCoh(X,W)$ is an equivalence.
  \item
    \label{enum:DLocfree-DQcoh-equiv}
    The functor $\DLocfree(X,W)\ra \DQcoh(X,W)$ is an equivalence.
  \item
    \label{enum:DFlatQcoh-DQcoh-equiv}
    The functor $\DFlatQcoh(X,W)\ra \DQcoh(X,W)$ is an equivalence.
  \end{enumerate}
\end{theorem}

\begin{proof}
  Consider the commutative diagrams
  of inclusions of
  full triangulated categories
  \begin{equation*}
    \xymatrix{
      {\Acycl[\MF(X,W)]} \ar@{}[r]|{\subset} \ar@{}[d]|{\cap}
      \ar@{}[rd]|{(\boxast 1)} &
      {\Acycl[\Coh(X,W)]} \ar@{}[r]|{\subset} \ar@{}[d]|{\cap}
      \ar@{}[rd]|{(\boxast 2)} &
      {\Acycl[\Qcoh(X,W)]} \ar@{}[r]|{\supset} \ar@{}[d]|{\cap}
      \ar@{}[rd]|{(\boxast 3)} &
      {\{0\}} \ar@{}[d]|{\cap} \\
      {[\MF(X,W)]} \ar@{}[r]|{\subset} &
      {[\Coh(X,W)]} \ar@{}[r]|{\subset} &
      {[\Qcoh(X,W)]} \ar@{}[r]|{\supset} &
      {[\InjQcoh(X,W)]}
    }
  \end{equation*}
  and
  \begin{equation*}
    \xymatrix{
      {\Acycl[\Locfree(X,W)]} \ar@{}[r]|{\subset} \ar@{}[d]|{\cap}
      \ar@{}[rd]|{(\boxast 4)} &
      {\Acycl[\Qcoh(X,W)]} \ar@{}[r]|{\supset} \ar@{}[d]|{\cap}
      \ar@{}[rd]|{(\boxast 5)} &
      {\Acycl[\FlatQcoh(X,W)]}
      \ar@{}[d]|{\cap}\\ 
      {[\Locfree(X,W)]} \ar@{}[r]|{\subset} &
      {[\Qcoh(X,W)]} \ar@{}[r]|{\supset} &
      {[\FlatQcoh(X,W)].}
    }
  \end{equation*}
  We will show that
  the three equivalent conditions
  \ref{enum:verdier-loc-precompose-to-Mor-V},
  \ref{enum:verdier-loc-WC-factors-WVC},
  \ref{enum:verdier-loc-Hom-WD-in-DmodV-equal-those-in-DmodC}
  of Proposition~\ref{p:verdier-localization-induced-functor-ff}
  hold for the squares
  $(\boxast 1)$ and
  $(\boxast 2)$
  (and then also for the rectangle formed out of these two
  squares),
  and for the squares
  $(\boxast 4)$
  and
  $(\boxast 5),$
  and that the three equivalent conditions
  \ref{enum:verdier-loc-postcompose-to-Mor-V},
  \ref{enum:verdier-loc-CW-factors-CVW},
  \ref{enum:verdier-loc-Hom-DW-in-DmodV-equal-those-in-DmodC}
  hold for the square
  $(\boxast 3).$
  This will imply that all five functors in
  Theorem~\ref{t:equivalences-curved-categories}
  are full and faithful.

  The following lemma is essentially contained
  in \cite[Thm.~3.6]{positselski-two-kinds}. It shows that the
  functors
  considered in parts \ref{enum:injQcoh-DQcoh-equiv},
  \ref{enum:MF-Dcoh-equiv}, 
  \ref{enum:DLocfree-DQcoh-equiv} and
  \ref{enum:DFlatQcoh-DQcoh-equiv}
  of
  Theorem~\ref{t:equivalences-curved-categories}
  are essentially surjective.

  \begin{lemma}
    \label{l:resolutions}
    \rule{1mm}{0mm}
    \begin{enumerate}
    \item
      \label{enum:Inj-reso}
      For any $F\in \Qcoh(X,W)$ there exists an exact sequence
      $0\ra F\ra I^0\ra \dots \ra I^n \ra 0$ in $Z_0(\Qcoh(X,W))$
      with all $I^j \in \InjQcoh(X,W).$ 
      In particular, 
      the obvious morphism 
      $F \ra \Tot(I)$ has its cone in $\Acycl[\Qcoh(X,W)]$
      and hence becomes an isomorphism
      in $\DQcoh(X,W).$
    \item
      \label{enum:MF-reso}
      Let $E\in \Coh(X,W).$ Then there exists an exact sequence
      $0\ra 
      P^n\ra \dots \ra P^0\ra E\ra 0$ in $Z_0(\Coh(X,W))$ with all
      $P^i\in \MF(X,W).$
      In particular, the 
      obvious morphism 
      $\Tot(P) \ra E$ has its cone in $\Acycl[\Coh(X,W)]$
      and hence becomes an isomorphism
      in $\DCoh(X,W).$
    \item
      \label{enum:locfree-reso}
      Let $E \in \Qcoh(X,W).$ Then there exists an exact sequence
      $0\ra 
      P^n\ra \dots \ra P^0\ra E\ra 0$ in $Z_0(\Qcoh(X,W))$ with all
      $P^i\in \Locfree(X,W) \subset \FlatQcoh(X,W).$
      In particular, the obvious morphism 
      $\Tot(P) \ra E$ has its cone in $\Acycl[\Qcoh(X,W)]$
      and hence becomes an isomorphism
      in $\DQcoh(X,W).$
    \end{enumerate}
  \end{lemma}

  \begin{proof}
    \ref{enum:Inj-reso}
    Choose injective
    morphisms $g_0 \colon  F_0 \ra J_0$ and $g_1 \colon  F_1 \ra J_1,$ such
    that $J_0$ 
    and $J_1$ are injective quasi-coherent sheaves
    (use Theorem~\ref{t:injective-in-qcoh-vs-all-OX}).
    Consider the object $I' \in \InjQcoh(X,W),$
    where $I'_0=I'_1=J_0\oplus J_1$ and $i'_0=W\oplus \id,$ $i'_1=\id
    \oplus W.$
    We denote $I'$ by $G^-(J)$ for future reference.
    Note that $G^-(J)$ only depends on the graded sheaf $J.$

    Let $h=(h_0, h_1) \colon  F \ra I'$ be the injective morphism in $Z_0(\Qcoh(X,W)$
    given by $h_0=(g_0, g_1 f_0)^\tp,$ $h_1=(g_0f_1,g_1)^\tp.$
    Now define $I^0:=I',$ replace $F$ by $\coker h$ and repeat the
    procedure. 
    Since $X$ is regular of finite Krull dimension
    we eventually arrive at the
    desired 
    finite resolution
    (the global dimension of all local rings $\mathcal{O}_{X,x}$
    is bounded by the Krull dimension of $X,$ and injectivity of
    a quasi-coherent sheaf can
    be tested on the stalks by
    \cite[Prop.~7.17]{hartshorne-residues-duality} 
    and Theorem~\ref{t:injective-in-qcoh-vs-all-OX}).
    The isomorphism
    $F \sira \Tot(I)$ in $\DQcoh(X,W)$ follows from
    Lemma~\ref{l:totalization}.\ref{enum:brutal-truncation-and-totalization}.

    \ref{enum:MF-reso}
    We apply the dual process. Namely, our assumptions on $X$
    allow us to choose vector
    bundles $N_0$ and $N_1$ with surjective morphisms
    $g_0 \colon  N_0 \sra E_0$ and $g_1 \colon  N_1 \sra E_1.$
    Consider
    the object $P'\in \MF(X,W)$
    where $P'_0=P'_1=N_0 \oplus N_1$ and
    $p'_0=\id \oplus W,$ $p'_1= W \oplus \id.$
    We denote $P'$ by $G^+(N)$ for
    future reference. It only depends on the
    graded sheaf $N.$

    Let $h \colon  P' \ra E$ be the surjective morphism in $Z_0(\Coh(X,W))$
    given by
    $h_0=(g_0, e_1g_1),$ $h_1=(e_0g_0, g_1).$
    Now replace $E$ by $\ker h$ and repeat the
    procedure. Since $X$ is regular of finite Krull dimension we
    eventually arrive at the desired 
    finite resolution. The last statement follows from
    Lemma~\ref{l:totalization}.\ref{enum:brutal-truncation-and-totalization}.

    \ref{enum:locfree-reso}:
    Since any quasi-coherent sheaf is the union of its coherent
    subsheaves
    (\cite[Exercise II.5.15(e)]{Hart})
    there are locally free sheaves $N_0$ and $N_1$ with
    epimorphisms $g_i \colon N_i \sra E_i.$
    We then proceed as in the proof of
    \ref{enum:MF-reso}, using
    \cite[Corollary~4.5]{bass-big-projective-free}. 
  \end{proof}

  \begin{remark}
    \label{rem:middle-objects-vanish}
    If $K$ is any quasi-coherent sheaf on $X,$ then
    $(\matfak{K}{\id}{K}{W}) \in \Qcoh(X,W)$
    is obviously zero in
    $[\Qcoh(X,W)]$ and in $\DQcoh(X,W).$ 
    So if $0 \ra F \ra I^0
    \ra \dots \ra I^n \ra 0$ is the exact sequence constructed in
    the proof of Lemma~\ref{l:resolutions}.\ref{enum:Inj-reso},
    all objects $I^0, \dots, I^{n-1}$ are zero in $\DQcoh(X,W).$
    This implies that $F$ and $[n]I^n$ are isomorphic in
    $\DQcoh(X,W)$ (use
    Lemma~\ref{l:totalization}.\ref{enum:ses-gives-triangle}).
    Similar conclusions hold for the exact sequences constructed
    in the proof of parts \ref{enum:MF-reso} and
    \ref{enum:locfree-reso} of Lemma~\ref{l:resolutions}.
  \end{remark}
  
  \begin{remark}
    \label{rem:morphism-to-reso-by-injectives}
    Let $p \colon  E \ra F$ be a morphism in $Z_0(\Qcoh(X,W)),$ and let
    $0 \ra E \ra A^0 \ra A^1 \ra \dots$ be an exact sequence
    in $Z_0(\Qcoh(X, W)).$
    Then there is a resolution $F \ra I$ as in
    Lemma~\ref{l:resolutions}.\ref{enum:Inj-reso}
    and a morphism $A \ra I$ of resolutions that lifts $p.$

    Namely, in the notation of the proof of 
    Lemma~\ref{l:resolutions}.\ref{enum:Inj-reso},
    find morphisms $q_l \colon  A^0_l \ra J_l$ that restrict to
    $g_l p_l$ on $E_l,$ for $l=0, 1.$ Then 
    $(q_0, q_1 a_0)^\tp \colon  A^0_0 \ra I^0_0=J_0 \oplus J_1$
    and 
    $(q_0 a_1, q_1)^\tp \colon  A^0_1 \ra I^0_1=J_0 \oplus J_1$
    define a morphism $A^0 \ra I^0$ that lifts $p \colon E \ra F.$
    Pass to the cokernels and proceed.
  \end{remark}

  \begin{lemma}
    \label{l:sufficient-for-inj-DQcoh-equiv}
    We have
    \begin{equation*}
      \Hom_{[\Qcoh(X,W)]}(\Acycl[\Qcoh(X,W)], [\InjQcoh(X,W)])=0.
    \end{equation*}
    In particular,
    condition~\ref{enum:verdier-loc-CW-factors-CVW}
    of Proposition~\ref{p:verdier-localization-induced-functor-ff}
    holds for the square
    $(\boxast 3).$
  \end{lemma}

  \begin{proof}
    Let $J \in \InjQcoh(X,W).$ 
    Let
    $0\ra E^1\ra E^2\ra E^3\ra 0$ be a short exact sequence in
    $Z_0(\Qcoh(X,W))$ with totalization $\Tot(E).$ The
    dg module $\Hom_{\Qcoh(X,W)}(\Tot(E), J)$ is
    the totalization of the short exact sequence
    \begin{equation*}
      0\ra \Hom_{\Qcoh(X,W)} (E^3, J) \ra \Hom_{\Qcoh(X,W)} (E^2, J) \ra
      \Hom_{\Qcoh(X,W)} (E^1, J) \ra 0
    \end{equation*}
    of dg modules.
    Hence it is obviously (or by 
    Lemma~\ref{l:double-complex-upper-halfplane} below) acyclic,
    so
    $\Hom_{[\Qcoh(X,W)]}(\Tot(E), J)=0.$ This implies the lemma.

  \end{proof}

  \begin{remark}
    \label{rem:morphisms-to-injectives}
    For any $F\in [\Qcoh(X,W)]$ and $I \in [\InjQcoh (X,W)]$
    the canonical map
    \begin{equation*}
      \Hom_{[\Qcoh(X,W)]}(F,I) \sira \Hom_{\DQcoh(X,W)}(F,I)
    \end{equation*}
    is an isomorphism, since
    condition \ref{enum:verdier-loc-Hom-DW-in-DmodV-equal-those-in-DmodC}
    holds for the square $(\boxast 3).$
  \end{remark}

  \begin{lemma}
    \label{l:factoring-through-coherent}
    Condition~\ref{enum:verdier-loc-WC-factors-WVC} of
    Proposition~\ref{p:verdier-localization-induced-functor-ff}
    holds for the square
    $(\boxast 2).$
    Namely,
    let $L \in [\Coh(X,W)]$  and $A\in
    \Acycl[\Qcoh(X,W)].$ Then any morphism
    $L \ra  A$ in $[\Qcoh(X,W)]$
    factors through some object $A' \in
    \Acycl[\Coh(X,W)].$
  \end{lemma}

  \begin{proof}
    \textbf{Step 1:}
    Let $E = (\matfak{E_1}{e_1}{E_0}{e_0}) \in \Qcoh(X,W)$
    and let
    $K \subset E$ be a graded coherent subsheaf, i.\,e.\
    $K_i \subset E_i$ is a coherent subsheaf for $i=0,1.$ Then there exists
    $F \in \Coh(X,W)$ such that $F \subset E$ in
    $Z_0(\Qcoh(X,W))$ and $K\subset F$ as graded sheaves.
    Indeed, take $F_1=K_1+e_0K_0,$ $F_0=K_0+e_1K_1.$

    \textbf{Step 2:} Given an exact sequence
    \begin{equation*}
      E=(0 \ra E^1 \xra{d^1}
      E^2 \ra
      \dots
      \ra E^{n-1}
      \xra{d^{n-1}} E^n \ra 0)
    \end{equation*}
    in $Z_0(\Qcoh(X,W))$ and graded coherent
    subsheaves $K^i \subset E^i,$ there exists an exact sequence
    \begin{equation*}
      0\ra F^1\ra  \dots \ra  F^n \ra 0
    \end{equation*}
    in $Z_0(\Coh(X,W))$ which is a subsequence of $E$ such that
    $K^i \subset F^i$ for all $i =1, \dots, n.$

    Indeed, first we may assume that $d^i(K^i)\subset K^{i+1}$
    (by replacing $K^{i+1}$ with $K^{i+1} + d^i(K^i)$).
    Using Step~1, we find a subobject $F^n \subset E^n,$ such that
    $F^n\in \Coh(X,W)$ and $K^n \subset F^n.$
    Between $K^{n-1}$ and $(d^{n-1})\inv(F^n)$ there is a
    graded coherent sheaf surjecting onto $F^n$
    (use \cite[Ex.~II.5.15]{Hart}). Step~1 again then shows that
    there is an object $F^{n-1}\in \Coh(X,W)$ such that
    $K^{n-1} \subset F^{n-1}\subset E^{n-1}$ and
    $d^{n-1}(F^{n-1})=F^n.$

    Now proceed by induction with $F^{n-1} \cap \ker d^{n-1} \subset \ker d^{n-1}$
    instead of $F^n \subset E^n$ and note that $d^{n-2}(K^{n-2})
    \subset K^{n-1} \cap \ker d^{n-1}.$

    \textbf{Step 3:}
    Assume that $A=\Tot(E)$ is the
    totalization of an exact sequence $E$ as above and let
    $g \colon L \ra A$ be a morphism in $[\Qcoh(X,W)].$ Represent $g$ by a
    morphism $\hat{g} \colon L \ra A$ in $Z_0(\Qcoh(X,W)]$ and let $K \subset A$ be
    the image of $\hat{g}.$ Let $K^i_l$ be
    the image of $K_{i+l}$ 
    under 
    the obvious 
    projection $\Tot(E)_{i+l} \ra E^i_l$ of sheaves. This defines
    the graded sheaves $K^i.$ Step 2
    applied in this 
    setting yields an exact sequence $F$
    in $Z_0(\Coh(X,W))$ such that $\hat{g}$ factors through $A'=
    \Tot(F) 
    \subset A.$ Hence $g$ factors through $A' \in
    \Acycl[\Coh(X,W)].$ 

    Now use that
    condition
    \ref{enum:verdier-loc-WC-factors-WVC-via-classical-generation}
    in Proposition~\ref{p:verdier-localization-induced-functor-ff}
    implies condition
    \ref{enum:verdier-loc-WC-factors-WVC}
    (it would have been
    sufficient to consider totalizations of short exact
    sequences only). 
    This finishes the proof.
  \end{proof}

  \begin{lemma}
    \label{l:conditions-hold-for-square1}
    The three equivalent conditions
    \ref{enum:verdier-loc-precompose-to-Mor-V},
    \ref{enum:verdier-loc-WC-factors-WVC},
    \ref{enum:verdier-loc-Hom-WD-in-DmodV-equal-those-in-DmodC}
    hold for the square
    $(\boxast 1).$
  \end{lemma}

  We will give two proofs of this key fact.
  The first proof from \cite{lin-pomerleano} is
  short but uses Theorem~\ref{t:factorizations=singularity}
  and hence only works in case the morphism $W \colon X\ra \DA^1$ is flat.
  The second proof is
  essentially the one given in
  \cite[Prop.~1.5]{positselski-coh-analogues-matrix-fact-sing-cats}
  (in a slightly different language) and works in general.

  \begin{lemma}
    \label{l:affine-MF}
    Assume that $X$ is affine.
    \begin{enumerate}
    \item
      \label{enum:intersection-MF-Acyclcoh-zero}
      Then
      \begin{equation*}
        \Hom_{[\Qcoh(X,W)]}([\MF(X,W)], \Acycl[\Qcoh(X,W)])=0.
      \end{equation*}
      In particular,
      $[\MF(X,W)]\cap \Acycl[\Qcoh(X,W)]=0,$
      $[\MF(X,W)]\cap \Acycl[\Coh(X,W)]=0$
      and $\Acycl[\MF(X,W)]=0.$
    \item
      \label{enum:affine-MF-equals-boldMF}
      $[\MF(X,W)] \sira \bfMF(X,W)$ canonically.
    \end{enumerate}
  \end{lemma}

  \begin{proof}
    Clearly \ref{enum:intersection-MF-Acyclcoh-zero}
    implies \ref{enum:affine-MF-equals-boldMF}.
    To prove \ref{enum:intersection-MF-Acyclcoh-zero} we argue as in
    the proof 
    of Lemma~\ref{l:sufficient-for-inj-DQcoh-equiv}.
    Namely let $P\in \MF(X,W)$ and let $E$
    be the totalization of a short exact sequence
    \begin{equation*}
      0\ra E^1\ra E^2\ra E^3\ra 0
    \end{equation*}
    in $Z_0(\Qcoh(X,W)).$
    Then the dg module $\Hom_{\Qcoh(X,W)}(P,E)$ is the totalization of
    the short exact (since $X$ is affine we can view both
    $P_i$ as projective $\Gamma(X, \mathcal{O}_X)$-modules)
    sequence
    \begin{equation*}
      0\ra \Hom_{\Qcoh(X,W)} (P,E^1)\ra \Hom_{\Qcoh(X,W)} (P,E^2)\ra
      \Hom_{\Qcoh(X,W)} (P,E^3) \ra 0
    \end{equation*}
    of dg modules and hence is acyclic.
    This implies
    all the assertions in
    \ref{enum:intersection-MF-Acyclcoh-zero}.
  \end{proof}

  \begin{proof}[Proof of Lemma~\ref{l:conditions-hold-for-square1}
    in case $W \colon X \ra \DA^1$ is flat]
    We show that condition
    \ref{enum:verdier-loc-precompose-to-Mor-V} holds for the
    square $(\boxast 1)$: 
    Let $s \colon  E \ra P$ in $[\Coh(X,W)]$ with 
    $P \in [\MF(X,W)]$ and cone in $\Acycl[\Coh(X,W)].$ Then
    there exists $t \colon  P' \ra E$ with $P' 
    \in [\MF(X,W)]$ such
    that the cone of $st$ is in $\Acycl[\MF(X,W)].$

    By Lemma~\ref{l:resolutions}.\ref{enum:MF-reso} we can find
    $t$ and $P'$ as required such that the cone of $t$ is in 
    $\Acycl[\Coh(X,W)].$ Hence the cone of $st$ is in
    $\Acycl[\Coh(X,W)],$ and obviously in $[\MF(X,W)].$ 
    We need to show that
    $[\MF(X,W)]\cap \Acycl[\Coh(X,W)]=\Acycl[\MF(X,W)].$ 
    Namely,
    let $F\in [\MF(X,W)]\cap \Acycl[\Coh(X,W)].$ It suffices to
    show that its image $\cokern F$ in $D_\Sg(X_0)$ under the
    equivalence of Theorem~\ref{t:factorizations=singularity} is
    zero. But this is true locally by Lemma~\ref{l:affine-MF},
    and hence globally.  
  \end{proof}

  \begin{proof}[Proof of Lemma~\ref{l:conditions-hold-for-square1}
    for arbitrary $W \colon X \ra \DA^1$]
    It suffices to prove the following claim
    (use condition~\ref{enum:verdier-loc-WC-factors-WVC-via-classical-generation}
    of Proposition~\ref{p:verdier-localization-induced-functor-ff}):
    Let $E\in [\MF(X,W)]$ and let $L$ be the
    totalization of a short exact sequence
    \begin{equation*}
      0 \ra U \xra{i} V \xra{p} Q \ra 0
    \end{equation*}
    in $Z_0(\Coh(X,W))$
    (with $U$ of odd degree).
    Then any morphism $E \ra L$ in $[\Coh(X,W)]$
    factors through an object
    of $\Acycl[\MF(X,W)].$

    \textbf{Step 1:}
    Let $G \in \Coh(X,W).$ Let $\gamma \colon G \ra L$ be a degree zero
    morphism 
    in $\Coh(X,W).$ Then $\gamma=(a,b,c)$
    where $a \colon G \ra U,$ $b \colon G\ra V$ and
    $c \colon G\ra Q$ are morphisms in $\Coh(X,W)$ of degrees 1, 0, 1,
    respectively. Notice that the
    differential of $\gamma$ is given by the formula
    $d\gamma=d(a,b,c)=(-da,ia+db,pb-dc).$

    \begin{lemma} 
      \label{homotopic-zero}
      In this setting assume that the degree zero morphism
      $\gamma=(a,b,c) \colon  G \ra L$ is closed and that $c$ can
      be lifted to a degree one morphism $t \colon G \ra V$ in $\Coh(X,W),$
      i.\,e.\ $pt=c.$ Then $\gamma$
      is homotopic to zero.
    \end{lemma}

    \begin{proof}
      Let $\Hom=\Hom_{\Coh(X,W)}.$
      We have an exact sequence of dg modules
      \begin{equation*}
        0 \ra \Hom (G,U) \xra{i_*} \Hom (G,V) \xra{p_*} \Hom (G,Q).
      \end{equation*}
      Note that $dc=d(pt)=pdt.$ Then
      $p(b-dt)=pb-dc=0,$ so there exists a degree zero morphism $s \in
      \Hom(G,U)$ such that $b-dt=is.$ Then
      $ids=d(is)=db=-ia,$ hence $-ds=a$ and $d(s,t,0)=(a,b,c)=\gamma.$
    \end{proof}

    \textbf{Step 2:}
    Now assume that $N$ is a graded coherent sheaf.
    Recall the object $G^+(N)\in \Coh(X,W)$ freely generated by $N$ (see
    the proof of Lemma~\ref{l:resolutions} above) and note that there
    is a canonical inclusion $N \subset G^+(N)$ of graded sheaves. For
    any $S\in \Coh(X,W)$
    a degree zero morphism $r \colon G^+(N) \ra S$ is uniquely
    determined by the restrictions $r|_N$ and $(dr)|_N$; conversely,
    given two graded morphisms $N \ra S$ of degrees 0 and 1
    respectively, they arise from such a morphism $r.$
    A similar statement holds for degree one morphisms $G^+(N) \ra S.$

    Let $\nu \colon  N \ra L$ be a degree zero morphism of graded
    sheaves. Similarly as
    above we represent it as a triple $\nu=(a',b',c')$ where $a' \colon N \ra
    U,$ $b' \colon N \ra V$ and $c' \colon N \ra Q$ are morphisms of graded sheaves
    of degrees 1, 0, 1, respectively.

    \begin{lemma} 
      \label{lifting} 
      In this setting assume that the
      degree one morphism $c' \colon N \ra Q$ of graded sheaves
      can be lifted to a degree one morphism $s \colon N \ra V,$
      i.\,e.\ $ps=c'.$ Let $\tildew{\nu} \colon G^+(N)\ra L$ be
      the closed degree zero morphism uniquely determined by
      $\tildew{\nu}|_N=\nu \colon N\ra L$ (and
      $d\tildew{\nu}=0$), and let $\tildew{\nu}=(a, b, c)$ be its
      components.  Then the degree one morphism $c \colon G^+(N)
      \ra Q$ can be lifted 
      to a degree one morphism $t \colon G^+(N) \ra V,$ i.\,e.\
      $pt=c.$
    \end{lemma}

    \begin{proof}
      Extend the degree one morphism $s \colon N\ra V$ to a unique degree one
      morphism $t \colon G^+(N)\ra V$ such that $(dt)|_N=b'.$
      Note that $\tildew{\nu}|_N=\nu$ implies $b|_N=b'$ and
      $c|_N=c',$ and that $d\tildew{\nu}=0$ implies $pb=dc.$
      So $pt|_N=ps=c'=c|_N$ and
      $(d(pt))|_N=p(dt)|_N=pb'=pb|_N=(dc)|_N.$ Hence
      $pt=c.$
    \end{proof}

    \textbf{Step 3:}
    To complete the proof assume that we are
    given a morphism $E\ra L$ in $[\Coh(X,W)],$ which we represent by
    a closed degree zero morphism morphism $\epsilon=(a'',b'',c'') \colon  E \ra L$ where
    $a'', b'', c''$ are as explained above.
    Let $N$ be a graded
    vector bundle mapping surjectively onto the "fiber product"
    $V\times_Q E$ of the morphisms $p \colon V\ra Q$
    and $c'' \colon E\ra Q$ 
    (for $l \in \DZ_2$ we have $(V \times_Q E)_l= V_{l+1}
    \times_{Q_{l+1}} E_l$). This yields a surjective degree zero
    morphism $q \colon N \ra E$ of graded 
    sheaves such that $c''q \colon N \ra Q$
    can be lifted to $V.$

    Let $\nu:= \epsilon q \colon N \ra L;$ its third component is
    $c':=c''q.$ 
    Let $\tildew{\nu}=(a,b,c) \colon G^+(N)\ra L$ be the closed degree zero extension
    of $\nu.$ By Lemma~\ref{lifting} the morphism $c$ can be lifted to a degree
    one morphism $t \colon G^+(N) \ra V,$ i.\,e.\ $pt=c.$

    Similarly $q \colon  N \ra E$ extends uniquely to a (surjective)
    closed degree zero morphism $\tildew{q} \colon  G^+(N) \ra E,$ and
    we have $\epsilon \tildew{q}=\tildew{\nu}.$ Let $\rho \colon  R \ra
    G^+(N)$ be the kernel of $\tildew{q}.$ Then $R \in \MF(X,W)$
    (since the kernel of a surjective morphism of vector bundles
    is a vector bundle).  Let $C:=\Cone(\rho).$ As a graded sheaf
    $C=G^+(N)\oplus [1]R,$ so $C \in \MF(X,W).$ The natural
    closed degree zero morphism $(\tildew{q},0) \colon C \ra E$ has cone
    $\Cone((\tildew{q},0))$ in $\Acycl[\MF(X,W)],$ cf.\ the proof
    of Lemma~\ref{l:totalization}.\ref{enum:ses-gives-triangle}.

    The composition $C \xra{(\tildew{q},0)} E
    \xra{\epsilon} L$ is a closed degree zero morphism and given
    by $(\epsilon \tildew{q}, 0)= (\tildew{\nu},0);$ its third
    component is $(c,0) \colon C=G^+(N) \oplus [1]R \ra Q$ and can be
    factored as $C \xra{(t,0)} V \xra{p} Q.$ Hence
    Lemma~\ref{homotopic-zero} shows that this composition $C
    \xra{(\tildew{q},0)} E \xra{\epsilon} L$ is homotopic to
    zero.  So it is zero in the triangulated category
    $[\Coh(X,W)],$ and the morphism $\epsilon \colon  E \ra L$ factors
    there through $\Cone((\tildew{q},0)) \in \Acycl[\MF(X,W)].$
    This proves our claim.
  \end{proof}

  \begin{lemma}
    \label{l:conditions-hold-for-squares-4-and-5}
    The three equivalent conditions
    \ref{enum:verdier-loc-precompose-to-Mor-V},
    \ref{enum:verdier-loc-WC-factors-WVC},
    \ref{enum:verdier-loc-Hom-WD-in-DmodV-equal-those-in-DmodC}
    hold for the squares
    $(\boxast 4)$ and $(\boxast 5).$
  \end{lemma}

  \begin{proof}
    The proof of Lemma~\ref{l:conditions-hold-for-square1} for
    arbitrary $W \colon X \ra \DA^1$ can easily be modified to show this
    result. Observe that the kernel of a surjective morphism of
    locally free sheaves (resp.\ flat quasi-coherent sheaves) is
    again locally free (resp.\ flat quasi-coherent).
  \end{proof}

  The proof of Theorem~\ref{t:equivalences-curved-categories} is
  complete.
\end{proof}

We deduce some corollaries from the proof of
Theorem~\ref{t:equivalences-curved-categories}.

\begin{corollary}
  \label{c:intersections-with-Acycl}
  We have
  \begin{align*}
    [\InjQcoh(X,W)]\cap \Acycl[\Qcoh(X,W)] & =0,\\
    [\MF(X,W)] \cap \Acycl[\Coh(X,W)] & = \Acycl[\MF(X,W)],\\
    [\Coh(X,W)] \cap \Acycl[\Qcoh(X,W)] & = \Acycl[\Coh(X,W)],\\
    [\Locfree(X,W)] \cap \Acycl[\Qcoh(X,W)] & =
    \Acycl[\Locfree(X,W)],\\ 
    [\FlatQcoh(X,W)] \cap \Acycl[\Qcoh(X,W)] & =
    \Acycl[\FlatQcoh(X,W)].
  \end{align*}
\end{corollary}

\begin{proof}
  The first equality follows immediately from
  Lemma~\ref{l:sufficient-for-inj-DQcoh-equiv}.
  Let $E \in [\MF(X,W)] \cap \Acycl[\Coh(X,W)].$ 
  We have seen in Lemma~\ref{l:conditions-hold-for-square1}
  that condition
  \ref{enum:verdier-loc-WC-factors-WVC}
  holds for the square $(\boxast 1).$
  Applied to $\id_E,$ this condition shows that $E$ is a direct
  summand of an object of  
  $\Acycl[\MF(X,W)]$ and hence in $\Acycl[\MF(X,W)].$
  This proves the second equality. The remaining equalities 
  are proved similarly using
  Lemmata~\ref{l:factoring-through-coherent} and
  \ref{l:conditions-hold-for-squares-4-and-5}.
\end{proof}

\begin{corollary}
  [{cf.\ proof of \cite[Thm.~3.6]{positselski-two-kinds}}]
  \label{c:left-orthogonal-of-InjQcoh}
  Let
  $\strict([\InjQcoh(X,W)])$
  be the strict closure of $[\InjQcoh(X,W)]$ in $[\Qcoh(X,W)].$
  Then 
  \begin{equation*}
    [\Qcoh(X,W)]= \big\langle
    \strict([\InjQcoh(X,W)]),
    \Acycl[\Qcoh(X,W)]
    \big\rangle
  \end{equation*}
  is a semi-orthogonal decomposition
  (see Def.~\ref{d:semi-orthogonal-decomposition}).
  In particular,
  $\Acycl[\Qcoh(X,W)]$
  is the left orthogonal of $[\InjQcoh(X,W)$ in $[\Qcoh(X,W)].$
\end{corollary}

\begin{proof}
  Lemma~\ref{l:resolutions}.\ref{enum:Inj-reso}
  yields for each $F \in [\Qcoh(X,W)]$ a triangle
  $A \ra F \ra J \ra [1]A$ with $A \in \Acycl[\Qcoh(X,W)]$
  and 
  $J \in [\InjQcoh(X,W)].$
  Together with Lemma~\ref{l:sufficient-for-inj-DQcoh-equiv}
  this proves the first claim. The second claim follows from
  Lemma~\ref{l:first-properties-semi-orthog-decomp}.\ref{enum:semi-orth-and-perps}.
\end{proof}

\begin{corollary}
  \label{c:direct-sums}
  The categories 
  $[\Qcoh(X,W)],$
  $[\InjQcoh(X,W)],$
  $\Acycl[\Qcoh(X,W)],$ and $\DQcoh(X,W)$ are cocomplete
  (closed under 
  arbitrary direct sums) and therefore Karoubian, and the functor
  $[\Qcoh(X,W)] \ra \DQcoh(X,W)$ preserves direct sums.
\end{corollary}

\begin{proof}
  It is clear that $[\Qcoh(X,W)]$ is cocomplete.
  Note that
  $[\InjQcoh(X,W)]$ is cocomplete by
  Theorem~\ref{t:injective-in-qcoh-vs-all-OX}.\ref{enum:inj-and-inj-qcoh-arbitrary-sums},
  and that
  $\Acycl[\Qcoh(X,W)]$ is cocomplete as the left orthogonal
  of $[\InjQcoh(X,W)]$ in $[\Qcoh(X,W)],$ see
  Lemma~\ref{c:left-orthogonal-of-InjQcoh}.
  Now use
  \cite[Lemma~1.5 and Prop.~3.2]{neeman-homotopy-limits}.
  Cocompleteness of $\DQcoh(X,W)$ follows also from
  Theorem~\ref{t:equivalences-curved-categories}.\ref{enum:injQcoh-DQcoh-equiv}.
\end{proof}

The following definition should be compared with 
Definition~\ref{d:absolute-derived-categories}.
Note that $[\Sh(X,W)$ and $[\Qcoh(X,W)]$ are cocomplete.

\begin{definition}
  \label{d:coderived-category}
  Denote by $\Acycl^\co[\Sh(X,W)]$ the full triangulated
  subcategory of $\Sh(X,W)$ that contains $\Acycl[\Sh(X,W)]$
  and is closed under arbitrary
  direct sums. 
  Following
  \cite{positselski-two-kinds,positselski-coh-analogues-matrix-fact-sing-cats} again
  we define
  the \define{coderived category
    $\DSh^\co(X,W)$ of $W$-curved dg sheaves}
  as the Verdier quotient
  \begin{equation*}
    \DSh^\co(X,W):=[\Sh(X,W)]/\Acycl^\co[\Sh(X,W)].
  \end{equation*}
\end{definition}

If we define $\DQcoh^\co(X,W)$ similarly, 
Corollary~\ref{c:direct-sums} shows that $\DQcoh(X,W) =
\DQcoh^\co(X,W).$ 

\begin{theorem}
  \label{t:big-curved-categories}
  \rule{1mm}{0mm}
  \begin{enumerate}
  \item
    \label{enum:InjMod-DModco-equiv}
    The functor $[\InjSh(X,W)] \ra \DSh^\co(X,W)$ is an
    equivalence. 
  \item
    \label{enum:DQcoh-DModco-ff}
    The functor $\DQcoh(X,W) \ra \DSh^\co(X,W)$ is full and
    faithful.
  \end{enumerate}
\end{theorem}

\begin{proof}
  \ref{enum:InjMod-DModco-equiv} implies
  \ref{enum:DQcoh-DModco-ff}:
  Note that we have a full and faithful functor $\InjQcoh(X,W)
  \ra \InjSh(X,W)$ by 
  Theorem~\ref{t:injective-in-qcoh-vs-all-OX}.\ref{enum:inj-qcoh-equal-inj-OX-that-qcoh}. Hence
  $[\InjQcoh(X,W)] \ra [\InjSh(X,W)]$ is full and faithful, and
  we can use
  Theorem~\ref{t:equivalences-curved-categories}.\ref{enum:injQcoh-DQcoh-equiv}.

  \ref{enum:InjMod-DModco-equiv}:
  Adapting the proof of
  Lemma~\ref{l:resolutions}.\ref{enum:Inj-reso} 
  shows: 
  For any $F\in \Sh(X,W)$ there exists an exact sequence
  $0\ra F\ra I^0 \ra I^1 \ra \dots$ in $Z_0(\Sh(X,W))$
  with all $I^j \in \InjSh(X,W).$ 
  It follows from
  Lemma~\ref{l:totalization-of-bounded-below-exact-complex-in-Acyclco}
  below that
  the obvious morphism 
  $F \ra \Tot(I)$ has cone in $\Acycl^\co[\Sh(X,W)]$
  and hence becomes an isomorphism
  in $\DSh^\co(X,W).$
  Theorem~\ref{t:injective-in-qcoh-vs-all-OX}.\ref{enum:inj-and-inj-qcoh-arbitrary-sums} 
  shows that $\Tot(I) \in \InjSh(X,W).$
  This implies that 
  $[\InjSh(X,W)] \ra \DSh^\co(X,W)$ is essentially surjective.

  Adapting the proof of Lemma~\ref{l:sufficient-for-inj-DQcoh-equiv}
  shows that the left orthogonal of
  $[\InjSh(X,W)]$ in $[\Sh(X,W)]$ contains 
  $\Acycl[\Sh(X,W)]$ and hence $\Acycl^\co[\Sh(X,W)]$ since any
  left orthogonal is stable under direct sums.
  Now use
  condition~\ref{enum:verdier-loc-CW-factors-CVW}
  of Proposition~\ref{p:verdier-localization-induced-functor-ff}.
\end{proof}

\begin{lemma}
  \label{l:totalization-of-bounded-below-exact-complex-in-Acyclco}
  If $F$ is a bounded below exact complex in $Z_0(\Sh(X,W)),$
  then $\Tot(F) \in \Acycl^\co[\Sh(X,W)].$
\end{lemma}

\begin{proof}
  We can assume that $F=(\ldots \ra 0 \ra F^0 \ra F^1 \ra
  \dots).$ Let $F_{\leq n}$ be the subcomplex that coincides with
  $F$ in degrees $<n,$ is zero in degrees $>n,$ and whose degree
  $n$ component is the kernel of $F^n \ra F^{n+1}.$ We have
  monomorphisms $F_{\leq 0} \ra F_{\leq 1} \ra F_{\leq 2} \ra \ldots$
  of bounded exact complexes, and $F=\colim F_{\leq n}.$
  Note that there is a short exact sequence 
  \begin{equation*}
    0 \ra 
    \bigoplus_{n \in \DN} F_{\leq n} \ra
    \bigoplus_{n \in \DN} F_{\leq n} \ra
    F \ra 0
  \end{equation*}
  of complexes in $Z_0(\Sh(X,W)).$ Totalizing yields a short
  exact sequence 
  \begin{equation*}
    0 \ra \bigoplus_{n \in \DN} \Tot(F_{\leq n}) \ra
    \bigoplus_{n \in \DN} \Tot(F_{\leq n}) \ra
    \Tot(F) \ra 0
  \end{equation*}
  in $Z_0(\Sh(X,W)).$
  Part~\ref{enum:ses-gives-triangle}
  of Lemma~\ref{l:totalization} 
  shows that this short exact sequence yields a triangle in
  $\DSh(X,W)$ and a fortiori in 
  $\DSh^\co(X,W),$
  and part~\ref{enum:totalization-bounded-exact-complex} 
  of the same lemma shows that
  $\bigoplus_{n \in \DN} \Tot(F_{\leq n}) \in \Acycl^\co[\Sh(X,W)].$
  Hence $\Tot(F)$ becomes zero in 
  $\DSh^\co(X,W).$ The claim follows.
\end{proof}

\begin{remark}
  \label{rem:results-if-injective-dim-finite}
  We don't know whether $\Sh(X)$ has finite injective
  dimension.
  If this is the case the method 
  used to prove Theorem~\ref{t:equivalences-curved-categories}.\ref{enum:injQcoh-DQcoh-equiv}
  easily implies that
  $[\InjSh(X,W)] \ra \DSh(X,W)$ is an equivalence; 
  moreover
  Theorem~\ref{t:big-curved-categories}.\ref{enum:InjMod-DModco-equiv}
  then 
  shows that $\DSh(X,W) = \DSh^\co(X,W)$ and
  $\Acycl[\Sh(X,W)]=\Acycl^\co[\Sh(X,W)].$
\end{remark}

\subsection{Case of constant \texorpdfstring{$W$}{W}}
\label{sec:case-const-texorpdfs}

We study the case that $W$ is a constant function; 
recall that this means that $W(X)$ consists of a single point
of $\DA^1=\Spec k[T]$ which is then necessarily closed.
First we note that the case of a constant nonzero $W$ is not
interesting.

\begin{lemma}
  \label{l:case-W=constant-nonzero}
  Assume that the function $W$ is constant but $W \neq 0.$ Then
  $[\Sh(X,W)]=0.$
  In particular, all the subcategories
  $[\Qcoh(X,W)], \dots, [\MF(X,W)]$
  and all the quotient categories
  $\DSh(X,W), \DQcoh(X,W), \dots, \bfMF(X,W)$ are zero.
\end{lemma}

\begin{proof}
  The assumption implies that the morphism $k[T] \ra \Gamma(X,
  \mathcal{O}_X),$ $T \mapsto W,$ factors as $k[T] \ra
  k[T]/\primp \ra \Gamma(X, \mathcal{O}_X)$ where $\primp \subset
  k[T]$ is a maximal ideal $\not=(T).$ In particular $T$ is
  invertible in the field $k[T]/\primp,$ so $W$ is invertible in
  $\Gamma(X, \mathcal{O}_X).$

  Hence for any $E \in \Sh(X,W)$ the degree one morphism
  \begin{equation*}
    h:=(W\inv e_0, 0) \in \End_{\Sh(X,W)}(E)_1=\Hom_{\Sh(X)}(E_0,
    E_1) \oplus \Hom_{\Sh(X)}(E_1, E_0)  
  \end{equation*}
  satisfies $d h =\id_E,$ i.\,e.\ $E$ is isomorphic to zero in
  $[\Sh(X,W)].$
\end{proof}

Hence let us study the case $W=0.$
Given an object $E\in \Qcoh(X,0)$ we may consider its
cohomology
$H(E)$ which is just a graded quasi-coherent sheaf with
components 
$H_0(E)$ and $H_1(E).$ 
Let
\begin{equation*}
  \Ex[\Qcoh(X,0)] := \{E \in [\Qcoh(X,0)] \mid \text{$H_p(E)=0$
    for all $p \in \DZ_2$}\},
\end{equation*}
and define 
$\Ex[\Coh(X,0)], \dots,
\Ex[\FlatQcoh(X,0)]$
accordingly.
These categories are thick subcategories of
$[\Qcoh(X,0)], \dots, [\FlatQcoh(X,0],$ respectively, and we can
form the corresponding Verdier quotients. The next proposition shows
that this yields alternative definitions of the categories
$\DQcoh(X,0), \dots, \DFlatQcoh(X,0).$

Note that any morphism $f \colon E \ra F$ in
$Z_0(\Qcoh(X,0))$ induces a 
morphism $H(f) \colon  H(E) \ra H(F)$ on cohomology objects;
it is called a 
quasi-isomorphism if $H(f)$ is an isomorphism. 
It is easy to see that $H_0 \colon [\Qcoh(X,0)] \ra \Qcoh(X)$ is a
cohomological functor.

\begin{remark}
  \label{rem:ExMod-equals-AcyclMod-if-injective-dim-finite}
  These definitions clearly also make sense for $[\Sh(X,0)].$
  If we knew that $\Sh(X)$ has finite injective dimension, the
  obvious modification of the proof of the following
  proposition would show that 
  $\Ex[\Sh(X,0)] = \Acycl[\Sh(X,0)].$
\end{remark}

\begin{proposition}
  \label{p:case-W-equals-zero-acycl-equals-exact}
  Let $\mathcal{M}$ be $\Qcoh(X,0),$ $\Coh(X,0),$ 
  $\MF(X,0),$ 
  $\InjQcoh(X,0),$ 
  $\Locfree(X,0),$
  or
  $\FlatQcoh(X,0).$
  Then
  \begin{equation*}
    \Ex[\mathcal{M}] = \Acycl[\mathcal{M}]
  \end{equation*}
  and in particular $\op{D}\mathcal{M} =
  [\mathcal{M}]/\Ex[\mathcal{M}]$
  (where $\op{D}\MF(X,0):= \bfMF(X,0)$).
  A morphism $f$ in
  $Z_0(\mathcal{M})$ becomes an isomorphism in
  $\op{D}\mathcal{M}$ if and 
  only if $H(f)$ is a quasi-isomorphism. 
\end{proposition}

\begin{proof}
  We first prove that $\Ex[\Qcoh(X,0)] = \Acycl[\Qcoh(X,0)].$ A
  diagram chase
  (or Lemma~\ref{l:double-complex-upper-halfplane} below)     
  shows that the totalization of any short exact sequence (or any
  bounded exact complex) has
  vanishing cohomology. 
  By applying the
  cohomological functor $H_0,$
  any triangle in $[\Qcoh(X,0)]$ gives rise to a
  (6-periodic) long exact cohomology sequence, and
  obviously any 
  direct summand of an object with vanishing cohomology has
  vanishing cohomology.
  This implies that
  $\Ex[\Qcoh(X,0)] \supset \Acycl[\Qcoh(X,0)].$

  Conversely let $E \in \Ex[\Qcoh(X,0)].$ Let $U := \Kern e_0$ and $V:=
  \Kern e_1.$ Let
  \begin{align*}
    (U \ra I) & = (U \ra I^0 \xra{d^0_I} I^1  \xra{d^1_I} \dots
    \xra{d^{n-1}} I^n \ra 0),\\
    (V \ra J) & = (V \ra J^0 \xra{d^0_J} J^1  \xra{d^1_J} \dots
    \xra{d^{n-1}} J^n \ra 0),
  \end{align*}
  be finite injective resolutions in $\Qcoh(X).$
  Note that we have a short exact sequence $U \hra E_0 \sra V.$
  The injective resolutions of $U$ and $V$ combine to an injective
  resolution of $E_0:$ there is a morphism $r \colon  [-1]J \ra
  I$ of complexes in $\Qcoh(X)$ such that 
  its cone $\Cone(r)$ (which equals $I \oplus J$ if we forget the
  differential) fits into
  the following commutative diagram
  \begin{equation*}
    \xymatrix{
      {I} \ar[r]^-{
        \big(\begin{smallmatrix}
          1\\ 0
        \end{smallmatrix}
        \big)
      } &
      {\Cone(r)} \ar[r]^-{
        \big(\begin{smallmatrix}
          0 & 1
        \end{smallmatrix}
        \big)
      } &
      {J}\\
      {U} \ar[u] \iar[r] &
      {E_0} \ar[u] \sar[r] &
      {V} \ar[u]
    }
  \end{equation*}
  whose columns are injective resolutions.
  Similarly there is a morphism $s \colon  [-1]I \ra J$ and a commutative
  diagram
  \begin{equation*}
    \xymatrix{
      {J} \ar[r]^-{
        \big(\begin{smallmatrix}
          1\\ 0
        \end{smallmatrix}
        \big)
      } &
      {\Cone(s)} \ar[r]^-{
        \big(\begin{smallmatrix}
          0 & 1
        \end{smallmatrix}
        \big)
      } &
      {I}\\
      {V} \ar[u] \iar[r] &
      {E_1} \ar[u] \sar[r] &
      {U.} \ar[u]
    }
  \end{equation*}
  Let $A = (\matfak{A_1}{a_1}{A_0}{a_0})$ be the complex in
  $Z_0(\Qcoh(X,0))$ with $A_0=\Cone(r),$ $A_1=\Cone(s)$ and
  $a_0=\big(
  \begin{smallmatrix}
    0 & 1\\
    0 & 0
  \end{smallmatrix}
  \big)$
  and
  $a_1=\big(
  \begin{smallmatrix}
    0 & 1\\
    0 & 0
  \end{smallmatrix}
  \big).$
  Note that we obtain the bounded exact complex
  \begin{equation*}
    B:=(E \ra A) = (\dots \ra 0 \ra E \ra A^0 \ra A^1 \ra \dots \ra A^n
    \ra 0 \ra \dots)
  \end{equation*}
  in $Z_0(\Qcoh(X,0)).$
  From
  Lemma~\ref{l:totalization}.\ref{enum:brutal-truncation-and-totalization}
  we obtain a
  triangle
  \begin{equation*}
    E \ra \Tot(A) \ra \Tot(B) \ra [1]E
  \end{equation*}
  in $[\Qcoh(X,0)].$
  Note that $A^p$ is the direct sum of the two objects
  $\matfak{J^p}{0}{J^p}{1}$ and 
  $\matfak{I^p}{1}{I^p}{0}.$ Hence
  Lemma~\ref{l:totalization}.\ref{enum:totalization-of-finite-complex-components-null-homotopic}
  implies that $\Tot(A) = 0$ in $[\Qcoh(X,0)].$
  Hence
  $\Tot(B) \sira [1]E$ in $[\Qcoh(X,0)],$ so 
  $E \in \Acycl[\Qcoh(X,0)]$
  by Lemma~\ref{l:totalization}.\ref{enum:totalization-bounded-exact-complex}.
  This proves $\Ex[\Qcoh(X,0)] = \Acycl[\Qcoh(X,0)].$

  Now let $\mathcal{M}$ be as in the proposition.
  Then
  $\Ex[\mathcal{M}]\supset \Acycl[\mathcal{M}]$ is proved as
  above, and 
  Corollary~\ref{c:intersections-with-Acycl}
  yields
  \begin{equation*}
    \Ex[\mathcal{M}] \subset [\mathcal{M}] \cap \Ex[\Qcoh(X,0)] 
    =
    [\mathcal{M}] \cap \Acycl[\Qcoh(X,0)]
    = \Acycl[\mathcal{M}].
  \end{equation*}

  The last statement is clear: $f$ becomes an isomorphism if and
  only if its cone is in $\Ex[\mathcal{M}];$ now use the
  six-periodic long exact sequence obtained from the
  cohomological functor $H_0.$
\end{proof}

\begin{remark}
  \label{rem:ExQcohX0-totalizations}
  In fact we have proved that each object of $\Ex[\Qcoh(X,0)] =
  \Acycl[\Qcoh(X,0)]$ is isomorphic to the totalization of a
  bounded exact complex in $Z_0(\Qcoh(X,0).$
\end{remark}

\subsection{Derived functors}
\label{sec:derived-functors}

We recall first some general results about derived functors and
then apply them to direct and inverse image functors, and to Hom
and tensor functors.

\subsubsection{Reminder on derived functors}
\label{sec:remind-deriv-funct}

We recall results and terminology from the elegant exposition of
derived functors in \cite{murfet-triang-cat-I} and refer the
reader to this note for more details. 
Let $\mathcal{D}$ be a triangulated category 
$\mathcal{D}$ with a strict full triangulated subcategory
$\mathcal{C},$ and let $F \colon  \mathcal{D} \ra
\mathcal{T}$ be a triangulated functor to some other triangulated
category $\mathcal{T}.$ The question is whether $F$ 
has a right derived functor 
$\bR F \colon  \mathcal{D}/\mathcal{C}
\ra \mathcal{T}$
with respect to $\mathcal{C}.$ 
More precisely, a right derived functor of $F$ with respect to
$\mathcal{C}$ is a pair $(\bR F,
\zeta)$ of a triangulated functor $\bR F \colon  \mathcal{D}
/\mathcal{C} \ra \mathcal{T}$ and a suitable natural
transformation $\zeta$ satisfying some universal property.

\begin{definition}
  \label{d:right-acyclic-objects}
  An object $A \in \mathcal{D}$ is \define{right $F$-acyclic (with
    respect to $\mathcal{C}$)} if the following condition holds:
  given any morphism $s \colon 
  A \ra D$ with cone in $\mathcal{C},$ there
  is a morphism $t \colon D
  \ra D'$ with cone in $\mathcal{C}$ such that $F(ts)$ is an
  isomorphism.
\end{definition}

Note that $F(A)=0$ if $A$ is right $F$-acyclic and in
$\mathcal{C}$ (apply the defining property to $A \ra 0$).

\begin{theorem}
  [{\cite[Thm.~116]{murfet-triang-cat-I}}]
  \label{t:right-derived-functor-explicit}
  In the above setting we additionally assume that $\mathcal{C}
  \subset \mathcal{D}$ is a thick subcategory.
  Suppose that for every $D \in \mathcal{D}$ there exists a
  morphism $\eta_D \colon  D \ra A_D$ with cone in $\mathcal{C}$ and
  $A_D$ 
  right $F$-acyclic with respect to $\mathcal{C}.$ 
  Then $F$ admits a 
  right derived functor $(\bR F,\zeta)$ with respect to
  $\mathcal{C}$ with the following
  properties:
  \begin{enumerate}
  \item 
    \label{enum:right-derived-functor-explicit-on-objects}
    For any $D \in \mathcal{D}$ we have $\bR F(D)=F(A_D)$ and
    $\zeta_D=F(\eta_D).$
  \item 
    \label{enum:right-derived-functor-explicit-characterize-right-F-acyclics}
    An object $D \in \mathcal{D}$ is right $F$-acyclic with
    respect to $\mathcal{C}$ if and
    only if $\zeta_D$ is an isomorphism in $\mathcal{T}.$
  \end{enumerate}
\end{theorem}

We will apply this theorem several times. When we then write $\bR
F$ later on we implicitly have used some fixed morphisms $\eta_D \colon 
D \ra A_D$ as in the theorem, or we say explicitly which morphism 
$\eta_D$ we use for a particular object $D.$
Usually we assume that $\eta_D=\id_D$ whenever $D$ is right
$F$-acyclic. 

\begin{remark}
  \label{rem:morphisms-under-RF}
  We explain how the functor $\bR F$ from
  Theorem~\ref{t:right-derived-functor-explicit} is defined on
  morphisms. 
  Let $\mathcal{A} \subset \mathcal{D}$ be the full subcategory
  of all right $F$-acyclic objects, and assume that the
  assumptions of Theorem~\ref{t:right-derived-functor-explicit}
  hold. Then in fact $\mathcal{A}$ is a triangulated subcategory,
  and $F$ vanishes on $\mathcal{A} \cap \mathcal{C}.$
  We obtain an induced triangulated functor
  $\ol{F} \colon \mathcal{A}/\mathcal{A}\cap 
  \mathcal{C} \ra \mathcal{T}.$  
  Moreover, the natural functor $\mathcal{A}/\mathcal{A}\cap
  \mathcal{C} \ra \mathcal{D}/\mathcal{C}$ is an equivalence,
  with a quasi-inverse induced by $D \mapsto A_D.$
  Then $\bR F$ is just the composition of this
  quasi-inverse with $\ol{F}.$ This determines $\bR F$ on
  morphisms. 
\end{remark}

Similar results hold for left derived functors.

\subsubsection{Direct and inverse image}
\label{sec:direct-inverse-image}

Let $Y$ be another scheme satisfying 
condition~\ref{enum:srNfKd},
and let $\pi  \colon Y\ra X$ be a
morphism. 
We denote the pullback function $\pi^*(W)$ on $Y$
again by $W.$ 

The usual direct image functor $\pi_* \colon \Qcoh(Y) \ra \Qcoh(X)$
induces
the dg functor $\pi_* \colon  \Qcoh(Y,W) \ra \Qcoh(X,W)$ and on homotopy
categories the triangulated functor
$\pi_* \colon [\Qcoh(Y,W)]\ra [\Qcoh(X,W)].$
Similarly, the usual inverse image functor $\pi^* \colon \Qcoh(X) \ra
\Qcoh(Y)$ induces a dg functor $\pi^* \colon \Qcoh(X,W) \ra \Qcoh(Y,W)$
and a triangulated functor $\pi^* \colon [\Qcoh(X,W)] \ra [\Qcoh(Y,W)].$
The adjunction $(\pi^*, \pi_*)$ in the usual setting induces an
adjunction of dg functors,
\begin{equation*}
  \Hom_{\Qcoh(Y,W)}(\pi^*(E), F) \sira
  \Hom_{\Qcoh(X,W)}(E, \pi_*(F)),
\end{equation*}
and then an adjunction on triangulated
functors. 
We also denote the compositions
\begin{equation*}
  [\Qcoh(Y,W)] \xra{\pi_*} [\Qcoh(X,W)] \ra \DQcoh(X,W)
\end{equation*}
and
\begin{equation*}
  [\Qcoh(X,W)] \xra{\pi^*} [\Qcoh(Y,W)] \ra \DQcoh(Y,W)
\end{equation*}
by $\pi_*$ and $\pi^*,$ respectively.

\begin{theorem}
  \label{t:derived-inverse-and-direct-image}
  \rule{1mm}{0mm}
  \begin{enumerate}
  \item
    \label{enum:direct-image-derived}
    The functor $\pi_* \colon  [\Qcoh(Y,W)] \ra \DQcoh(X,W)$ 
    has a right derived functor 
    $\bR \pi_* \colon  \DQcoh(Y,W) \ra \DQcoh(X,W)$ 
    with respect to $\Acycl[\Qcoh(Y,W)].$ 
  \item 
    \label{enum:inverse-image-derived}
    The
    functor
    $\pi^* \colon  [\Qcoh(X,W)] \ra \DQcoh(Y,W)$
    has a left derived functor $\bL \pi^* \colon  \DQcoh(X,W) \ra
    \DQcoh(Y,W)$ with respect to $\Acycl[\Qcoh(X,W)].$
    This left derived functor
    maps $\DCoh(X,W)$ to $\DCoh(Y,W)$ and
    $\bfMF(X,W)$ to $\bfMF(Y,W).$ We can assume
    that $\bL \pi^*=\pi^* \colon  \bfMF(X,W) \ra \bfMF(Y,W).$
  \item 
    \label{enum:adjunction-derived-images}
    There is an adjunction $(\bL \pi^*, \bR \pi_*)$ of
    triangulated functors.
  \end{enumerate}
\end{theorem}

\begin{proof}
  \ref{enum:direct-image-derived}:
  Lemma~\ref{l:resolutions}.\ref{enum:Inj-reso}
  provides for each $E \in [\Qcoh(Y,W)]$ a morphism
  $\eta_E \colon  E \ra I_E$ with $I_E \in [\InjQcoh(Y,W)]$ and cone in
  $\Acycl[\Qcoh(Y,W)].$
  Hence to apply Theorem~\ref{t:right-derived-functor-explicit}
  we need to show that any object $I \in [\InjQcoh(Y,W)]$
  is right $\pi_*$-acyclic with respect to 
  $\Acycl[\Qcoh(Y,W)].$ Let $s \colon  I \ra F$ be a morphism in
  $[\Qcoh(Y,W)]$ with cone in  
  $\Acycl[\Qcoh(Y,W)].$ 
  Apply $\Hom_{[\Qcoh(Y,W)]}(-,I)$ and use
  Lemma~\ref{l:sufficient-for-inj-DQcoh-equiv}. This shows that
  there is a (unique) morphism $g \colon  F \ra I$ in $[\Qcoh(Y,W)]$ with
  $gs=\id_I.$ The octahedral axiom implies that $g$ has cone in
  $\Acycl[\Qcoh(Y,W)],$ and $gs=\id_I$ certainly implies that
  $\pi_*(gs)$ is an isomorphism
  in $[\Qcoh(X,W)]$ and $\DQcoh(X,W).$

  \ref{enum:inverse-image-derived}:
  Lemma~\ref{l:resolutions}.\ref{enum:locfree-reso}
  yields for each $E \in [\Qcoh(X,W)]$ a morphism
  $\epsilon_E \colon  P_E \ra E$ with $P_E \in [\Locfree(X,W)] \subset
  [\FlatQcoh(X,W)]$ and cone in
  $\Acycl[\Qcoh(X,W)].$ 
  We want to use the left version of
  Theorem~\ref{t:right-derived-functor-explicit}.
  We need to show that any object $P \in [\FlatQcoh(X,W)]$ is left
  $\pi^*$-acyclic with respect to $\Acycl[\Qcoh(X,W)].$ 
  Let $s \colon F \ra P$ be a morphism in
  $[\Qcoh(Y,W)]$ with cone in  
  $\Acycl[\Qcoh(Y,W)].$ Consider the morphism $t:=\epsilon_F
  \colon P_F \ra F.$ We need to show that $\pi^*(s t)$
  is an isomorphism in $\DQcoh(Y,W).$ 
  The cone of $s t$ is in $\Acycl[\Qcoh(X,W)],$ and we
  can assume that it is in 
  $[\FlatQcoh(X,W)].$
  Hence it is enough to show
  that $\pi^*(Q) = 0$ in $\DQcoh(Y,W)$ for any 
  $Q \in \Acycl[\FlatQcoh(X,W)]= 
  [\FlatQcoh(X,W)] \cap \Acycl[\Qcoh(X,W)]$
  (see Corollary~\ref{c:intersections-with-Acycl}).
  Certainly we can reduce to the case that $Q=\Tot(G),$ where $G$
  is a short exact sequence in $Z_0(\FlatQcoh(X,W)).$
  But then $\pi^*(G)$ is a short exact sequence in 
  $Z_0(\Qcoh(Y,W)),$ and hence $\pi^*(Q)=\Tot(\pi^*(G))$ is zero
  in $\DQcoh(Y,W).$

  Lemma~\ref{l:resolutions}.\ref{enum:MF-reso}
  shows that we can take $P_E \in [\MF(X,W)]$
  for $E \in [\Coh(X,W)].$ For $E \in [\MF(X,W)]$ we take
  $P_E=E.$

  \ref{enum:adjunction-derived-images}:
  Apply
  \cite[Thm.~122]{murfet-triang-cat-I}, whose assumptions are
  satisfied by the proof of
  {\cite[Thm.~116]{murfet-triang-cat-I}}.
\end{proof}

\begin{remark}
  \label{rem:direct-inverse-image-coproducts}
  Both $\bL \pi^*$ and $\bR \pi_*$ preserve
  direct sums, cf.\ Corollary~\ref{c:direct-sums}.
  This is clear for $\bL \pi^*$ from the adjunction
  $(\bL \pi^*, \bR \pi_*).$
  For $\bR \pi_*$ this follows from 
  the above proof: 
  use Corollary~\ref{c:direct-sums}
  and the fact that $\pi_*$ 
  preserves direct sums since Noetherian schemes
  are quasi-compact.
\end{remark}

\begin{lemma}
  \label{l:pi-proper-Rpi-lower-star-and-MF-and-adjunction}
  Assume that the map $\pi$ is proper. Then the functor
  $\bR \pi_*$ maps $\bfMF(Y,W)$ to (the essential image of)
  $\bfMF(X,W),$ 
  and 
  $(\pi^*,\bR \pi_*)$ is an adjoint pair of functors between the
  categories $\bfMF(X,W)$ and $\bfMF(Y,W).$
\end{lemma}

\begin{proof}
  Let $E \in \bfMF(Y,W).$ Choose a finite
  resolution $E\ra I$ as in 
  Lemma~\ref{l:resolutions}.\ref{enum:Inj-reso}.
  Then
  $\bR \pi_*(E)$ is isomorphic to 
  $\pi_*(\Tot(I))=\Tot(\pi_*(I))$ and the
  cohomologies of the complex
  $\pi_*(I)$ all lie in $\Coh(X,W),$
  by \cite[Thm.~3.2.1]{EGAIII-i}.
  Hence
  $\bR \pi_*(E)$ is isomorphic to an object of $\DCoh(X,W)$
  by
  Lemma~\ref{l:useful-lemma}.\ref{enum:cohomologies-in-MF-then-Tot}
  below,
  and also to an object of $\bfMF(X,W)$
  by Theorem~\ref{t:equivalences-curved-categories}.\ref{enum:MF-Dcoh-equiv}.
  This proves the first claim.
  The second claim is a direct consequence of
  Theorem~\ref{t:derived-inverse-and-direct-image}.
\end{proof}

The proof of Theorem~\ref{t:derived-inverse-and-direct-image}
shows that all objects of $[\InjQcoh(Y,W)]$ are right
$\pi_*$-acyclic 
and that all objects of $[\FlatQcoh(X,W)]$ are left
$\pi^*$-acyclic.
Here is an improvement.

\begin{lemma}
  \label{l:componentwise-acyclics}
  Let $E \in [\Qcoh(Y,W)]$ and assume that its components $E_0,$
  $E_1$ are right $\pi_*$-acyclic quasi-coherent sheaves
  in the sense that
  $R^i\pi_*(E_p)=0$ for all $p \in \DZ_2$ and $i \in
  \DZ\setminus\{0\}.$ 
  Then $E$ is right $\pi_*$-acyclic, so in particular
  $\pi_*(E) \sira \bR\pi_*(E)$ canonically.

  Similarly, if the components of $F \in [\Qcoh(X,W)]$ are
  left $\pi^*$-acyclic quasi-coherent sheaves, then $F$ is left
  $\pi^*$-acyclic, and $\bL\pi^*(F) \sira \pi^*(F)$ canonically.
\end{lemma}

\begin{proof}
  Lemma~\ref{l:resolutions}.\ref{enum:Inj-reso}
  provides a finite resolution $E \ra I$ in
  $Z_0(\Qcoh(Y,W))$ 
  with components $I^l \in \InjQcoh(Y,W).$
  Since all involved quasi-coherent 
  sheaves are $\pi_*$-acyclic, 
  $\pi_*(E) \ra \pi_*(I)$ is still a resolution in
  $Z_0(\Qcoh(X,W)).$ 
  Hence the obvious morphism $\pi_*(E) \ra \Tot(\pi_*(I))$
  becomes an isomorphism in $\DQcoh(X,W).$
  On the other hand, if we use $E \ra \Tot(I)$ for
  computing $\bR 
  \pi_*(E),$ we have
  $\bR \pi_*(E) = \pi_*(\Tot(I)) =\Tot(\pi_*(I))$
  in $\DQcoh(X,W).$
  Now
  Theorem~\ref{t:right-derived-functor-explicit}.\ref{enum:right-derived-functor-explicit-characterize-right-F-acyclics}
  shows our first claim.
  The second claim is proved similarly using
  Lemma~\ref{l:resolutions}.\ref{enum:locfree-reso}.
\end{proof}

\begin{remark}
  \label{rem:derived-direct-image-for-affine-morphism}
  If $\pi$ is an affine morphism, 
  all objects of $[\Qcoh(Y,W)]$ are right $\pi_*$-acyclic
  by Lemma~\ref{l:componentwise-acyclics}, so 
  $\pi_* \colon [\Qcoh(Y,W)] \ra \DQcoh(X,W)$ maps $\Acycl[\Qcoh(Y,W)]$
  to zero. The induced functor
  $\pi_* \colon  \DQcoh(Y,W) \ra \DQcoh(X,W)$
  is canonically isomorphic to $\bR \pi_*.$

  If $\pi$ is proper and affine
  (for example a closed embedding),
  then all objects of $[\Coh(Y,W)]$ are right acyclic 
  for
  $\pi_* \colon [\Coh(Y,W)] \ra \DCoh(X,W)$
  with respect to $\Acycl[\Coh(Y,W)],$
  and hence
  $\pi_* = \bR \pi_* \colon  \DCoh(Y,W) \ra \DCoh(X,W)$
  canonically.

  Similarly, if $\pi$ is flat (for example an open embedding), we
  have $\pi^*=\bL\pi^*$ 
  canonically. 
\end{remark}

\begin{lemma}
  \label{l:useful-lemma}
  Let $F=(0 \ra F^m \xra{d^m} F^{m+1} \ra \ldots \ra F^{n-1}
  \xra{d^{n-1}} F^n \ra 0)$
  be a complex in $Z_0(\Qcoh(X,W)).$
  \begin{enumerate}
  \item
    \label{enum:cohomologies-in-MF-then-Tot}
    Consider the
    cohomologies $H^i(F)$ and the totalization $\Tot(F)$
    as objects of $\DQcoh(X,W).$
    Assume that each $H^i(F)$ is
    isomorphic
    to an object of $\bfMF(X,W)$ (resp.\ $\DCoh(X,W)$). 
    Then the same is true for $\Tot(F).$ 
  \item
    \label{enum:general-splitting-of-complexes-mf}
    Assume that
    \begin{equation*}
      \Hom_{\DQcoh(X,W)}(H^p(F), [v]H^q(F))=0
    \end{equation*}
    for all
    $p>q$ and $v \in \DZ_2$ (enough: with $v \equiv q+1-p \mod 2$).
    Then $\Tot(F) \cong \bigoplus_{i=m}^n [i]H^i(F)$ in $\DQcoh(X,W).$
  \end{enumerate}
\end{lemma}

\begin{proof}
  If $m=n$ all statements are trivial, so assume $m < n.$
  Consider the (vertical) short exact sequence of complexes
  \begin{equation*}
    \xymatrix@=4ex{
      {\tau_{\leq n-1}(F) \colon } \iar[d] &
      {\dots} \ar[r] &
      {F^{n-2}} \ar[r] \gar[d] &
      {\Kern d^{n-1}} \ar[r] \ar[d] &
      {0} \ar[r] \ar[d] &
      {0} \ar[r] \ar[d] &
      {\dots}\\
      {F \colon } \sar[d] &
      {\dots} \ar[r] &
      {F^{n-2}} \ar[r]^{d^{n-2}} \ar[d] &
      {F^{n-1}} \ar[r]^{d^{n-1}} \ar[d] &
      {F^{n}} \ar[r] \gar[d] &
      {0} \ar[r] \ar[d] &
      {\dots}\\
      {\tau_{>n-1}(F) \colon } &
      {\dots} \ar[r] &
      {0} \ar[r] &
      {\Bild d^{n-1}} \ar[r] &
      {F^{n}} \ar[r] &
      {0} \ar[r] &
      {\dots}
    }
  \end{equation*}
  It induces an exact sequence of their totalizations which becomes a
  triangle
  in $\DQcoh(X,W),$
  by Lemma~\ref{l:totalization}.\ref{enum:ses-gives-triangle}.
  The short exact sequence
  $\Bild d^{n-1} \hra F^{n} \sra{H^n(F)}$ gives rise to an
  isomorphism $\Tot(\tau_{> n-1}(F)) \sira [n]H^n(F)$
  in $\DQcoh(X,W)$
  by Lemma~\ref{l:totalization}.\ref{enum:brutal-truncation-and-totalization}.
  Hence we obtain the triangle
  \begin{equation}
    \label{eq:tot-of-trunc-triangle}
    \Tot(\tau_{\leq n-1}(F)) \ra
    \Tot(F) \ra
    [n]H^n(F)
    \ra
    [1]\Tot(\tau_{\leq n-1}(F))
  \end{equation}
  in $\DQcoh(X,W).$

  \ref{enum:cohomologies-in-MF-then-Tot}:
  By induction and our assumptions
  the first and third object of the triangle
  \eqref{eq:tot-of-trunc-triangle}
  are isomorphic to objects of $\bfMF(X,W)$ 
  (resp.\ $\DCoh(X,W)$).
  The same is then true for
  $\Tot(F).$

  \ref{enum:general-splitting-of-complexes-mf}:
  By induction the triangle \eqref{eq:tot-of-trunc-triangle}
  is isomorphic to the triangle
  \begin{equation*}
    \bigoplus_{i=m}^{n-1} [i]H^i(F)
    \ra
    \Tot(F) \ra
    [n]H^n(F)
    \ra
    [1]\bigoplus_{i=m}^{n-1} [i]H^i(F).
  \end{equation*}
  By assumption the third morphism in this triangle vanishes, and
  hence $\Tot(F)$ is the direct sum of the first and the third object.
\end{proof}

\subsubsection{Sheaf Hom and tensor product}
\label{sec:sheaf-hom-tensor-product}

Let $W$ and $V$ be arbitrary morphisms
$X \ra \DA^1.$
For $P \in \Qcoh(X, W)$ and
$Q \in \Qcoh(X,V)$ 
consider the following diagram\footnote{
  If $A$ and $B$ are complexes with differentials $d_A=a$ and
  $d_B=b,$ 
  the differential $d$ in the
  Hom-complex $\Hom(A,B)$ is given by 
  $d(f)= b \comp f -(-1)^{|f|} f \comp a$ 
  for homogeneous $f$ of degree $|f|.$
  This explains the signs.
}
\begin{equation}
  \label{eq:def-sheafHom-QcohXWV}
  \xymatrix{
    {
      {
        \sheafHom_{\mathcal{O}_X}(P_1, Q_0)
        \oplus
        \sheafHom_{\mathcal{O}_X}(P_0, Q_1) 
      }
    }
    \ar@<0.4ex>[rrr]^-{
      \left[
        \begin{smallmatrix}
          -p_0^* 
          &
          q_{1*} 
          \\
          q_{0*}
          &
          -p_1^* 
        \end{smallmatrix}
      \right]
    } &&&
    {
      \sheafHom_{\mathcal{O}_X}(P_0, Q_0) \oplus \sheafHom_{\mathcal{O}_X}(P_1, Q_1).}
    \ar@<0.4ex>[lll]^-{
      \left[
        \begin{smallmatrix}
          -p_1^* &
          q_{1*}
          \\
          q_{0*} &
          -p_0^* 
        \end{smallmatrix}
      \right]
    }
  }
\end{equation}
in the category $\Sh(X).$ 
It is easy to check that both compositions are multiplication by
$V-W:$ note for example that $p_0^*p_1^*=-(p_1p_0)^*$ by the
usual sign convention, since $p_0$ and $p_1$ both have degree
one. Hence this 
diagram defines an object
$\sheafHom(P,Q)$
of $\Sh(X,V-W).$ 

\begin{remark}
  \label{rem:sheafHom-case-V-equals-W}
  In case $W=V$ note that $\sheafHom(P,Q)$ is in $\Sh(X,0),$
  i.\,e.\ it is a dg sheaf, and that
  \begin{equation}
    \label{eq:Qcoh-dg-Homs-as-global-sections}
    \Hom_{\Qcoh(X,W)}(P,Q) = \Gamma(X, \sheafHom(P,Q))
  \end{equation}
  as a dg abelian group.
\end{remark}

In fact $(P,Q) \mapsto \sheafHom(P,Q)$ is a dg
bifunctor
\begin{equation*}
  \sheafHom(-,-) \colon 
  \Qcoh(X,W)^\opp \times \Qcoh(X,V) \ra \Sh(X, V-W).
\end{equation*}
It induces a bifunctor
$\sheafHom(-,-) \colon 
[\Qcoh(X,W)]^\opp \times [\Qcoh(X,V)] \ra [\Sh(X, V-W)]$ of
triangulated categories. 
For fixed $P \in [\Qcoh(X,W)]$ the obvious composition
$\sheafHom(P,-) \colon  [\Qcoh(X,V)] \ra \DSh(X, V-W)$ has
a right derived functor 
with respect to $\Acycl[\Qcoh(X,V)]:$
we construct it by choosing morphisms $Q \ra I_Q$
with $I_Q \in [\InjQcoh(X,V)]$ and cone in $\Acycl[\Qcoh(X,V),$
for every $Q \in [\Qcoh(X,V)],$ and then proceed as in the
construction of $\bR \pi_*$ in the proof
of Theorem~\ref{t:derived-inverse-and-direct-image}.

For fixed $I \in [\InjQcoh(X,V)],$ the functor
$\sheafHom(-,I) \colon  [\Qcoh(X,W)]^\opp \ra \DSh(X, V-W)$
maps $\Acycl[\Qcoh(X,W)]$ to zero since
$\sheafHom(-,I)$ maps short exact sequences in $Z_0(\Qcoh(X,W))$
to short exact sequences in $Z_0(\Sh(X,V-W))$
(use
Theorem~\ref{t:injective-in-qcoh-vs-all-OX}.\ref{enum:inj-qcoh-equal-inj-OX-that-qcoh}).
We define 
\begin{align*}
  \bR\sheafHom(-,-) \colon 
  \DQcoh(X,W)^\opp \times \DQcoh(X,V) & \ra \DSh(X, V-W),\\
  (P,Q) & \mapsto \sheafHom(P,I_Q),
\end{align*}
and leave it to the reader to check that this defines a bifunctor
of triangulated categories. Note that for $(P,Q) \in  
[\Qcoh(X,W)]^\opp \times [\Qcoh(X,V)]$ there is a natural
morphism
\begin{equation*}
  \sheafHom(P,Q) \ra \bR\sheafHom(P,Q) =\sheafHom(P,I_Q)
\end{equation*}
in $\DSh(X, V-W)$
induced by $Q \ra I_Q.$ It is an isomorphism if $P$ is in
$[\MF(X,W)]$ (or in $[\Locfree(X,W)]$),
or of course if $Q$ is
in $[\InjQcoh(X,V)].$ This also shows that if $P$ is in $\Coh(X,W)$
and we choose 
$F_P \ra P$ in $[\Coh(X,W)]$ 
with $F_P \in [\MF(X,W)]$
and cone in $\Acycl[\Qcoh(X,W)],$
then the morphisms
\begin{equation*}
  \sheafHom(F_P,Q) \ra \sheafHom(F_P, I_Q) \la \sheafHom(P,I_Q)
\end{equation*}
become isomorphisms in $\DSh(X, V-W).$ This gives another way of
computing $\bR\sheafHom(P,Q)$ for $P \in \Coh(X,W).$

Note also that $\bR\sheafHom(P,Q)$ is in $\DQcoh(X,V-W)$ if 
$P \in \DCoh(X,W),$ and in (the essential image of )
$\DCoh(X,V-W)$ if 
$P \in \DCoh(X,W)$ and $Q \in \DCoh(X,V).$ 

We can also directly obtain a bifunctor
\begin{equation}
  \label{eq:dg-bifunctor-sheafhom-MF}
  \sheafHom(-,-) \colon 
  \bfMF(X,W)^\opp \times \bfMF(X,V) \ra \bfMF(X, V-W)
\end{equation}
of triangulated categories. It is isomorphic to the restriction
of $\bR\sheafHom(-,-)$ to $\bfMF(X,W)^\opp \times \bfMF(X,V).$
Slightly more general this works also for
$\sheafHom(-,-) \colon 
\bfMF(X,W)^\opp \times \DCoh(X,V) \ra \DCoh(X, V-W).$

For $P \in \Qcoh(X, W)$ and
$Q \in \Qcoh(X,V)$ note that
\begin{equation*}
  \xymatrix{
    {(P_1 \otimes_{\mathcal{O}_X} Q_0)
      \oplus
      (P_0 \otimes_{\mathcal{O}_X} Q_1)}
    \ar@<0.4ex>[rrr]^-{
      \left[
        \begin{smallmatrix}
          \id \otimes q_0 & 
          p_0 \otimes \id \\
          p_1 \otimes \id &
          \id \otimes q_1 
        \end{smallmatrix}
      \right]
    } &&&
    {(P_1 \otimes_{\mathcal{O}_X} Q_1) \oplus (P_0 \otimes_{\mathcal{O}_X} Q_0)}
    \ar@<0.4ex>[lll]^-{
      \left[
        \begin{smallmatrix}
          \id \otimes q_1 & 
          p_0 \otimes \id \\
          p_1 \otimes \id & 
          \id \otimes q_0
      \end{smallmatrix}
      \right]
    }
  }
\end{equation*}
defines an object $P \otimes Q$ of $\Qcoh(X, V+W).$
We obtain a dg bifunctor
\begin{equation*}
  (- \otimes -) \colon  \Qcoh(X,W) \times \Qcoh(X,V) \ra \Qcoh(X,
  V+W)
\end{equation*}
and a bifunctor of triangulated categories on homotopy
categories.
For fixed $P \in [\Qcoh(X,W)]$ the obvious composition
$(P \otimes -) \colon  [\Qcoh(X,V)] \ra \DSh(X, V+W)$ has
a left derived functor 
with respect to $\Acycl[\Qcoh(X,V)]:$
for each $Q \in [\Qcoh(X,V)]$ we fix a morphism
$F_Q \ra Q$ with $F_Q \in [\Locfree(X,V)] \subset
[\FlatQcoh(X,V)]$ and cone in
$\Acycl[\Qcoh(X,V)]$ 
and proceed then as in the
construction of $\bL \pi^*$ in the proof
of Theorem~\ref{t:derived-inverse-and-direct-image}.
It is then easy to see that
\begin{align*}
  (- \otimes^\bL -) \colon  \DQcoh(X,W) \times \DQcoh(X,V) & \ra \DQcoh(X,
  V+W),\\
  (P,Q) \mapsto P \otimes F_Q,
\end{align*}
defines a bifunctor of triangulated categories. 
Again we have for $(P,Q) \in  
[\Qcoh(X,W)] \times [\Qcoh(X,V)]$ a natural
morphism
\begin{equation*}
  P \otimes^\bL Q = P \otimes F_Q \ra P \otimes Q 
\end{equation*}
in $\DSh(X, V+W)$
induced by $F_Q \ra Q$ which is an isomorphism if $P$ or $Q$ has
flat components.

Note that there is an obvious isomorphism
\begin{equation}
  \label{eq:adjunction-tensor-sheafHom}
  \Hom_{\Qcoh(X, W+V)}(P \otimes Q, R) \sira
  \Hom_{\Qcoh(X, W)}(P, \sheafHom(Q, R))
\end{equation}
of dg modules which is natural in 
$P \in \Qcoh(X,W),$ 
$Q \in \Coh(X,V),$
and
$R \in \Qcoh(X,W+V).$

\begin{lemma}
  \label{l:adjunction-Hom-tensor}
  We have
  \begin{equation*}
    \Hom_{\DQcoh(X, W+V)}(P \otimes^\bL Q, R) \cong
    \Hom_{\DQcoh(X, W)}(P, \bR\sheafHom(Q, R))
  \end{equation*}
  naturally in $P \in \DQcoh(X,W),$ 
  $Q \in \DCoh(X,V),$
  and
  $R \in \DQcoh(X,W+V).$
\end{lemma}

\begin{proof}
  First note that we can assume that $Q \in \bfMF(X,V)$ 
  and that moreover $R \in \InjQcoh(X,W+V).$
  Then $P \otimes^\bL Q \sira P \otimes Q$ and 
  $\sheafHom(Q, R) \sira \bR\sheafHom(Q, R).$ Note also that
  $\sheafHom(Q, R) \in \InjQcoh(X,W)$
  by \cite[Prop.~7.17]{hartshorne-residues-duality}.
  Now take $H_0$ of the above isomorphism
  \eqref{eq:adjunction-tensor-sheafHom}
  of dg modules and use
  Remark~\ref{rem:morphisms-to-injectives}.
\end{proof}

\subsubsection{External tensor product}
\label{sec:extern-tens-prod}

Let $Y$ be a scheme such that $Y$ and $X \times Y$ satisfy 
condition~\ref{enum:srNfKd}, and let $V \colon  Y \ra \DA^1$ be a
morphism.
Let $p \colon X \times Y \ra X$ and $q \colon X \times Y \ra Y$ be the
projections. We define $W*V:= p^*(W)+q^*(V),$ so 
$(W*V)(x,y)=W(x)+V(y).$
We define the dg bifunctor $\boxtimes$ by
\begin{equation*}
  (- \boxtimes -) := (p^*(-) \otimes q^*(-)) \colon  \Qcoh(X,W) \times
  \Qcoh(Y,V) \ra \Qcoh(X \times Y, W*V).
\end{equation*}
This functor immediately induces the
bifunctor 
\begin{equation*}
  (- \boxtimes -) \colon  \DQcoh(X,W) \times
  \DQcoh(Y,V) \ra \DQcoh(X \times Y ,W*V) 
\end{equation*}
of triangulated categories. This functor 
coincides with the composition $(- \otimes^\bL -) \comp (\bL p^*
\times \bL q^*)$ (since $p$ and $q$ are flat we have $\bL p^* \sira
p^*$ and $\bL q^* \sira q^*,$ and
moreover $(- \otimes^\bL -) \sira (- \otimes -)$ on objects of
the form $(p^*(P), q^*(Q)),$ cf.\ the proof of
Lemma~\ref{l:componentwise-acyclics}). 
Note also that $P \boxtimes Q$ is in $\bfMF(X \times Y, W*V)$
(resp.\ $\DCoh(X \times Y, V*W)$) if 
$P \in \bfMF(X, W)$ and $Q \in \bfMF(Y,V)$ 
(resp.\ $P \in \DCoh(X, W)$ and $Q \in \DCoh(Y,V).$ 

\subsubsection{Duality}
\label{sec:duality}

We introduce a duality $D_X$ on the category of matrix
factorizations.
Let $\mathcal{D}_X:=(\matfak{0}{}{\mathcal{O}_X}{}) \in \MF(X,0);$ 
note that
$\mathcal{O}_X$ sits in even degree.
Then
\begin{equation*}
  D_X := (-)^\cek:= \sheafHom(-,\mathcal{D}_X) \colon  \MF(X,W) \ra
  \MF(X, -W)^\opp 
\end{equation*}
is a equivalence of dg categories and induces an equivalence
\begin{equation}
  \label{eq:duality-derived-level}
  D_X := (-)^\cek:= \sheafHom(-,\mathcal{D}_X) \colon  \bfMF(X,W) \ra
  \bfMF(X, -W)^\opp 
\end{equation}
of triangulated categories.
This is just the functor 
\eqref{eq:dg-bifunctor-sheafhom-MF}
with $\mathcal{D}_X$ as its fixed second argument.
We refer to $D_X$ as the duality since $D_X^2=\id$
naturally.
Explicitly, $D_X$ maps
$P=\big(\matfak{P_1}{p_1}{P_0}{p_0}\big)$
to
\begin{equation}
  \label{eq:duality-explicitly}
  D_X(P)=P^\cek= \Big(
  \xymatrix{
    {P^\cek_1=\sheafHom(P_1,\mathcal{O}_X)}
    \ar@<0.4ex>[rr]^-{p^\cek_1=-p^*_0}
    &&
    {P^\cek_0=\sheafHom(P_0, \mathcal{O}_X).}
    \ar@<0.4ex>[ll]^-{p^\cek_0=-p^*_1}
  }
  \Big).
\end{equation}
Occasionally we view the duality as the functor
$D_X =\bR \sheafHom(-, \mathcal{D}_X) \colon  \DCoh(X,W) \ra
\bfMF(X,-W)^\opp.$
The next lemma says that the inverse image functor and duality
commute.

\begin{lemma}
  \label{l:pi-us-and-duality-MF-setting}
  Let $\pi \colon  Y \ra X$ be a morphism of 
  schemes satisfying 
  condition~\ref{enum:srNfKd}, and let $W\colon X \ra \DA^1$ be a morphism.
  Then there is an isomorphism 
  $\pi^* \comp D_X
  \sira D_Y \comp \pi^*$ of functors 
  $\bfMF(X,W) \sira \bfMF(Y,-W)^\opp.$
\end{lemma}

\begin{proof}
  For $\mathcal{F} \in \Qcoh(X)$ 
  consider the morphism
  \begin{equation*}
    \sheafHom_{\mathcal{O}_{X}}(\mathcal{F}, \mathcal{O}_X)
    \ra
    \sheafHom_{\mathcal{O}_{X}}(\mathcal{F},
    \pi_*\mathcal{O}_Y)
    = \pi_*\sheafHom_{\mathcal{O}_Y}
    (\pi^*\mathcal{F}, \mathcal{O}_{Y}).
  \end{equation*}
  The arrow is induced by
  $\mathcal{O}_X \ra \pi_*\mathcal{O}_{Y},$ and
  the equality is the usual adjunction.
  It corresponds under the adjunction to a morphism
  $    \pi^*\sheafHom_{\mathcal{O}_{X}}(\mathcal{F}, \mathcal{O}_X)
  \ra \sheafHom_{\mathcal{O}_{Y}}(\pi^*\mathcal{F},
  \mathcal{O}_{Y}).$
  This morphism is an isomorphism if $\mathcal{F}$ is a vector
  bundle. 
  From this we obviously obtain the isomorphism we want.
\end{proof}

\subsection{Enhancements}
\label{sec:enhancements}

In this section we define several
enhancements of $\bfMF(X,W)$
and show that they are equivalent
(see e.\,g.\ \cite{lunts-orlov-enhancement} for the
definitions). Similarly we define two equivalent enhancements of
$\DQcoh(X,W).$

\subsubsection{Enhancements using injective quasi-coherent sheaves}
\label{sec:enhanc-inject}

Recall that the obvious
functor $[\InjQcoh(X, W)] \ra \DQcoh(X,W)$ is an equivalence
(Theorem~\ref{t:equivalences-curved-categories}.\ref{enum:injQcoh-DQcoh-equiv}), in other words $\InjQcoh(X,W)$ is an
enhancement of the triangulated category $\DQcoh(X,W).$
This enhancement induces an
enhancement for the full subcategory $\bfMF(X,W)
\sira \DCoh(X,W) \subset \DQcoh(X,W)$
(cf.\ Theorem~\ref{t:equivalences-curved-categories}). Namely, let
$\InjQcoh_{\bfMF}(X,W) \subset \InjQcoh(X,W)$ be the full dg
subcategory consisting of objects which are 
isomorphic in $\DQcoh(X,W)$ to an object of $\bfMF(X,W).$ Then
\begin{equation*}
  [\InjQcoh_{\bfMF}(X,W)] \simeq \bfMF(X,W),
\end{equation*}
so $\InjQcoh_{\bfMF}(X,W)$ is an enhancement
of $\bfMF(X,W).$

\subsubsection{Enhancements by dg quotients}
\label{sec:enhanc-dg-quot}

There is a different enhancement of $\bfMF(X,W).$ Namely, let
$\AcyclMF(X,W)\subset \MF(X,W)$ be the
full dg
subcategory consisting of objects that belong to $\Acycl[\MF(X,W)].$
Consider the Drinfeld dg quotient category
$\MF(X,W)/\AcyclMF(X,W)$
(which is pretriangulated, cf.\
\cite[Lemma~1.5]{lunts-orlov-enhancement}). 
Then by \cite[Thm.~1.6.2, Thm.~3.4]{drinfeld-dg-quotients}
there is an
equivalence
\begin{equation}
  \label{eq:dg-quot-enhancement-equiv}
  \bfMF(X,W) = [\MF(X,W)]/\Acycl[\MF(X,W)] \sira [\MF(X,W)/\AcyclMF(X,W)],
\end{equation}
hence $\MF(X,W)/\AcyclMF(X,W)$ is an enhancement of $\bfMF(X,W).$
Similarly, by defining
$\AcyclCoh(X,W) \subset \Coh(X,W)$ to consist of those objects that
belong to $\Acycl[\Coh(X,W)],$  
we see that
$\Coh(X,W)/\AcyclCoh(X,W)$ is an enhancement of 
$\DCoh(X,W) \sila \bfMF(X,W).$

The same approach works for the category $\DQcoh(X,W)$: Let $\AcyclQcoh(X,W)
\subset \Qcoh(X,W)$ be the full dg subcategory consisting of objects that
belong to $\Acycl[\Qcoh(X,W)].$
Then
\begin{equation*}
  \DQcoh(X,W) = [\Qcoh(X,W)]/\Acycl[\Qcoh(X,W)] 
  \sira 
  [\Qcoh(X,W)/\AcyclQcoh(X,W)],
\end{equation*}
i.\,e.\ the dg quotient
$\Qcoh(X,W)/\AcyclQcoh(X,W)$ is an enhancement of $\DQcoh(X,W).$

The two enhancements of
$\DQcoh(X,W)$ using injectives resp.\ dg quotients 
are equivalent, and similarly for
the three enhancements of $\bfMF(X,W).$
Namely we have the following lemma.

\begin{lemma}
  \label{l:enhancements-qequi}
  \rule{1mm}{0mm}
  \begin{enumerate}
  \item
    \label{enum:Inj-qequi-Qcoh-mod-AcyclQcoh}
    The dg categories $\InjQcoh(X,W)$ and
    $\Qcoh(X,W)/\AcyclQcoh(X,W)$ are quasi-equivalent.
  \item
    \label{enum:inj-qequi-MF-mod-AcyclMF}
    The dg categories 
    $\MF(X,W)/\AcyclMF(X,W)$ and
    $\Coh(X,W)/\AcyclCoh(X,W)$ and
    $\InjQcoh_{\bfMF}(X,W)$ 
    are quasi-equivalent.
  \end{enumerate}
\end{lemma}

\begin{proof}
  \ref{enum:Inj-qequi-Qcoh-mod-AcyclQcoh}
  The Drinfeld dg quotient comes with the canonical
  quotient dg functor $\Qcoh(X,W) \ra \Qcoh(X,W)/\AcyclQcoh(X,W).$
  Restriction to the dg subcategory $\InjQcoh(X,W)$ yields the
  desired quasi-equivalence $\alpha  \colon  \InjQcoh(X,W) \ra \Qcoh(X,W)/\AcyclQcoh(X,W).$

  \ref{enum:inj-qequi-MF-mod-AcyclMF} 
  Consider the dg functor
  $\alpha \colon  \InjQcoh_{\bfMF}(X,W)\ra \Qcoh(X,W)/\AcyclQcoh(X,W)$
  obtained by restriction and the
  canonical dg functor $\beta \colon  \MF(X,W)/\AcyclMF(X,W) \ra
  \Qcoh(X,W)/\AcyclQcoh(X,W).$ The induced homotopy functors
  $[\alpha]$ and $[\beta]$ are full and faithful and have the
  same essential image in $[\Qcoh(X,W)/\AcyclQcoh(X,W)].$ Let
  $\mathcal{A} \subset \Qcoh(X,W)/\AcyclQcoh(X,W)$ be the full dg
  subcategory consisting of objects that belong to this essential
  image. 
  Then the dg functors
  \begin{equation*}
    \InjQcoh_{\bfMF}(X,W)\xra{\alpha}\mathcal{A}\xla{\beta} \MF(X,W)/\AcyclMF(X,W)
  \end{equation*}
  are the
  desired quasi-equivalences. Similarly we prove that 
  $\Coh(X,W)/\AcyclCoh(X,W)$ and
  $\InjQcoh_{\bfMF}(X,W)$ are
  quasi-equivalent. 
\end{proof}

\subsubsection{Morphism oriented \v{C}ech enhancement}
\label{sec:morphism-cech-enhancement}

After some
preparations we will define an enhancement for
$\bfMF(X,W)$ whose morphism spaces are defined using \v{C}ech
complexes. 

Let $\mathcal{U}=(U_i)_{i \in I}$ be an open covering of $X$ and let
$\mathcal{F}$ be a dg
sheaf on $X,$ i.\,e.\ an object of
$\Sh(X,0).$ 
We define a $\DZ_2 \times \DZ$-graded abelian group
$C^*(\mathcal{U}, \mathcal{F}_*)$ as follows: Its component of
degree $(p,q) \in \DZ_2 \times \DZ$ is
\begin{equation*}
  C^q(\mathcal{U}, \mathcal{F}_p)
  = \prod_{(i_0, \dots , i_q) \in I^{q+1}}
  \mathcal{F}_p(U_{i_0} \cap \dots \cap U_{i_q}).
\end{equation*}
We turn $C^*(\mathcal{U}, \mathcal{F}_*)$ into a double complex
as follows:
its first differential (in the $p$-direction) is induced by that of
$\mathcal{F}$ and its second differential is the usual \v{C}ech
differential.
The \v{C}ech complex
$C(\mathcal{U}, \mathcal{F})$ is the total complex of
$C^*(\mathcal{U}, \mathcal{F}_*)$: Its $m$-th component for $m \in
\DZ_2$ is given by
\begin{equation*}
  C(\mathcal{U}, \mathcal{F})_m
  = \bigoplus_{p \in \DZ_2,\, q \in \DZ,\, p+q=m} C^q(\mathcal{U}, \mathcal{F}_p)
\end{equation*}

There is an obvious map
\begin{equation}
  \label{eq:global-to-Cech-complex}
  \Gamma(X, \mathcal{F}) \ra C(\mathcal{U}, \mathcal{F})
\end{equation}
of dg abelian groups.

A different perspective on $C(\mathcal{U}, \mathcal{F})$
is as follows. Taking the
\v{C}ech complex defines a functor from $\Sh(X)$ to the category
of complexes of vector spaces over $k,$ and hence maps 
$\mathcal{F} \in \Sh(X,0)$ to a complex $C^*(\mathcal{U},
\mathcal{F})$ in $Z_0(\Sh(\Spec k, 0)).$ Its totalization is
$C(\mathcal{U}, \mathcal{F}).$

\begin{lemma}
  \label{l:global-to-Cech-qiso-for-flabby}
  The morphism \eqref{eq:global-to-Cech-complex} is a
  quasi-isomorphism if $\mathcal{F}$ is componentwise flabby
  (i.\,e.\ $\mathcal{F}_0$ and $\mathcal{F}_1$ are flabby).
\end{lemma}

\begin{proof}
  This follows from \cite[Thm.~5.2.3]{godement} and
  part~\ref{enum:double-complex-upper-halfplane-column-qisos}
  of the following
  Lemma~\ref{l:double-complex-upper-halfplane}.
\end{proof}

\begin{lemma}
  \label{l:double-complex-upper-halfplane}
  Let $f \colon  A \ra B$ be a morphism of
  $\DZ_2 \times \DZ$-graded double complexes
  $A=(A^{p,q})_{p \in \DZ_2,\, q \in \DZ},$
  $B=(B^{p,q})_{p \in \DZ_2,\, q \in \DZ}$
  of abelian groups.
  We assume that $A^{p,q}=0$ and $B^{p,q}=0$ for all $q<M,$ for
  some fixed $M \in \DZ.$
  Assume that one of the following two conditions is true:
  \begin{enumerate}
  \item
    \label{enum:double-complex-upper-halfplane-column-qisos}
    $f$ induces isomorphisms $H(A^{p,*}) \ra H(B^{p,*})$ for all
    $p \in \DZ_2.$
  \item
    \label{enum:double-complex-upper-halfplane-row-qisos}
    $f$ induces isomorphisms $H(A^{*,q}) \ra H(B^{*,q})$ for all
    $q \in \DZ,$ and $A$ and $B$ are 
    bounded in the
    $q$-direction, i.\,e.\ there is $N \in \DZ$ such that
    $A^{p,q}=0$ and $B^{p,q}=0$ for all $q>N$ and $p \in \DZ_2.$
  \end{enumerate}
  Then $f$ induces a quasi-isomorphism $\Tot(f) \colon  \Tot(A) \ra
  \Tot(B)$ of the total complexes
  associated to $A$ and $B.$
\end{lemma}

\begin{proof}
  In this proof we view $A$ and $B$ in the obvious way as $\DZ \times
  \DZ$-graded double complexes that are 2-periodic in the
  $p$-direction.

  Assume that \ref{enum:double-complex-upper-halfplane-column-qisos}
  holds. Let $F_n A$ be the double subcomplex of $A$ defined by
  \begin{equation*}
    (F_n A)^{p,q} =
    \begin{cases}
      A^{p,q} & \text{if $q < n,$}\\
      \Kern (A^{p,n} \ra A^{p,n+1}) & \text{if $q = n,$}\\
      0 & \text{if $q > n,$}
    \end{cases}
  \end{equation*}
  and similarly for $B.$
  Then $f$ induces maps $F_nf \colon  F_n A \ra F_n B$ for all $n \in \DN,$
  and these maps induce quasi-isomorphisms on total complexes by
  \cite[Thm.~1.9.3]{KS}. This obviously implies the claim.

  If \ref{enum:double-complex-upper-halfplane-row-qisos} is satisfied
  we can immediately apply \cite[Thm.~1.9.3]{KS}.
\end{proof}

Let us apply the above now to sheaf Hom object 
$\sheafHom(E,I)$ defined in
section~\ref{sec:sheaf-hom-tensor-product}.

\begin{lemma}
  \label{l:Cech-computes-morphisms}
  Let $E \in \Qcoh(X,W)$ and $I \in \InjQcoh(X,W).$ Then
  \begin{equation*}
    \Hom_{\Qcoh(X,W)}(E, I) = \Gamma(X, \sheafHom(E, I)) \ra
    C(\mathcal{U}, \sheafHom(E,I))
  \end{equation*}
  (cf.\ 
  \eqref{eq:Qcoh-dg-Homs-as-global-sections}
  and
  \eqref{eq:global-to-Cech-complex})
  is a quasi-isomorphism.
  In particular,
  \begin{equation*}
    \Hom_{\DQcoh(X,W)}(E, [p]I) \cong H_p(C(\mathcal{U},
    \sheafHom(E,I))) 
  \end{equation*}
  canonically, for $p \in \DZ_2.$
\end{lemma}

\begin{proof}
  Since any injective object of $\Qcoh(X)$ is also an injective
  object of $\Sh(X),$ by
  Theorem~\ref{t:injective-in-qcoh-vs-all-OX}.\ref{enum:inj-qcoh-equal-inj-OX-that-qcoh},
  $\sheafHom(E, I)$ is componentwise flabby.
  Thus
  Lemma~\ref{l:global-to-Cech-qiso-for-flabby} shows that the
  first map is a quasi-isomorphism,
  and then
  Remark~\ref{rem:morphisms-to-injectives} proves the
  second claim.
\end{proof}

\begin{lemma}
  \label{l:qiso-induces-Cech-qiso}
  Let $\mathcal{F} \ra \mathcal{G}$ be a quasi-isomorphism in
  $Z_0(\Qcoh(X,0)).$
  If $\mathcal{U}=(U_i)_{i \in I}$ is an affine open covering of
  $X,$ then $C(\mathcal{U}, \mathcal{F}) \ra C(\mathcal{U},
  \mathcal{G})$ is a quasi-isomorphism.
\end{lemma}

\begin{proof}
  Since $X$ is quasi-compact there is a finite subset $I'
  \subset I$ such that $\mathcal{U}':=(U_i)_{i \in I'}$ is a
  covering of $X.$
  If $\mathcal{A}$ is any quasi-coherent sheaf on $X,$ the \v{C}ech
  cohomologies $H(\mathcal{U}, \mathcal{A})$
  and $H(\mathcal{U}', \mathcal{A})$ are canonically isomorphic
  to $H(X, \mathcal{A}),$ since our coverings are by affine open
  subsets.
  This together with
  part~\ref{enum:double-complex-upper-halfplane-column-qisos}
  of Lemma~\ref{l:double-complex-upper-halfplane} shows that
  $C(\mathcal{U}', \mathcal{F}) \ra C(\mathcal{U}, \mathcal{F})$ is an
  isomorphism.

  The usual \v{C}ech complex of a sheaf contains the alternating
  subcomplex and its inclusion is a homotopy equivalence. Similarly,
  the \v{C}ech complex $C(\mathcal{U}', \mathcal{F})$ has a homotopy
  equivalent subcomplex $C_{\alt}(\mathcal{U}', \mathcal{F}).$

  These arguments show that it is sufficient to show that
  $C_{\alt}(\mathcal{U}', \mathcal{F}) \ra
  C_{\alt}(\mathcal{U}', \mathcal{G})$ is a quasi-isomorphism.
  This follows from part~\ref{enum:double-complex-upper-halfplane-row-qisos}
  of Lemma~\ref{l:double-complex-upper-halfplane}: any finite
  intersection $U'$ of elements of $\mathcal{U}'$ is affine, and hence
  $\mathcal{F}(U') \ra \mathcal{G}(U')$ is a quasi-isomorphism by
  assumption.
\end{proof}

\begin{corollary}
  \label{c:P-to-closed-degree-zero-iso-in-DQcoh-qiso-in-Cech}
  Let $E \ra F$ be a morphism in $Z_0(\Qcoh(X,W))$ that becomes an
  isomorphism in $\DQcoh(X,W),$ let $P \in \MF(X,W),$
  and let $\mathcal{U}$ be an affine open covering of $X.$ 
  Then
  \begin{equation*}
    C(\mathcal{U}, \sheafHom(P,E)) \ra C(\mathcal{U},
    \sheafHom(P,F)) 
  \end{equation*}
  is a quasi-isomorphism.
\end{corollary}

\begin{proof}
  The morphism 
  $\sheafHom(P,E) \ra \sheafHom(P,F)$
  in $Z_0(\Qcoh(X,0))$ 
  becomes an isomorphism in $\DQcoh(X,0),$ 
  cf.\ section~\ref{sec:sheaf-hom-tensor-product}.
  Hence it is a quasi-isomorphism by
  Proposition~\ref{p:case-W-equals-zero-acycl-equals-exact}. 
  Now use Lemma~\ref{l:qiso-induces-Cech-qiso}.
\end{proof}

We fix an affine open covering $\mathcal{U}=(U_i)_{i \in I}$ of
$X$ for defining the morphism oriented \v{C}ech enhancement.
We define a dg category $\MF_{\Cechmor}(X,W)$ as follows
(it depends on the affine open covering $\mathcal{U}=(U_i)_{i \in I}$
but we suppress this in the notation). The objects of $\MF_{\Cechmor}(X,W)$
coincide with the objects of $\MF(X,W),$ and the morphisms are given
by
\begin{equation*}
  \Hom_{\MF_{\Cechmor}(X,W)}(P,Q) := C(\mathcal{U}, \sheafHom(P,Q)).
\end{equation*}
The composition in this category is defined using the cup-product
for \v{C}ech complexes as defined
in \cite{stacks-project},
chapter 18, section 19 "\v{C}ech 
cohomology of complexes"
(adapted to our differential $\DZ_2$-graded situation in the
obvious way).

We can repeat this construction
starting with any dg subcategory $\mathcal{C} \subset \Qcoh(X,W)$ to obtain
the corresponding dg category $\mathcal{C}_{\Cechmor}.$
We always have an obvious dg
functor $\mathcal{C} \ra \mathcal{C}_{\Cechmor}$
obtained from
\eqref{eq:Qcoh-dg-Homs-as-global-sections}
and
\eqref{eq:global-to-Cech-complex}
and the induced functor
$[\mathcal{C}] \ra [\mathcal{C}_{\Cechmor}]$
on homotopy categories.

\begin{proposition}
  \label{p:Cech-enhancement}
  The dg categories $\InjQcoh_{\bfMF}(X,W)$ and $\MF_{\Cechmor}(X,W)$ are
  quasi-equivalent, i.\,e.\ connected by a zig-zag of
  quasi-equivalences (explicitly constructed in the proof).

  Moreover, 
  $\MF_{\Cechmor}(X,W)$ is a pretriangulated dg category, and
  the functor $[\MF(X,W)] \ra [\MF_{\Cechmor}(X,W)]$
  factors
  through
  the Verdier localization
  $[\MF(X,W)] \ra \bfMF(X,W)$ to an equivalence
  \begin{equation*}
    \bfMF(X,W) \sira [\MF_{\Cechmor}(X,W)]
  \end{equation*}
  of triangulated categories. This shows that $\MF_{\Cechmor}(X,W)$ is a
  dg enhancement of 
  $\bfMF(X,W)$ naturally.
  We call it the \define{morphism oriented \v{C}ech enhancement}
  of $\bfMF(X,W).$ 
\end{proposition}

In particular this shows that
the enhancements $\InjQcoh_{\bfMF}(X,W)$ and
$\MF_{\Cechmor}(X,W)$ of $\bfMF(X, W)$ are equivalent.

\begin{proof}
  We construct the zig-zag of quasi-equivalences first.
  To ease the notation we abbreviate 
  $\mathcal{C}:=\InjQcoh_{\bfMF}(X,W).$
  We use the auxiliary dg category
  $\mathcal{C}_{\Cechmor}$
  with the
  dg functor $\gamma \colon  \mathcal{C} \ra \mathcal{C}_{\Cechmor}$
  as explained above. 
  Lemma~\ref{l:Cech-computes-morphisms} shows that $\gamma$ induces
  quasi-isomorphisms on morphism spaces. It is bijective on objects
  and hence a quasi-equivalence.

  It remains to prove that the dg categories $\MF_{\Cechmor}(X,W)$
  and $\mathcal{C}_{\Cechmor}$ are quasi-equivalent. For this
  we define a new dg category $\mathcal{B}$ and two dg functors
  \begin{equation*}
    \MF_{\Cechmor}(X,W) \xla{p} \mathcal{B} \xra{q}
    \mathcal{C}_{\Cechmor} 
  \end{equation*}
  which are quasi-equivalences.

  By definition the objects of $\mathcal{B}$ are triples
  $(P,I,\delta),$ where $P \in \MF(X,W),$ $I\in \mathcal{C}$ and
  $\delta \colon  P\ra I$ is a morphism in $Z_0(\Qcoh(X,W))$ which
  becomes an isomorphism in $\DQcoh(X,W).$ Given objects $(P, I,
  \delta)$ and $(Q, J, \epsilon),$ the dg module
  $\Hom_{\mathcal{B}} ((P, I, \delta), (Q, J, \epsilon))$ can be
  conveniently written in matrix form
  \begin{equation*}
    \begin{bmatrix}
      (I,J) & [-1](P,J)\\
      0 & (P,Q)
    \end{bmatrix}
  \end{equation*}
  where $(-,-)= \Hom_{\Qcoh(X,W)_{\Cechmor}}(-,-).$
  The differential is
  defined by
  \begin{equation*}
    d \colon 
    \begin{bmatrix}
      r & m\\
      0 & l
    \end{bmatrix}
    \mapsto
    \begin{bmatrix}
      dr & \epsilon l -r\delta +d_{[-1](P,J)}m\\
      0 & dl
    \end{bmatrix}
    =
    \begin{bmatrix}
      dr & \epsilon l -r\delta -d_{(P,J)}m\\
      0 & dl
    \end{bmatrix},
  \end{equation*}
  and composition by
  \begin{equation*}
    \begin{bmatrix}
      \rho & \mu\\
      0 & \lambda
    \end{bmatrix}
    \comp
    \begin{bmatrix}
      r & m\\
      0 & l
    \end{bmatrix}
    =
    \begin{bmatrix}
      \rho r & \rho.m+\mu.l\\
      0 & \lambda l
    \end{bmatrix}
    =
    \begin{bmatrix}
      \rho r & (-1)^{|\rho|}\rho m+\mu l\\
      0 & \lambda l
    \end{bmatrix}
  \end{equation*}
  where $m$ is considered as an element of $[-1](P,J)$ in the middle term
  and as an element of $(P,J)$ in the right term, and similarly for
  $\mu.$

  The obvious projections
  $\MF_{\Cechmor}(X,W) \xla{p} \mathcal{B} \xra{q} \mathcal{C}_{\Cechmor}$ are dg functors.
  These functors are
  surjective on objects
  (use Theorem~\ref{t:equivalences-curved-categories} and
  Remark~\ref{rem:morphisms-to-injectives}). Hence in order to show
  that they are quasi-equivalences we need to see that they induces
  quasi-isomorphisms on morphism spaces.

  Let us prove this for $p$ first. The map $\delta \colon  P \ra I$ yields
  a closed degree zero morphism $\delta^* \colon  (I,J) \ra (P,J)$ in the dg
  category of dg modules. The shift of its cone $\Cone(\delta^*)$
  is the kernel of the map
  $p \colon \Hom_\mathcal{B}((P, I, \delta), (Q, J, \epsilon)) \ra
  \Hom_{\MF_{\Cechmor}}(P,Q).$
  Hence it is sufficient to
  show that $\Cone(\delta^*)$
  is acyclic. Equivalently we show that $\delta^*$ is a
  quasi-isomorphism. But this is true by
  Lemma~\ref{l:Cech-computes-morphisms} and
  Remark~\ref{rem:morphisms-to-injectives} and our assumption that
  $\delta$ is an isomorphism in $\DQcoh(X,W).$

  Similarly, when considering $q,$ we have to show that
  $\epsilon_*  \colon (P,Q) \ra (P,J)$ is a quasi-isomorphism. But this
  follows from
  Corollary~\ref{c:P-to-closed-degree-zero-iso-in-DQcoh-qiso-in-Cech}.
  This shows that $p$ and $q$ are quasi-equivalences, and finishes the
  proof of the first statement.

  Our zig-zag of quasi-equivalences yields the equivalences
  \begin{equation*}
    [\mathcal{C}] \xra{[\gamma]} [\mathcal{C}_{\Cechmor}] \xla{[q]}
    [\mathcal{B}] \xra{[p]} [\MF_{\Cechmor}(X,W)]
  \end{equation*}
  on the level of homotopy categories.  This shows that
  $\MF_{\Cechmor}(X,W)$ is pretriangulated.  Moreover, if we fix for
  any $P \in \MF(X, W)$ an object $(P, I_P, \delta_P)$ of
  $\mathcal{B},$ this implies that $P \mapsto I_P$ is an
  equivalence $[\MF_{\Cechmor}(X,W)] \ra [\mathcal{C}].$

  On the other hand $\bfMF(X,W) \ra [\mathcal{C}],$ $P \mapsto
  I_P,$ is also an equivalence. These two equivalences and the
  obvious functors fit into the commutative
  diagram
  \begin{equation*}
    \xymatrix{
      {[\MF(X,W)]} \ar[r] \ar[d] & {[\MF_{\Cechmor}(X,W)]}
      \ar[d]^-\sim \\ 
      {\bfMF(X,W)} \ar[r]^-\sim & {[\mathcal{C}]}
    }
  \end{equation*}
  (commutativity is obvious for objects; for morphisms
  go through the above equivalences)
  which shows that the upper horizontal functor vanishes on
  $\Acycl[\MF(X,W)].$ We obtain an induced functor
  $\bfMF(X,W) \ra [\MF_{\Cechmor}(X,W)]$ of triangulated categories
  which is then obviously an equivalence. 
\end{proof}

\begin{corollary}
  \label{c:MF-Cech-indep-up-to-quasi-equi-on-covering}
  The category $\MF_{\Cechmor}(X,W)$ does not depend
  (up to quasi-equivalence) on the choice of the affine open
  covering $\mathcal{U}=(U_i)_{i \in I}$ of $X.$
\end{corollary}

\subsubsection{Object oriented \v{C}ech enhancement}
\label{sec:object-cech-enhancement}

In \cite{valery-olaf-matfak-motmeas-in-prep}
we will introduce another equivalent enhancement
$\MF_\Cechobj(X,W)$ of $\bfMF(X,W)$ 
whose objects are defined using \v{C}ech resolutions.

\subsubsection{Enhancement for affine \texorpdfstring{$X$}{X}}
\label{sec:enhanc-affine-X}

If $X$ is affine 
Lemma~\ref{l:affine-MF}.\ref{enum:affine-MF-equals-boldMF}
says that $\MF(X,W)$ is an enhancement of $\bfMF(X,W).$
In fact this enhancement is a special case of the object oriented
\v{C}ech enhancement (for the trivial affine open covering $\{X\}$
of $X$). It is equivalent to the   
enhancement $\InjQcoh_{\bfMF}(X,W)$ (use the method
of proof of Proposition~\ref{p:Cech-enhancement}).

\subsection{Compact generators}
\label{sec:compact-generators}

Recall that the category $\DQcoh(X,W)$ is cocomplete
(Corollary~\ref{c:direct-sums}). 

\begin{proposition}
  \label{p:compact-objects-and-generators-DQcoh}
  \rule{1mm}{0mm}
  \begin{enumerate}
  \item
    \label{enum:bfMFs-compact}
    The objects of $\bfMF(X,W)$ are compact in
    $\DQcoh(X,W).$
  \item
    \label{enum:DQcoh-gen-by-bfMF}
    The triangulated category $\DQcoh(X,W)$ is generated by the
    objects of $\bfMF(X,W).$
  \item
    \label{enum:Karoubi-closure-of-bfMF-in-DQcoh-equals-compacts}
    The subcategory $\DQcoh(X,W)^c$ of compact objects in
    $\DQcoh(X,W)$ is a Karoubi envelope
    of $\bfMF(X,W).$
    We denote this Karoubi envelope by $\ol{\bfMF(X,W)}.$
  \end{enumerate}
\end{proposition}

\begin{proof}
  Results of Neeman \cite{neeman-connection-TTYBR}
  imply  
  \cite[Thm.~2.1.2 (and Prop~2.1.1)]{bondal-vdbergh-generators}.
  In particular
  assertions \ref{enum:bfMFs-compact} and
  \ref{enum:DQcoh-gen-by-bfMF} 
  imply \ref{enum:Karoubi-closure-of-bfMF-in-DQcoh-equals-compacts}.

  \ref{enum:bfMFs-compact}:
  Follows from 
  Theorem~\ref{t:equivalences-curved-categories}.\ref{enum:injQcoh-DQcoh-equiv},
  Remark~\ref{rem:morphisms-to-injectives},
  and
  Corollary~\ref{c:direct-sums}. Use
  \cite[Exercise~II.1.11]{Hart}.

  \ref{enum:DQcoh-gen-by-bfMF}:
  We essentially copy the proof of
  \cite[3.11~Thm.~2]{positselski-two-kinds}.

  Assume that
  $J\in \InjQcoh(X,W)$ is such that every morphism $E \ra J$ in
  $Z_0(\Qcoh(X,W))$ with $E \in \Coh(X,W)$ is homotopic to zero.
  By Theorem~\ref{t:equivalences-curved-categories} and
  Remark~\ref{rem:morphisms-to-injectives} it suffices to prove that
  $J = 0$ in $[\InjQcoh(X,W)].$

  Apply Zorn's lemma to the ordered set of pairs $(M,h),$ where
  $M$ is a subobject of $J$ and $h \colon M \ra J$ is a
  contracting homotopy of the embedding $\iota \colon M \hra J,$
  i.\,e.\ $d(h)=\iota.$ It suffices to check that given $(M,h)$
  with $M \subsetneq J$ there exists $M \subsetneq M' \subset J$
  and a contracting homotopy $h' \colon M' \ra J$ for the
  embedding $M' \hra J$ such that $h'|_M=h.$ Let $M' \subset J$
  be a subobject such that $M \subsetneq M'$ and $M'/M \in
  \Coh(X,W)$ (use \cite[Ex.~II.5.15.(e)]{Hart} and the first step
  in the proof of
  Lemma~\ref{l:factoring-through-coherent}). Since $J$ has
  injective quasi-coherent components, the degree one morphism $h
  \colon M \ra J$ can be extended to a degree one morphism $h''
  \colon M' \ra J.$ Let $\iota \colon M \ra J$ and $\iota' \colon
  M' \ra J$ denote the embeddings. The map $\iota'- d(h'')$ is a
  closed degree zero morphism and vanishes on $M,$ so it induces
  a morphism $g \colon M'/M \ra J$ in $Z_0(\Qcoh(X,W)).$ By our
  assumption, there exists a contracting homotopy $c \colon M'
  /M\ra J$ for $g.$ Denote the composition $M' \sra M'/M \xra{c}
  J$ also by $c.$ Then $h'=h''+c \colon M' \ra J$ is a
  contracting homotopy for $\iota'$ extending $h.$
\end{proof}

\begin{proposition}
  \label{p:existence-of-classical-generator-in-mf}
  Assume in addition that $X$ is of finite type over $k.$
  Then the triangulated category $\bfMF(X,W)$ has a
  classical generator.
  Hence so does the category $\ol{\bfMF(X,W)}.$
\end{proposition}

\begin{proof}
  By Remark~\ref{rem:X-disconnected}
  we may assume that $X$ is connected. Then we distinguish two cases:
  the map $W \colon X \ra \DA^1$ is flat or else $W=0.$ The remaining case
  of a constant nonzero $W$ is trivial since then
  $\bfMF(X,W)=0$ by Lemma~\ref{l:case-W=constant-nonzero}.

  Assume that $W \colon X\ra \DA^1$ is flat. Then by Theorem~\ref{t:factorizations=singularity} $\bfMF(X,W)\simeq D_\Sg(X_0).$
  It is well-known
  that the triangulated category $D^b(\Coh(X_0))$ has a
  classical generator (the proof of this fact in
  \cite[6.3.(a)]{lunts-categorical-resolution} also
  works if $k$
  is not perfect). Hence also the quotient category
  $D_\Sg(X_0)=D^b(\Coh(X_0))/\mfPerf (X_0)$ has a classical
  generator.

  Assume now that $W=0.$ In this case we will use the equivalence
  $\bfMF(X,0) \sira \DCoh(X,0)$
  from Theorem~\ref{t:equivalences-curved-categories}
  and the description
  $\DCoh(X,0) = [\Coh(X,0)]/\Ex[\Coh(X,0)]$
  from Proposition~\ref{p:case-W-equals-zero-acycl-equals-exact}.
  Consider the usual bounded
  derived category
  $D^b(\Coh(X))$ of  coherent sheaves on $X.$ We have the obvious
  triangulated folding functor
  $D^b(\Coh(X)) \ra \DCoh(X,0)$ which takes a $\DZ$-graded complex of
  coherent sheaves to
  the corresponding $\DZ_2$-graded one. Since the category
  $D^b(\Coh(X))$ has a classical generator it suffices to show that
  $\DCoh(X,0)$ is the triangulated envelope of the collection of
  objects
  which are in the image
  of the folding functor.

  For every $E \in \Coh(X,0)$ we have a short exact sequence
  \begin{equation*}
    (\matfak{\im e_0}{0}{\im e_1}{0}) \hra 
    E \sra 
    (\matfak{E_1/\im e_0}{0}{E_0/\im e_1}{0})
  \end{equation*}
  in $Z_0(\Coh(X,0))$
  and hence a triangle in 
  $\DCoh(X,0),$ by 
  Lemma~\ref{l:totalization}.\ref{enum:ses-gives-triangle}.
  But it is obvious that any object in $\DCoh(X,0)$
  with zero differential is in the image of the folding
  functor.
\end{proof}

The folding functor appearing in the above proof will be studied
in detail in \cite{olaf-folding-derived-categories-in-prep}.

\begin{remark}
  \label{rem:classical-generator-in-mf-zero}
  The above proof shows that the folding of a classical
  generator $G$ of $D^b(\Coh(X))$ is a classical generator of
  $\DCoh(X,0).$ By replacing $G$ by the direct sum of its
  cohomologies one can assume that $G \in \Coh(X).$ Then $G$ has
  a finite resolution by vector bundles, and by replacing $G$ by
  the direct sum of the involved vector bundles we can assume
  that $G$ itself is a vector bundle. Then the folding of
  $G$ has the form $(\matfak{0}{}{G}{}) \in \bfMF(X,0)$
  and is a classical generator of $\bfMF(X,0).$
\end{remark}

\subsection{Some useful results}
\label{sec:some-useful-results}

\begin{lemma}
  \label{l:orthog-mf}
  Let $E, F \in \Qcoh(X,W)$ and assume that
  $\Hom_{D(\Qcoh(X))}(E_p, [i]F_{p'})=0$
  for all $p,$ $p' \in \DZ_2$ and $i \in \DZ.$ Then
  $\Hom_{\DQcoh(X,W)}(E, [q]F)=0$ for all $q \in \DZ_2.$
\end{lemma}

\begin{proof}
  Let $F \ra I$ be as in
  Lemma~\ref{l:resolutions}.\ref{enum:Inj-reso}. Then the
  isomorphism $F \sira \Tot(I)$ in $\DQcoh(X,W)$ and
  Remark~\ref{rem:morphisms-to-injectives} imply that we obtain
  isomorphisms
  \begin{equation*}
    \Hom_{\DQcoh(X,W)}(E, [q]F) \sira
    \Hom_{\DQcoh(X,W)}(E, [q]\Tot(I))
    \sila
    H_q(\Hom_{\Qcoh(X,W)}(E, \Tot(I))).
  \end{equation*}
  Hence we need to show that dg module
  $\Hom_{\Qcoh(X,W)}(E, \Tot(I))$ is acyclic.
  This dg module is the totalization of the (finite) complex
  \begin{equation*}
    0 \ra \Hom_{\Qcoh(X,W)}(E, I^0)
    \ra \Hom_{\Qcoh(X,W)}(E, I^1)
    \ra \Hom_{\Qcoh(X,W)}(E, I^2)
    \ra \dots.
  \end{equation*}
  This complex is exact by assumption
  since $F_0 \ra I_0$ and $F_1 \ra I_1$ are (finite) injective
  resolutions in the abelian category $\Qcoh(X).$
  Hence $\Hom_{\Qcoh(X,W)}(E, \Tot(I))$ is acyclic
  by 
  Lemma~\ref{l:double-complex-upper-halfplane}.\ref{enum:double-complex-upper-halfplane-column-qisos}.
\end{proof}

\begin{lemma}
  [{\cite[Rem.~1.3]{positselski-coh-analogues-matrix-fact-sing-cats}}]
  \label{l:locality-of-being-zero}
  Let $\mathcal{U}$ be an open covering of $X$ and let
  $E$ be an object of $\DQcoh(X,W).$
  Assume that $E|_U=0$
  in $\DQcoh(X,W)$
  for all $U \in \mathcal{U}.$ 
  Then $E=0$ in $\DQcoh(X,W).$
\end{lemma}

\begin{remark}
  \label{rem:locality-of-being-zero-via-orlov}
  The corresponding result for $E \in \bfMF(X,W)$ 
  can also be shown using
  Remark~\ref{rem:X-disconnected},
  Lemma~\ref{l:case-W=constant-nonzero},
  Proposition~\ref{p:case-W-equals-zero-acycl-equals-exact},
  and 
  Theorem~\ref{t:factorizations=singularity} 
  (being a perfect complex is defined locally).
\end{remark}

\begin{proof}
  We repeat the proof of
  \cite[Rem.~1.3]{positselski-coh-analogues-matrix-fact-sing-cats}.
  We can assume that $\mathcal{U}$ is finite and consists of
  affine open subsets. For $V \subset X$ open let $j_V \colon  V \ra X$
  be the inclusion.
  The \v{C}ech resolution
  \begin{equation*}
    0 \ra E 
    \ra \bigoplus_{U_0 \in \mathcal{U}} j_{U_0*} j^*_{U_0} E
    \ra \bigoplus_{U_0, U_1 \in \mathcal{U}} j_{U_0 \cap U_1*}
    j^*_{U_0 \cap U_1} E \ra \dots
  \end{equation*}
  is a bounded exact complex in $Z_0(\Qcoh(X,W)).$
  For any finite intersection $V$ of (a positive number of)
  elements of 
  $\mathcal{U}$ we 
  have $j_V^*(E) \in \Acycl[\Qcoh(V,W)]$ by assumption.
  Since $X$ is separated, $j_V$ is affine and hence
  $j_{V*}j_V^*(E) \in \Acycl[\Qcoh(X,W)]$
  by Remark~\ref{rem:derived-direct-image-for-affine-morphism}.
  Lemma~\ref{l:totalization}.\ref{enum:totalization-bounded-exact-complex}-\ref{enum:totalization-of-finite-complex-components-null-homotopic}
  then shows that $E \in \Acycl[\Qcoh(X,W)].$
\end{proof}

\begin{corollary}
  \label{c:locality-of-being-an-isomorphism}
  Let $f \colon  E \ra E'$ be a morphism in $\DQcoh(X,W).$
  Assume that $f|_U \colon E|_U \ra E'|_U$ is an isomorphism for all
  elements $U$ of an open covering of $X.$ Then $f$ is an
  isomorphism.
\end{corollary}

\begin{proof}
  A morphism in a triangulated category is an isomorphism if and
  only if its cone is zero. In our case, this can be checked
  locally by Lemma~\ref{l:locality-of-being-zero}.
\end{proof}

\begin{corollary}
  \label{c:acycl-vs-locally-contractible}
  An object $E$ in $[\MF(X,W)]$ belongs to
  $\Acycl[\MF(X,W)]$
  if and only if $E$ is locally contractible, i.\,e.\ any point
  of $X$ has an open neighborhood $U$ such that $E=0$ in
  $[\MF(U,W)].$ 
\end{corollary}

\begin{proof}
  If $E$ is locally contractible then 
  $E=0$ in $\bfMF(X,W)$
  by Lemma~\ref{l:locality-of-being-zero}, hence  
  $E \in \Acycl[\MF(X,W)].$

  Conversely, let
  $E \in \Acycl[\MF(X,W)].$
  Let $U \subset X$ be any affine open subscheme.
  Then $E|_U=0$ in $\bfMF(U,W).$ But $[\MF(U,W)] \sira \bfMF(U,W)$
  by Lemma~\ref{l:affine-MF}, so $E|_U$ is contractible.
\end{proof}

\begin{proposition}
  [Locality of orthogonality]
  \label{p:locality-of-orthogonality}
  Let $\mathcal{U}$ be an open covering
  of $X$ and let
  $A,$ $B \in \Qcoh(X,W).$
  Assume that $\Hom_{\DQcoh(U',W)}(A|_{U'}, [p]B|_{U'})=0$ for all
  finite intersections $U'$ of elements of $\mathcal{U}$
  and all $p \in \DZ_2.$
  Then
  $\Hom_{\DQcoh(X,W)}(A, [p]B)=0$ for all $p \in \DZ_2.$
\end{proposition}

\begin{proof}
  Lemma~\ref{l:resolutions}.\ref{enum:Inj-reso}
  allows us to assume that $B \in \InjQcoh(X,W).$
  Then Lemma~\ref{l:Cech-computes-morphisms} shows that it is
  enough to prove that $C(\mathcal{U}, \sheafHom(A,B))$ is
  acyclic.
  Since $X$ is quasi-compact we can assume that $\mathcal{U}$ is
  finite. We order the elements of $\mathcal{U},$ say
  $\mathcal{U}=\{U_1, \dots, U_n\}.$ 

  As in the proof of
  Lemma~\ref{l:qiso-induces-Cech-qiso}
  it is enough to show that
  $C_\alt(\mathcal{U}, \sheafHom(A,B))$ is
  acyclic. Instead of the alternating \v{C}ech complex we can
  work with the isomorphic ordered \v{C}ech complex
  $C_\ord(\mathcal{U}, \sheafHom(A,B))$ (defined in the obvious
  way).

  In order to apply 
  Lemma~\ref{l:double-complex-upper-halfplane}.\ref{enum:double-complex-upper-halfplane-row-qisos}
  it is enough to show the following: for all $q \in \DN$ and 
  $1 \leq i_0 < i_1 < \dots < i_q \leq n$ the dg module
  $\sheafHom(A,B)(U')$ is acyclic, where $U':= U_{i_0} \cap \dots
  \cap U_{i_q}.$
  But
  \begin{equation*}
    \sheafHom(A,B)(U') =
    \Gamma(U'; \sheafHom(A|_{U'},B|_{U'}))
    = \Hom_{\Qcoh(U',W)}(A|_{U'},B|_{U'})
  \end{equation*}
  by \eqref{eq:Qcoh-dg-Homs-as-global-sections},
  and the latter is acyclic by
  Theorem~\ref{t:injective-in-qcoh-vs-all-OX}.\ref{enum:inj-qcoh-restrict-to-inj-qcoh},
  Remark~\ref{rem:morphisms-to-injectives},
  and our assumptions. 
\end{proof}

\begin{proposition}
  \label{p:support-mf}
  Let $X$ and $W \colon X \ra \DA^1$ be as before.
  Let $Z$ be a closed subscheme of $X$
  defined by a sheaf of ideals $\mathcal{I} \subset
  \mathcal{O}_X,$ and let $U=X-Z$ be its open complement. Let $M\in
  \bfMF(X,W)$ be such that $M|_U=0$ in
  $\bfMF(U,W).$ Then, for every $n \gg 0,$ the canonical morphism
  $p \colon  M\ra M/\mathcal{I}^nM$ has a left inverse $l$ in
  $\DCoh(X,W),$ i.\,e.\ the composition $l \comp p  \colon M \ra M$ is
  the identity of $M.$ In particular, $M$ is isomorphic to a
  direct summand of $M/\mathcal{I}^nM$
  in $\DCoh(X,W).$
\end{proposition}

\begin{proof}
  Let $M \ra I$ be a morphism in $Z_0(\Qcoh(X,W))$ with 
  $I \in \InjQcoh(X,W)$ that becomes an isomorphism
  in $\DQcoh(X,W)$
  (Lemma~\ref{l:resolutions}.\ref{enum:Inj-reso}).

  We recall some results from
  \cite[II.\S7, cf.\ proof of
  Thm.~7.18]{hartshorne-residues-duality} 
  (see also
  Theorem~\ref{t:injective-in-qcoh-vs-all-OX}).
  Any injective quasi-coherent sheaf on $X$ is isomorphic to a
  direct sum of indecomposable injective quasi-coherent sheaves.
  Every indecomposable injective quasi-coherent sheaf 
  is isomorphic to some
  $J(x):=i_{x*}(I(x)),$ where $x \in X$ is a point, $i_x \colon \Spec
  \mathcal{O}_{X,x} \ra X$ is 
  the natural inclusion and $I(x)$ is the injective hull of the
  $\mathcal{O}_{X,x}$-module $k(x).$

  If a nonzero morphism $J(x)\ra J(y)$
  exists, then $y\in \ol{\{x\}}:$ use that $J(x)$ considered as a
  sheaf of abelian groups is the skyscraper sheaf at $x$ with
  stalk $I(x);$ this follows from
  \cite[Prop.~7.5]{hartshorne-residues-duality}.

  In particular, the components of $I$ are direct sums of
  indecomposable quasi-coherent sheaves.  Denote by $I_Z \subset
  I$ the graded subsheaf consisting of all summands $J(z),$ for
  $z \in Z.$ Then $I_Z$ is in fact a subobject, i.\,e. $I_Z\in
  \InjQcoh(X,W).$ Let $\epsilon  \colon U \ra X$ denote the
  inclusion. 
  The object $\epsilon^*I$ is in $[\InjQcoh(U,W)]$ by
  Theorem~\ref{t:injective-in-qcoh-vs-all-OX}.\ref{enum:inj-qcoh-restrict-to-inj-qcoh},
  and becomes zero in $\DQcoh(U,W)$ by assumption.
  By 
  Theorem~\ref{t:equivalences-curved-categories}.\ref{enum:injQcoh-DQcoh-equiv}
  $\epsilon^*I=0$ in $[\InjQcoh(U,W)],$ i.\,e.\ $\epsilon^*I$ 
  is contractible. Hence the object $\epsilon_*\epsilon^*I \in
  [\InjQcoh(X,W)]$ is also contractible. 
  It is easy to check (use that $\epsilon_*$ preserves coproducts)
  that the 
  sequence
  \begin{equation*}
    0\ra I_Z\ra I\ra \epsilon_*\epsilon^*I\ra 0
  \end{equation*}
  in $Z_0(\InjQcoh(X,W))$ is short exact.
  Hence $I_Z \ra I$ is an isomorphism in $[\InjQcoh(X,W)].$
  Let $I \ra I_Z$ in $Z_0(\InjQcoh(X,W))$ represent an inverse.
  Thus the composition $\alpha \colon M\ra I\ra I_Z$
  becomes an isomorphism in $\DQcoh(X,W).$ 
  Since the components of $M$ are
  coherent sheaves and every local section of $I_Z$ has support in
  $Z,$ 
  by \cite[Prop.~7.5]{hartshorne-residues-duality},
  it follows that for some $n_0 \gg 0$ the morphism $\alpha$
  factors as
  \begin{equation*}
    M\ra M/\mathcal{I}^{n_0}M \xra{\beta} I_Z
  \end{equation*}
  in $Z_0(\Qcoh(X,W)).$ But then, 
  in $\DQcoh(X,W),$
  the composition
  $\alpha^{-1}\circ \beta  \colon M/\mathcal{I}^{n_0}M\ra M$ is the
  splitting of the projection $M\ra M/\mathcal{I}^{n_0}.$
  Similarly one obtains a splitting of the projection $M\ra
  M/\mathcal{I}^n M$ for any $n>n_0.$
  For the last statement fit $p \colon  M \ra M/\mathcal{I}^nM$ into a
  triangle in $\DCoh(X,W)$ and note that its third morphism is
  zero. 
\end{proof}

\section{Semi-orthogonal decompositions for matrix factorizations
  arising from projective space bundles and blowing-ups}
\label{sec:semi-orth-decomp}

There are well-known semi-orthogonal decomposition theorems for
bounded derived categories of coherent sheaves on
projective space bundles and blowing-ups. 
We recall them and then
state and prove the corresponding results for categories of
matrix factorizations. 
For the definitions of an admissible subcategory and of a
semi-orthogonal decomposition we refer to Appendix~\ref{sec:app:remind-admiss-subc}.

\subsection{Projective space bundles}
\label{sec:proj-space-bundl}

Let $Y$ be a scheme satisfying condition~\ref{enum:srNfKd},
and let
$\mathcal{N}$ be a locally free coherent sheaf on $Y$ of rank
$r.$ Let
$E:=\DP(\mathcal{N})$ be the associated projective space
bundle. It comes with a projection morphism $p \colon  E
\ra Y$ and an invertible sheaf $\mathcal{O}(1)=\mathcal{O}_E(1).$
Recall the following semi-orthogonal decomposition
theorem from\footnote{
  The assumption there is that $Y$
  is a smooth projective variety over a field.
}
\cite{orlov-monoidal,bondal-orlov-semi-orthogonal-only-arXiv},
\cite[Cor.~8.36]{huy-fourmukai}.

\begin{theorem}
  \label{t:basic-semi-orthog-1}
  Assume that $r\geq 1.$ Let $l \in \DZ.$
  \begin{enumerate}[label=$(\text{Coh\arabic*})_E$]
  \item
    \label{enum:a:basic-semi-orthog-1}
    The functor $\mathcal{O} (l)\otimes
    p^*(-) \colon D^b(\Coh(Y))\ra D^b(\Coh(E))$ is full and faithful.
  \end{enumerate}
  We denote the essential image of this
  functor by $\mathcal{O} (l) \otimes p^*D^b(\Coh(Y)).$  
  \begin{enumerate}[resume*]
  \item 
    \label{enum:b:basic-semi-orthog-1}
    The subcategory $\mathcal{O} (l)\otimes
    p^*D^b(\Coh(Y))\subset D^b(\Coh(E))$ is admissible.
  \item
    \label{enum:c:basic-semi-orthog-1}
    The category $D^b(\Coh(E))$ has the 
    semi-orthogonal
    decomposition
    \begin{equation*}
      D^b(\Coh(E))=
      \Big\langle \mathcal{O} (-r+1) \otimes
      p^*D^b(\Coh(Y)), \dots, 
      \mathcal{O} (-1)\otimes p^*D^b(\Coh(Y)),p^*D^b(\Coh(Y))
      \Big\rangle.
    \end{equation*}
  \end{enumerate}
\end{theorem}
 
Now let $W \colon Y \ra \DA^1$ be a morphism. We denote the composition
$E \xra{p} Y \xra{W} \DA^1$ also by $W.$
We have $\bL p^*=p^* \colon \bfMF(Y,W) \ra \bfMF(E,W)$
(see Theorem~\ref{t:derived-inverse-and-direct-image}.\ref{enum:inverse-image-derived}),  
and tensoring
with the line bundle $\mathcal{O}(l)$ induces 
autoequivalences of the category $\bfMF(E,W).$ 
The analog of Theorem~\ref{t:basic-semi-orthog-1} for matrix
factorizations is the following theorem.

\begin{theorem}
  \label{t:semi-orthog-1-mf}
  Assume that $r\geq 1.$ Let $l \in \DZ.$
  \begin{enumerate}[label=$(\text{MF\arabic*})_E$]
  \item
    \label{enum:a:semi-orthog-1-mf}
    The functor $\mathcal{O} (l)\otimes p^*(-) \colon 
    \bfMF(Y,W) \ra \bfMF(E,W)$ is
    full and faithful.
  \end{enumerate}
  We denote 
  the essential image of this
  functor by 
  $\mathcal{O}(l) \otimes p^*\bfMF(Y,W).$
  \begin{enumerate}[resume*]
  \item
    \label{enum:b:semi-orthog-1-mf}
    The
    subcategory $\mathcal{O} (l)\otimes p^*
    \bfMF(Y,W)\subset \bfMF(E,W)$ is admissible.
  \item
    \label{enum:c:semi-orthog-1-mf}
    The category $\bfMF(E,W)$ has the semi-orthogonal
    decomposition\footnote{
      This is also true for $r=0$ since then $E=\emptyset.$
    }
    \begin{equation*}
      \bfMF(E,W)=
      \Big\langle 
      \mathcal{O} (-r+1)\otimes p^*\bfMF(Y,W), \dots,
      \mathcal{O} (-1)\otimes p^*\bfMF(Y,W), 
      p^*\bfMF(Y,W)
      \Big\rangle.
    \end{equation*}
  \end{enumerate}
\end{theorem}

\begin{proof}[Proof of \ref{enum:a:semi-orthog-1-mf}]
  Note that $\mathcal{O}_Y \sira 
  \bR p_*\mathcal{O}_E.$
  If $V$ is a vector bundle on $Y,$ this implies that the
  adjunction morphism $V \ra \bR p_* p^* V$ is an
  isomorphism. This means that if $p^*V \ra J$ is a (finite)
  resolution by injective quasi-coherent sheaves, then the
  obvious morphism $V \ra p_*(J)$ is a resolution of $V.$ Now let
  $F \in \bfMF(Y,W)$ and let $p^*F \ra I$ be an exact sequence as
  in Lemma~\ref{l:resolutions}.\ref{enum:Inj-reso}. Then the
  obvious morphism $F \ra p_*(I)$ is an exact sequence in
  $Z_0(\Qcoh(Y,W)),$ and
  Lemma~\ref{l:totalization}.\ref{enum:brutal-truncation-and-totalization}
  implies that the adjunction morphism $F \ra \bR p_* p^*F$ is an
  isomorphism. Hence  $p^* \colon  \bfMF(Y,W)
  \ra \bfMF(E,W)$ is full and faithful, and this clearly implies
  \ref{enum:a:semi-orthog-1-mf}.
\end{proof}

\begin{proof}[Proof of \ref{enum:b:semi-orthog-1-mf}] 
  It is certainly enough to show that
  $p^*\bfMF(Y,W) \subset \bfMF(E,W)$ 
  is admissible.
  By 
  Remark~\ref{rem:adjoints-and-semi-orthogonal-decomposition}
  and its dual version we need to prove that 
  the full and faithful functor $p^* \colon  \bfMF(Y,W) \ra
  \bfMF(E,W)$ has 
  a right and a left adjoint.
  Lemma~\ref{l:pi-proper-Rpi-lower-star-and-MF-and-adjunction}
  provides a right adjoint
  $\bR p_* \colon  \bfMF(E, W) \ra \bfMF(Y, W).$
  On the other hand, we see from
  Lemma~\ref{l:pi-us-and-duality-MF-setting} 
  that $D_Y \comp \bR p_* \comp D_E$ is left adjoint
  to $p^*.$
\end{proof}

It remains to prove
\ref{enum:c:semi-orthog-1-mf}. More precisely we need to prove
that the specified sequences of admissible subcategories are
semi-orthogonal and complete (see
Definition~\ref{d:semi-orthogonal-decomposition}). 

\begin{proof}[Proof of semi-orthogonality in 
  \ref{enum:c:semi-orthog-1-mf}]
  Lemma~\ref{l:orthog-mf} shows that this
  is a direct consequence of
  Theorem~\ref{t:basic-semi-orthog-1}.\ref{enum:c:basic-semi-orthog-1}
  (and this statement is not difficult to prove using the
  local-to-global Ext spectral sequence).
\end{proof}

We now prepare for the proof of completeness in
\ref{enum:c:semi-orthog-1-mf}.

Recall that the projection $p \colon E \ra Y$ is a $\DP^{r-1}$-bundle.
Let $\Omega_{E/Y}$ be the 
sheaf of relative differentials of $E$ over $Y$
(= the relative cotangent bundle on $E$)
and let $\Omega_{E/Y}^t=\wedge^t \Omega_{E/Y}$ (and $\Omega^0_{E/Y}=\mathcal{O}_E$).
Consider the pullback diagram
\begin{equation*}
  \xymatrix{
    {E\times_Y E} \ar[r]^-{q_2} \ar[d]^-{q_1} & 
    {E} \ar[d]^-{p_1}\\
    {E} \ar[r]^-{p_2} & 
    {Y}
  }
\end{equation*}
where $p_1=p_2=p.$ 
We define $F \boxtimes G := q_1^*F \otimes q_2^*G$ 
for $F,$ $G \in \Coh(E).$ 
Denote by $\Delta_E \subset
E\times_YE$ the diagonal subscheme.

In this situation
we have an exact sequence 
\begin{multline}
  \label{eq:sequence-resolution-diagonal-on-E}
  0 \ra 
  \mathcal{O}_E(-r+1)\boxtimes \Omega^{r-1}_{E/Y}(r-1) \ra 
  \dots \ra
  \mathcal{O}_E(-t)\boxtimes \Omega^t_{E/Y}(t) \ra
  \\
  \dots \ra
  \mathcal{O}_E(-1) \boxtimes \Omega_{E/Y}(1) \ra 
  \mathcal{O}_{E \times_Y E} \ra 
  \mathcal{O}_{\Delta_E} \ra 
  0
\end{multline}
in $\Coh(E\times_Y E)$
(cf.\ \cite[Remark 8.35]{huy-fourmukai}).
We denote this locally free resolution of
$\mathcal{O}_{\Delta_E}$ as $\mathcal{F} \ra
\mathcal{O}_{\Delta_E},$ i.\,e.\
$\mathcal{F}^{-t}=\mathcal{O}_E(-t) \boxtimes \Omega_{E/Y}^t(t)$
for $t \geq 0.$

\begin{proof}[Proof of completeness in 
  \ref{enum:c:semi-orthog-1-mf}]
  We essentially adapt the proof of
  \cite[Cor.~8.29]{huy-fourmukai}. 
  Let $t \geq 0.$
  For any $T \in \Coh(E)$ we have
  \begin{multline}
    \label{eq:proj-base-change}
    \bR q_{1*}\Big(
    \mathcal{F}^{-t}
    \otimes q_{2}^*(T)\Big)
    =
    \bR q_{1*}\Big(q_1^*\big(\mathcal{O}_E(-t)\big) \otimes
    q_2^*\big(\Omega^t_{E/Y}(t) \otimes T\big)\Big)\\
    =
    \mathcal{O}_E(-t) \otimes 
    \bR q_{1*} q_2^*\big(\Omega^t_{E/Y}(t)\otimes T\big)
    = \mathcal{O}_E(-t) \otimes 
    p^*_2\bR p_{1*}\big(\Omega^t_{E/Y}(t) \otimes T\big)
  \end{multline}
  in $D^b(\Qcoh(E))$
  (or $D^b(\Coh(E)),$ cf.\
  proof of
  Lemma~\ref{l:pi-proper-Rpi-lower-star-and-MF-and-adjunction}).
  Indeed the second
  equality is the projection formula and the third one is
  flat base change.

  In the following we
  use the exact functor $\Coh(E \times_Y E) \ra Z_0(\Coh(E
  \times_Y E,0)),$ $S \mapsto (\matfak{0}{}{S}{}),$ in order to
  view coherent sheaves as matrix factorizations. For example, 
  \eqref{eq:sequence-resolution-diagonal-on-E}
  can be viewed as a resolution of
  $(\matfak{0}{}{\mathcal{O}_{\Delta_E}}{})$ in $Z_0(\Coh(E \times_Y
  E,0)).$ 
   
  We claim that
  for $T \in
  \DCoh(E,W)$ 
  equation~\eqref{eq:proj-base-change}
  is also true in $\DQcoh(E,W)$ (or $\DCoh(E,W)$).
  Just
  observe that 
  projection formula and 
  flat base change also hold for matrix factorizations.
  This is easy to prove for the projection formula. For flat base
  change note that there is a natural morphism and use the
  following: if $I$ is in $[\InjQcoh(E,W)],$ then the usual flat
  base 
  change shows that $q_2^*(I)$ is right $q_{1*}$-acyclic, by   
  Lemma~\ref{l:componentwise-acyclics}.

  We break the exact
  sequence~\eqref{eq:sequence-resolution-diagonal-on-E} up into
  short 
  exact sequences
  $\mathcal{F}^{-r+1} \hra \mathcal{F}^{-r+2} \sra
  \mathcal{K}^{-r+3},$ \dots, $\mathcal{K}^{-t} \hra
  \mathcal{F}^{-t} \sra \mathcal{K}^{-t+1},$ \dots, $\mathcal{K}^0
  \hra \mathcal{F}^0 \sra \mathcal{O}_{\Delta_E}.$
  These short exact sequences give rise to triangles in
  $\DCoh(E \times_Y E, 0).$
  
  Let $T \in \DCoh(E,W).$ Form the derived tensor product of
  $q_2^*(T)$ with these triangles and apply $\bR q_{1*}.$
  Using induction and \eqref{eq:proj-base-change} 
  we see that
  \begin{equation*}
    \bR q_{1*}\big(
    \mathcal{O}_{\Delta_E}
    \otimes^{\bL} q_{2}^*(T)\big)
    \in   
    \tria\big(
    \mathcal{O}_E(-r+1) \otimes p^*_2 \DCoh(Y,W),
    \dots,
    p^*_2 \DCoh(Y,W)\big).
  \end{equation*}
  The object on the left is the image of $T$ under the
  Fourier-Mukai 
  transform with kernel
  $\mathcal{O}_{\Delta_E}$ 
  (in the setting of matrix
  factorizations).
  Hence it is isomorphic
  to $T:$  
  for $\delta \colon  E \ra E \times_Y E$ the (affine) diagonal inclusion
  we have 
  \begin{equation*}
    \mathcal{O}_{\Delta_E}
    \otimes^{\bL} q_{2}^*(T)
    =
    \delta_*(\mathcal{O}_E)
    \otimes^{\bL} q_{2}^*(T)
    =
    \delta_*(\mathcal{O}_{E}
    \otimes^{\bL} \delta^*(q_{2}^*(T)))
    =\delta_*(T)
  \end{equation*}
  by the projection formula, and hence 
  $\bR q_{1*}\big(
  \mathcal{O}_{\Delta_E}
  \otimes^{\bL} q_{2}^*(T)\big)
  =\bR q_{1*}(\delta_*(T))=T.$
  Now  completeness in 
  \ref{enum:c:semi-orthog-1-mf}
  is immediate from
  Theorem~\ref{t:equivalences-curved-categories}.\ref{enum:MF-Dcoh-equiv}.
\end{proof}
This finishes the proof of Theorem~\ref{t:semi-orthog-1-mf}

The following result lifts the semi-orthogonal decomposition
from 
Theorem~\ref{t:semi-orthog-1-mf} to the dg level.
We will need it in 
\cite{valery-olaf-matfak-motmeas-in-prep}. 
We use the enhancement by dg quotients explained in 
section~\ref{sec:enhanc-dg-quot}. 

\begin{corollary}
  \label{c:semi-orthog-1-mf-dg-level}
  There are full dg subcategories $\mathcal{Y}_{l}$ of
  $\MF(E,W)/\AcyclMF(E,W)$ (for $l \in \DZ$) which are 
  quasi-equivalent to $\MF(Y,W)/\AcyclMF(Y,W)$ such that 
  the semi-orthogonal decomposition into admissible
  subcategories from 
  Theorem~\ref{t:semi-orthog-1-mf}.\ref{enum:c:semi-orthog-1-mf}
  is given by
  \begin{equation*}
    [\MF(E,W)/\AcyclMF(E,W)]= \langle [\mathcal{Y}_{-r+1}], \dots,
      [\mathcal{Y}_{-1}], [\mathcal{Y}_0] \rangle
  \end{equation*}
  if we identify $\bfMF(E,W)$ with the left-hand side along 
  \eqref{eq:dg-quot-enhancement-equiv}.
\end{corollary}
 
\begin{proof}
  The functor $\mathcal{O}(l) \otimes p^*(-) \colon  \MF(Y,W) \ra
  \MF(E,W)$ maps 
  $\AcyclMF(Y,W)$ to $\AcyclMF(E,W)$ and hence induces a dg
  functor
  \begin{equation*}
     \mathcal{O}(l) \otimes p^*(-) \colon  \MF(Y,W)/\AcyclMF(Y,W) \ra
  \MF(E,W)/\AcyclMF(E,W).
  \end{equation*}
  On homotopy categories this is the full and faithful functor
  $\mathcal{O}(l) \otimes p^*(-) \colon  \bfMF(Y,W) \ra \bfMF(E,W)$ 
  from \ref{enum:a:semi-orthog-1-mf};
  here and in the following we identify along \eqref{eq:dg-quot-enhancement-equiv}.
  Define $\mathcal{Y}_l$ to be the full dg subcategory of 
  $\MF(E,W)/\AcyclMF(E,W)$ consisting of objects that 
  belong to $\mathcal{O}(l) \otimes p^*\bfMF(Y,W).$
  All claims follow now from
  Theorem~\ref{t:semi-orthog-1-mf}.
\end{proof}

\subsection{Blowing-ups}
\label{sec:blowing-ups}

Now we describe the setting of a blowing-up.
Let $X$ be a scheme satisfying condition~\ref{enum:srNfKd}
and let
$i \colon Y \hra X$ be the embedding of a regular
equi-codimensional closed 
subscheme.
Let $\pi \colon \tildew{X} \ra X$ be the
blowing-up of 
$X$ along $Y,$ cf.\ \cite[8.1]{Liu} and 
\cite[13]{goertz-wedhorn-AGI}, and 
denote by $j \colon E \hra \tildew{X}$ the inclusion of
the exceptional divisor. 
We have the following pullback diagram
\begin{equation*}
  \xymatrix{
    {E} \ar[r]^-{j} \ar[d]^-{p} &
    {\tildew{X}} \ar[d]^-{\pi} \\
    {Y} \ar[r]^-i &
    {X.}
  }
\end{equation*}
By the usual construction of the blowing-up, $\tildew{X}$ is
endowed with the line bundle
$\mathcal{O}(1)=\mathcal{O}_{\tildew{X}}(1).$ 
This line bundle is
the ideal sheaf corresponding to the closed subscheme $E \subset
\tildew{X},$ i.\,e.\ we have a short exact sequence
\begin{equation}
  \label{eq:ses-ideal-sheaf-E}
  \mathcal{O}_{\tildew{X}}(1) \hra \mathcal{O}_{\tildew{X}} \sra
  \mathcal{O}_E.
\end{equation}
We often denote the restriction $\mathcal{O}_{E}(1)$
of $\mathcal{O}(1)=\mathcal{O}_{\tildew{X}}(1)$ to $E$ by
$\mathcal{O}(1)$ as well.

Let $\mathcal{J} \subset \mathcal{O}_X$ be the ideal sheaf of $Y
\subset X.$ Note that $i$ is a regular immersion of a fixed codimension by \cite[6.3.1]{Liu}; 
we denote this codimension by $r.$
In particular $\mathcal{J}/\mathcal{J}^2$ is locally free of rank
$r$ on $Y.$ Moreover,
the projection $p \colon E \ra Y$ 
is a $\DP^{r-1}$-bundle (as in
subsection~\ref{sec:proj-space-bundl}), more precisely it is
isomorphic to 
$\DP(\mathcal{J}/\mathcal{J}^2) \ra Y$
(use \cite[Thm.~8.1.19]{Liu}, cf.\ \cite[Thm.~II.8.24]{Hart}).


Recall the following semi-orthogonal decomposition
theorem from\footnote{
  The assumption there is that $Y$
  is a smooth projective variety over a field.
} 
\cite{orlov-monoidal,bondal-orlov-semi-orthogonal-only-arXiv}\footnote{
  The proof of Theorem~\ref{t:basic-semi-orthog-2}
  in 
  \cite{orlov-monoidal,bondal-orlov-semi-orthogonal-only-arXiv}
  is incomplete. We thank A.~Kuznetsov for explaining to us how
  to fill in the gap. We use his suggestion to prove our
  Theorem~\ref{t:semi-orthog-2-mf} below. 
},
\cite[Prop.~11.18]{huy-fourmukai}.

\begin{theorem}
  \label{t:basic-semi-orthog-2}
  Assume that $r\geq 2.$ 
  \begin{enumerate}[label=$(\text{Coh\arabic*})_{\tildew{X}}$,
    start=0]
  \item The functor $\bL \pi^* \colon D^b(\Coh(X))\ra
    D^b(\Coh(\tildew{X}))$ is full and faithful.
  \end{enumerate}
  Let $l \in \DZ$
  and consider the functor
  \begin{equation*}
    t_l(-):=j_*(\mathcal{O} (l)\otimes p^*(-)) \colon  
    D^b(\Coh(Y)) \ra D^b(\Coh(\tildew{X})).
  \end{equation*}
  \begin{enumerate}[label=$(\text{Coh\arabic*})_{\tildew{X}}$]
  \item
    \label{enum:a:basic-semi-orthog-2}
    The functor $t_l$ is full and faithful.
  \end{enumerate}
  Denote by 
  $D^b(\Coh(Y))_l$
  the essential image of $t_l,$
  and by 
  $\bL \pi^*D^b(\Coh(X))$ the essential image of
  $\bL \pi^* \colon D^b(\Coh(X))\ra D^b(\Coh(\tildew{X})).$ 
  \begin{enumerate}[resume*]
  \item
    \label{enum:b:basic-semi-orthog-2}
    The subcategories $D^b(\Coh(Y))_l$
    and
    $\bL\pi^*D^b(\Coh(X))$ are admissible in
    $D^b(\Coh(\tildew{X})).$ 
  \item
    \label{enum:c:basic-semi-orthog-2}
    The category $D^b(\Coh(\tildew{X}))$ has the 
    semi-orthogonal decomposition 
    \begin{equation*}
      D^b(\Coh(\tildew{X}))=
      \Big\langle D^b(\Coh(Y))_{-r+1}, \dots, 
      D^b(\Coh(Y))_{-1},
      \bL\pi^*D^b(\Coh(X))\Big\rangle.
    \end{equation*}
  \end{enumerate}
\end{theorem}

Now assume that we are given a morphism $W \colon X\ra \DA^1.$ 
It induces morphisms from $Y,$ $\tildew{X}$ and $E$ to $\DA^1$ 
which we again denote by $W.$
Note that $X,$ $Y,$ $\tildew{X},$ and $E$ satisfy
condition \ref{enum:srNfKd}.
Consider the commutative diagram
\begin{equation*}
  \xymatrix{
    {\bfMF(E,W)} &
    {\bfMF(\tildew{X},W)} \ar[l]_{j^*}\\
    {\bfMF(Y,W)}  \ar[u]^{p^*} &
    {\bfMF(X,W).} \ar[l]_{i^*} \ar[u]^{\pi^*}
  }
\end{equation*}
Here $\bL\pi^*=\pi^*,$
and similarly for the other functors in this diagram.
We also have the functor $j_*=\bR j_* \colon  \bfMF(E,W)
\ra \DCoh(\tildew{X},W),$
see
Remark~\ref{rem:derived-direct-image-for-affine-morphism}.
Note that tensoring with the line bundles $\mathcal{O}(l)$ induces
autoequivalences of the categories $\bfMF(E,W)$ and
$\bfMF(\tildew{X},W).$

Our goal now is to prove the
following analog of Theorem~\ref{t:basic-semi-orthog-2}.

\begin{theorem}
  \label{t:semi-orthog-2-mf}
  Let $r \geq 2.$
  \begin{enumerate}[label=$(\text{MF\arabic*})_{\tildew{X}}$,
    start=0,series=MFXtilde]
  \item
    \label{enum:a:semi-orthog-2-mf}
    The functor
    $\pi^*  \colon \bfMF(X,W)\ra \bfMF(\tildew{X},W)$ is full and
    faithful.
  \end{enumerate}
  For any integer $l$ consider the functor
  \begin{equation*}
    s_l(-) := j_*(\mathcal{O} (l)\otimes p^*(-)) \colon 
    \bfMF(Y,W) \ra \DCoh(\tildew{X}, W)  
  \end{equation*}
  and recall that the latter category is equivalent to
  $\bfMF(\tildew{X},W).$ 
  \begin{enumerate}[resume*=MFXtilde]
  \item
    \label{enum:b:semi-orthog-2-mf}
    The functor $s_l$ is full and faithful.
  \end{enumerate}
  Denote by
  $\pi^*\bfMF(X,W)$ the essential image of 
  $\pi^*  \colon \bfMF(X,W)\ra \bfMF(\tildew{X},W),$
  and by $\bfMF(Y,W)_l$
  the intersection of 
  the essential image of $s_l$
  with
  $\bfMF(\tildew{X},W).$
  \begin{enumerate}[resume*=MFXtilde]
  \item
    \label{enum:c:semi-orthog-2-mf}
    The subcategory $\pi^*\bfMF(X,W) \subset
    \bfMF(\tildew{X},W)$ is admissible, and so are the
    subcategories $\bfMF(Y,W)_l \subset
    \bfMF(\tildew{X},W),$ for any $l \in \DZ.$
  \item
    \label{enum:d:semi-orthog-2-mf}
    The category $\bfMF(\tildew{X},W)$ has the    
    semi-orthogonal decomposition\footnote{
      This is also trivially true for $r=0$ and $r=1.$
    }
    \begin{equation*}
      \bfMF(\tildew{X},W)=
      \Big\langle \bfMF(Y,W)_{-r+1}, \dots,\bfMF(Y,W)_{-1}, \pi^*
      \bfMF(X,W)\Big\rangle.
    \end{equation*}
  \end{enumerate}
\end{theorem}

\begin{proof}[Proof of \ref{enum:a:semi-orthog-2-mf}]
  We can proceed as in the proof of 
  \ref{enum:a:semi-orthog-1-mf} since
  $\mathcal{O}_X \sira 
  \bR\pi_* \bL\pi^*\mathcal{O}_{X} 
  =\bR\pi_*\mathcal{O}_{\tildew{X}}$
  (this follows for example from
  \cite[Thm.~8, Rem.~9]{salas-direct-proof-formal-functions}).
\end{proof}

\begin{proof}[Proof of \ref{enum:b:semi-orthog-2-mf}]
  Fix $\ul{M},$ $\ul{N} \in \bfMF(Y,W)$ and $l \in \DZ.$ Put
  \begin{equation*}
    M :=\mathcal{O} (l)\otimes p^*\ul{M}, \quad N :=
    \mathcal{O} (l)\otimes p^*\ul{N}.
  \end{equation*}
  We already know \ref{enum:a:semi-orthog-1-mf}. Hence it
  suffices to show that the morphism
  \begin{equation}
    \label{eq:j-ls-provides-isom}
    j_* \colon \Hom_{\bfMF(E,W)} (M ,N ) \ra \Hom_{\DCoh(\tildew{X},W)} (j_*M, j_*N)
  \end{equation}
  is an isomorphism.

  Using the short exact sequence
  \eqref{eq:ses-ideal-sheaf-E}
  and the method used in the proof of
  Lemma~\ref{l:resolutions}.\ref{enum:MF-reso}
  we find an exact sequence $0\ra Q^{-1}\ra Q^0\ra j_*M \ra 0$
  in $Z_0(\Coh(\tildew{X},W))$ with
  $Q^0,$ $Q^{-1} \in \bfMF(\tildew{X},W).$ Let
  $Q=(\dots \ra 0 \ra Q^{-1} \ra Q^0 \ra 0 \ra \dots),$
  and let $r \colon  \Tot(Q) \ra j_*M$ be the obvious morphism.
  Then by the definition of $\bL j^*$ we can assume that 
  $\bL j^*j_*M = j^*(\Tot(Q))= \Tot(j^*(Q)).$
  Consider the composition 
  \begin{equation*}
    \theta \colon  \bL j^* j_*(M)=j^*(\Tot(Q)) \xra{j^*(r)}
    j^*j_* M \ra M  
  \end{equation*}
  where the last morphism is the obvious one.
  It is enough to show that the morphism
  \begin{equation}
    \label{eq:theta-us-provides-isom}
    \theta^* \colon \Hom_{\bfMF(E,W)} (M, N) \ra 
    \Hom_{\bfMF(E,W)} (j^*(\Tot(Q)), N)
  \end{equation}
  is an isomorphism:
  if we compose
  the morphism \eqref{eq:theta-us-provides-isom}
  with the isomorphism given by the adjunction in
  Theorem~\ref{t:derived-inverse-and-direct-image}, we 
  obtain the morphism \eqref{eq:j-ls-provides-isom}.

  Fit $\theta$ into a triangle
  \begin{equation*}
    C \ra \Tot(j^*(Q)) \xra{\theta} M \ra [1]C
  \end{equation*}
  in $\bfMF(E,W).$ Applying the cohomological functor
  $\Hom_{\bfMF(E,W)}  (?, N)$ to this triangle shows that we need to
  prove that
  \begin{equation*}
    \Hom_{\bfMF(E,W)}  ([v]C, N) =0 \quad \text{for all $v \in \DZ_2.$}
  \end{equation*}
  By Proposition~\ref{p:locality-of-orthogonality} it is
  sufficient to prove this under the additional assumption that
  $X$ (and hence $Y$) are affine. Moreover we can and will assume
  that $\ul{M}$ and $\ul{N}$ are free $\mathcal{O}_Y$-modules of
  finite rank; then $M$ and $N$ are finite direct sums of copies
  of the line bundle $\mathcal{O}_E(l).$

  It is easy to see that $M=H^0(j^*(Q))$ and that
  $M':=H^{-1}(j^*(Q))$ coincides with $M(1)$ as a graded vector
  bundle on $E$ (use the short exact sequence
  \eqref{eq:ses-ideal-sheaf-E}).  We claim that $C \cong [1]M'$
  in this case, i.\,e.\ the morphism $\theta \colon  \Tot(j^*(Q)) \ra
  M$ fits into a triangle $[1]M' \ra \Tot(j^*(Q)) \xra{\theta} M \ra
  [2]M'.$

  Let
  $A:=j^*(Q).$ Let $\tau_{\leq -1}(A)$ be the kernel of the
  obvious surjective morphism $A \ra H^0(A)=M$ of complexes in
  $Z_0(\MF(E,W)),$ where $M$ is concentrated in degree 0. We
  obtain a short exact sequence $\tau_{\leq -1}(A) \ra A \ra
  H^0(A)=M$ of complexes in $Z_0(\MF(E,W)).$ Taking totalizations
  we obtain a short exact sequence in $Z_0(\MF(E,W)$ which becomes a
  triangle in $\bfMF(E,W)$ (by
  Lemma~\ref{l:totalization}.\ref{enum:ses-gives-triangle}).  On
  the other hand note that there is an obvious quasi-isomorphism
  $M'=H^{-1}(A) \ra \tau_{\leq -1}(A)$ of complexes in
  $Z_0(\MF(E,W)),$ where $M'$ is put in degree -1. It gives rise
  to a morphism $[1]M' \ra \Tot(\tau_{\leq -1}(A))$ in
  $Z_0(\MF(E,W))$ and to an isomorphism in $\bfMF(E,W).$ Altogether
  we obtain the triangle $[1]M' \ra \Tot(j^*(Q)) \xra{\theta} M
  \ra [2]M'$ 
  we claimed to exist, in particular $C\cong [1]M'$ in
  $\bfMF(E,W).$

  Hence we are reduced to proving that
  \begin{equation*}
    \Hom_{\bfMF(E,W)}([v]M', N)=0 \quad \text{for all $v \in \DZ_2.$}
  \end{equation*}
  Since $M'$ and $M(1)$ coincide (at least) as graded vector
  bundles this follows 
  from Lemma~\ref{l:orthog-mf} and our assumptions on $\ul{M}$ 
  and $\ul{N}$
  since
  \begin{equation*}
    \Hom_{D(\Qcoh(E))}([v]\mathcal{O}_E(l+1), \mathcal{O}_E(l))
    = H^{-v}(E, \mathcal{O}_E(-1))=0 
  \end{equation*}
  for all $v \in \DZ.$ Here we use that $p \colon E \ra Y$ is a
  $\DP^{r-1}$-bundle and that $r-1 \geq 1.$
  This finishes the proof of \ref{enum:b:semi-orthog-2-mf}.
\end{proof}

Proposition~\ref{p:j-ls-and-duality-MF-setting} below
is essential for the proof of
(the second part of) \ref{enum:c:semi-orthog-2-mf}.
It says how $j_*$ commutes with the duality
\eqref{eq:duality-derived-level}.
Its proof will use the following trivial result.

\begin{lemma}
  \label{l:rearrange-modified-resolution}
  Let $A$ be a ring. Let $p \colon  P \sra M$ be a surjection of
  $A$-modules
  and let $q \colon  Q \ra M$ be any
  morphism of $A$-modules with $Q$ projective.
  Consider the morphism $(p,q) \colon  P \oplus Q \ra M.$
  Then there is a morphism $l \colon  Q \ra P$ such that the diagram
  \begin{equation*}
    \xymatrix{
      {P \oplus Q} \ar[r]^-{(p,0)} 
      \ar[d]^-{\sim}_-{
        \Big[
        \begin{smallmatrix}
          1 & -l \\
          0 & 1
        \end{smallmatrix}
        \Big]
      }
      &
      {M} \gar[d] \\
      {P \oplus Q} \ar[r]^-{(p,q)} &
      {M} \\
    }
  \end{equation*}
  commutes.
\end{lemma}

\begin{proof}
  Since $p$ is surjective and $Q$ is projective there is $l \colon  Q
  \ra P$ such that $pl =q.$
\end{proof}

\begin{proposition}
  \label{p:j-ls-and-duality-MF-setting}
  There is an isomorphism 
  $\tau \colon  j_* \comp D_E \sira [1](1) D_{\tildew{X}} \comp j_*$
  of functors 
  $\bfMF(E, W) \ra \DCoh(\tildew{X}, -W)^\opp.$
\end{proposition}

\begin{proof}
  We first define the morphism $\tau$ globally and show afterwards
  locally that it is an isomorphism.

  The short exact sequence \eqref{eq:ses-ideal-sheaf-E}
  gives rise to a short exact sequence 
  in $Z_0(\Coh(\tildew{X}, 0))$ and then to a triangle
  \begin{equation*}
    \mathovalbox{\matfak{0}{}{(1)\mathcal{O}_{\tildew{X}}}{}}
    \hra
    \mathovalbox{\matfak{0}{}{\mathcal{O}_{\tildew{X}}}{}}
    \sra
    \mathovalbox{\matfak{0}{}{j_*\mathcal{O}_E}{}}
    \xra{\delta}
    \mathovalbox{\matfak{(1)\mathcal{O}_{\tildew{X}}}{}{0}{}}
  \end{equation*}
  in $\DCoh(\tildew{X}, 0).$ For later use we describe $\delta$
  explicitly. Consider the obvious morphisms
  \begin{equation*}
    \mathovalbox{\matfak{0}{}{j_*\mathcal{O}_E}{}}
    \xla{\rho}
    \mathovalbox{\matfak{(1)\mathcal{O}_{\tildew{X}}}{\iota}
      {\mathcal{O}_{\tildew{X}}}{0}}
    \xra{\delta'}
    \mathovalbox{\matfak{(1)\mathcal{O}_{\tildew{X}}}{}{0}{}}
  \end{equation*}
  in $Z_0(\Coh(\tildew{X},0))$ 
  where the inclusion $(1)\mathcal{O}_{\tildew{X}}
  \hra \mathcal{O}_{\tildew{X}}$ is denoted by $\iota.$
  The morphism $\rho$ becomes
  invertible in $\DCoh(\tildew{X}, 0),$ and there we have
  $\delta= \delta' \comp \rho\inv.$

  Now define $\tau$ to be the composition
  \begin{multline}
    \label{eq:definition-tau}
    \tau \colon  j_* \comp D_E = 
    j_* \bR\sheafHom_{\mathcal{O}_E}
    (-, \big(\matfak{0}{}{\mathcal{O}_E}{}\big))
    \ra
    \bR\sheafHom_{\mathcal{O}_{\tildew{X}}}(j_*(-),
    j_*\big(\matfak{0}{}{\mathcal{O}_E}{}\big))\\
    \xra{\delta_*}
    \bR\sheafHom_{\mathcal{O}_{\tildew{X}}}(j_*(-),
    \big(\matfak{(1)\mathcal{O}_{\tildew{X}}}{}{0}{}\big))
    =[1](1) D_{\tildew{X}} \comp j_*
  \end{multline}
  where the first morphism is the obvious one
  and the second one is induced by $\delta.$ The last equality
  is obvious.

  Our aim is now to show that $\tau$ is in fact an isomorphism.
  It is enough to test this locally
  on an affine open subset $\Spec A \subset \tildew{X}$
  (use Corollary~\ref{c:locality-of-being-an-isomorphism}). 
  We can moreover assume that 
  \eqref{eq:ses-ideal-sheaf-E} is given by 
  $A \xhra{f} A \xsra{c} A/f,$ for some $f \in A.$

  Let $M = (\matfak{M_1}{m_1}{M_0}{m_0}) \in \bfMF(\Spec A/f,
  W).$ By further shrinking $\Spec A$ we can and will assume
  that the components of $M$ are free $A/f$ modules of finite
  rank, $M_0=(A/f)^{\oplus s_0}$ and $M_1=(A/f)^{\oplus s_1}$
  for suitable $s_0, s_1 \in \DN.$

  Let $P_i := A^{\oplus s_i}.$ We denote the morphisms
  $c^{\oplus s_i} \colon  P_i \sra M_i$ and $f^{\oplus s_i} \colon  P_i
  \hra P_i$ simply by $c$ and $f$ respectively.

  The method
  of Lemma~\ref{l:resolutions}.\ref{enum:MF-reso}
  (with a little help from 
  Lemma~\ref{l:rearrange-modified-resolution})
  provides
  the following (vertical) short exact sequence
  $Q^{-1} \xhra{q} Q^0 \xsra{r} j_*M$
  in $Z_0(\Coh(\Spec A,W)),$ a two-step resolution of
  $j_*(M)$ by objects of $\MF(\Spec A, W).$ 
  \begin{equation*}
    \xymatrix@R=2cm@C=1.5cm{
      {j_*M \colon } & 
      {M_1}
      \ar@<0.3ex>[rr]^-{m_1}
      &&
      {M_0}
      \ar@<0.3ex>[ll]^-{m_0}
      \\
      {Q^0 \colon } \ar[u]^-{r} &
      {P_0 \oplus P_1}
      \ar@<0.3ex>[rr]^-{
        \footnotesize{\begin{bmatrix}
            fu_1 & \beta \\
            -\alpha & 1
          \end{bmatrix}}
      }
      \ar[u]^-{
        \footnotesize{\begin{bmatrix}
            0 & c
          \end{bmatrix}}
      }
      &&
      {P_0 \oplus P_1}
      \ar@<0.3ex>[ll]^-{
        \footnotesize{\begin{bmatrix}
            1 & -\beta \\
            \alpha & fu_0
          \end{bmatrix}}
      }
      \ar[u]_-{
        \footnotesize{\begin{bmatrix}
            c & 0
          \end{bmatrix}}
      }
      \\
      {Q^{-1} \colon } \ar[u]^-{q} &
      {P_0 \oplus P_1}
      \ar@<0.3ex>[rr]^-{
        \footnotesize{\begin{bmatrix}
            u_1 & \beta \\
            - \alpha & f
          \end{bmatrix}}
      }
      \ar[u]^-{
        \footnotesize{\begin{bmatrix}
            1 & 0 \\
            0 & f
          \end{bmatrix}}
      }
      &&
      {P_0 \oplus P_1}
      \ar@<0.3ex>[ll]^-{
        \footnotesize{\begin{bmatrix}
            f & - \beta \\
            \alpha & u_0
          \end{bmatrix}}
      }
      \ar[u]_-{
        \footnotesize{\begin{bmatrix}
            f & 0 \\
            0 & 1
          \end{bmatrix}}
      }
    }
  \end{equation*}
  Here $\alpha, \beta, u_0, u_1$ are suitable morphisms
  satisfying
  \begin{align*}
    c \alpha & = m_0 c, & W & = fu_1 + \beta \alpha, & 
    \alpha u_1 & = u_0 \alpha,\\
    \notag
    c \beta & = m_1 c, & W & = fu_0 + \alpha \beta, &
    \beta u_0 & = u_1 \beta.
  \end{align*}
  Note that $Q^0$ is isomorphic to zero in $[\Coh(\Spec A, W)],$
  as observed in Remark~\ref{rem:middle-objects-vanish}.

  Let $T:= \Tot(Q^{-1} \xra{q} Q^0)$ be the cone of $q.$ Then $r$
  defines a 
  morphism $r' \colon  T \ra j_*M$ in $Z_0(\Coh(\Spec A, W))$ that
  becomes an 
  isomorphism in $\DCoh(\Spec A, W).$
  Explicitly $r'$ is given by the upper part of the following
  diagram.
  \begin{equation*}
    \xymatrix@R=2.7cm@C=1.6cm{
      {j_*M \colon } & 
      {M_1}
      \ar@<0.3ex>[rrr]^-{m_1}
      &&&
      {M_0}
      \ar@<0.3ex>[lll]^-{m_0}
      \\
      {T \colon } \ar[u]^-{r'} \ar[d]^-{u} &
      {P_0 \oplus P_1 \oplus P_0 \oplus P_1}
      \ar@<0.3ex>[rrr]^-{
        \footnotesize{\begin{bmatrix}
            fu_1 & \beta & f & 0 \\
            -\alpha & 1 & 0 & 1 \\
            0 & 0 & -f & \beta \\
            0 & 0 & -\alpha & -u_0 
          \end{bmatrix}}
      }
      \ar[u]^-{
        \footnotesize{\begin{bmatrix}
            0 & c & 0 & 0
          \end{bmatrix}}
      }
      \ar[d]^-{
        u_1 =
        \footnotesize{\begin{bmatrix}
            0 & 0 & 1 & 0\\
            0 & 0 & 0 & 1
          \end{bmatrix}}      
      }
      &&&
      {P_0 \oplus P_1 \oplus P_0 \oplus P_1.}
      \ar@<0.3ex>[lll]^-{
        \footnotesize{\begin{bmatrix}
            1 & -\beta & 1 & 0\\
            \alpha & fu_0 & 0 & f\\
            0 & 0 & -u_1 & -\beta\\
            0 & 0 & \alpha & -f
          \end{bmatrix}}
      }
      \ar[u]_-{
        \footnotesize{\begin{bmatrix}
            c & 0 & 0 & 0
          \end{bmatrix}}
      }
      \ar[d]^-{
        u_0 =
        \footnotesize{\begin{bmatrix}
            0 & 0 & 1 & 0\\
            0 & 0 & 0 & 1
          \end{bmatrix}}      
      }
      \\
      {[1]Q^{-1} \colon } 
      \ar@/^3mm/@{..>}[u]^-{
        t
      }
      &
      {P_0 \oplus P_1}
      \ar@<0.3ex>[rrr]^-{
        \footnotesize{\begin{bmatrix}
            -f & \beta \\
            -\alpha & -u_0 
          \end{bmatrix}}
      }
      \ar@/^3mm/@{..>}[u]^-{
        t_1=
        \footnotesize{\begin{bmatrix}
            0 & 0\\
            0 & -1\\
            1 & 0\\
            0 & 1
          \end{bmatrix}}
      }
      &&&
      {P_0 \oplus P_1.}
      \ar@<0.3ex>[lll]^-{
        \footnotesize{\begin{bmatrix}
            -u_1 & -\beta\\
            \alpha & -f
          \end{bmatrix}}
      }
      \ar@/^3mm/@{..>}[u]^-{
        t_0=
        \footnotesize{\begin{bmatrix}
            -1 & 0\\
            0 & 0\\
            1 & 0\\
            0 & 1
          \end{bmatrix}}
      }
    }
  \end{equation*}
  The morphism $u$ in the lower part of this diagram is the
  obvious projection morphism in the triangle $Q^{-1} \xra{q} Q^0
  \ra T \xra{u} [1]Q^{-1}$ in $[\Coh(\Spec A, W)].$ 
  We have observed above that $Q^0 = 0$ in $[\Coh(\Spec A, W)],$
  so $u$ is an isomorphism. The dotted morphism $t$ in the
  diagram in  
  $Z_0(\Coh(\Spec A, W))$ represents the inverse of $u$ in 
  $[\Coh(\Spec A, W)].$
  
  Hence the morphism $r'':= r' \comp t \colon [1]Q^{-1} \ra j_*M$ in
  $Z_0(\Coh(\Spec A, W))$ becomes an isomorphism in $\DCoh(\Spec
  A, W).$ It is given by
  \begin{equation*}
    \xymatrix@R=2cm@C=1.5cm{
      {j_*M \colon } & 
      {M_1}
      \ar@<0.3ex>[rr]^-{m_1}
      &&
      {M_0}
      \ar@<0.3ex>[ll]^-{m_0}
      \\
      {[1]Q^{-1} \colon } 
      \ar[u]^-{
        r''
      }
      &
      {P_0 \oplus P_1}
      \ar@<0.3ex>[rr]^-{
        \footnotesize{\begin{bmatrix}
            -f & \beta \\
            -\alpha & -u_0 
          \end{bmatrix}}
      }
      \ar[u]^-{
        r''_1=
        \footnotesize{\begin{bmatrix}
            0 & -c
          \end{bmatrix}}
      }
      &&
      {P_0 \oplus P_1.}
      \ar@<0.3ex>[ll]^-{
        \footnotesize{\begin{bmatrix}
            -u_1 & -\beta\\
            \alpha & -f
          \end{bmatrix}}
      }
      \ar[u]_-{
        r''_0=
        \footnotesize{\begin{bmatrix}
            -c & 0
          \end{bmatrix}}
      }
    }
  \end{equation*}
  We need to prove that the composition
  \begin{multline*}
    \Hom_{A/f}(M, \big(\matfak{0}{}{A/f}{}\big))\\
    \ra
    \mathcal{E}:=\Hom_{A}(j_*M, \big(\matfak{0}{}{A/f}{}\big))
    \xra{r''^*}
    \mathcal{F}:=\Hom_{A}([1]Q^{-1}, \big(\matfak{0}{}{A/f}{}\big))
    \\
    \xra{(\rho_*)\inv}
    \mathcal{G}:=\Hom_{A}([1]Q^{-1}, \big(\matfak{A}{f}{A}{0}\big))
    \xra{\delta'_*}
    \mathcal{H}:=\Hom_{A}([1]Q^{-1}, \big(\matfak{A}{}{0}{}\big))
  \end{multline*}
  in $\DCoh(\Spec A, -W)$ is an isomorphism.
  The first arrow clearly is an isomorphism.
  The first (resp.\ last) two arrows correspond to the first
  (resp.\ second) arrow in the definition  
  \eqref{eq:definition-tau} of $\tau_M.$
  The following diagram in
  $Z_0(\Coh(\Spec A, -W))$ 
  depicts 
  $\mathcal{E}
  \xra{r''^*}
  \mathcal{F}
  \xla{\rho_*}
  \mathcal{G}
  \xra{\delta'_*}
  \mathcal{H}$
  explicitly, cf.\ 
  \eqref{eq:duality-explicitly}
  and
  \eqref{eq:def-sheafHom-QcohXWV}
  (we write $\ol{P}_i$ for $P_i/f P_i;$
  note that $\Hom_A(A/f,A/f)=A/f$ and $\Hom_A(A,A/f)=A/f$ and 
  $\Hom_A(A,A)=A$ canonically, so that we can for example
  identify $\Hom_A(M_i, A/f)=M_i$ and $\Hom(P_0,
  A/f)=\ol{P}_0;$ 
  note that some matrix entries are zero
  since $f \colon A/f \ra A/f$ is zero;
  we indicate the transpose of a matrix by an upper index~$\tp$).
  \begin{equation*}
    \xymatrix@R=2.8cm@C=1.4cm{
      {\mathcal{E}} 
      \ar[d]_-{r''^*}
      &
      {M_1}
      \ar@<0.3ex>[rrr]^-{m_0^\tp}
      \ar[d]_-{r''^*_1=
        \footnotesize{\begin{bmatrix}
            0 \\ 
            -1
          \end{bmatrix}}
      }
      &&&
      {M_0}
      \ar@<0.3ex>[lll]^-{-m_1^\tp}
      \ar[d]^-{r''^*_0=
        \footnotesize{\begin{bmatrix}
            -1 \\
            0 
          \end{bmatrix}}
      }
      \\
      {\mathcal{F}}
      &
      {\ol{P}_0 \oplus \ol{P}_1}
      \ar@<0.3ex>[rrr]^-{
        \footnotesize{\begin{bmatrix}
            -\ol{u}_1^\tp & \ol{\alpha}^\tp \\
            -\ol{\beta}^\tp & -f=0
          \end{bmatrix}}
      }
      &&&
      {\ol{P}_0 \oplus \ol{P}_1}
      \ar@<0.3ex>[lll]^-{
        \footnotesize{\begin{bmatrix}
            f=0 & \ol{\alpha}^\tp \\
            -\ol{\beta}^\tp & \ol{u}_0^\tp 
          \end{bmatrix}}
      }
      \\
      {\mathcal{G}}
      \ar[u]^-{\rho_*}
      \ar[d]^-{\delta'_*}
      &
      {P_0 \oplus P_1 \oplus P_0 \oplus P_1}
      \ar[u]^-{\rho_{*1}=\can \comp \pr_{12}}
      \ar[d]^-{\delta'_{*1}=\pr_{34}}
      \ar@<0.3ex>[rrr]^-{
        \footnotesize{\begin{bmatrix}
            -u_1^\tp & \alpha^\tp & f\\
            -\beta^\tp & -f && f\\
            && -f & -\alpha^\tp \\
            && \beta^\tp & -u_0^\tp
          \end{bmatrix}}
      }
      &&&
      {P_0 \oplus P_1 \oplus P_0 \oplus P_1}
      \ar[u]_-{\rho_{*0}=\can \comp \pr_{12}}
      \ar[d]^-{\delta'_{*0}=\pr_{34}}
      \ar@<0.3ex>[lll]^-{
        \footnotesize{\begin{bmatrix}
            f & \alpha^\tp & f & \\
            -\beta^\tp & u_0^\tp & & f\\
            && u_1^\tp & -\alpha^\tp \\
            && \beta^\tp & f
          \end{bmatrix}}
      }
      \\
      {\mathcal{H}}
      \ar@/^5mm/@{..>}[u]^-{s}
      &
      {P_0 \oplus P_1}
      \ar@/^3mm/@{..>}[u]^-{
        s_1:=
        \footnotesize{\begin{bmatrix}
            0 & 0\\
            0 & 1\\
            1 & 0\\
            0 & 1
          \end{bmatrix}}
      }
      \ar@<0.3ex>[rrr]^-{
        \footnotesize{\begin{bmatrix}
            -f & -\alpha^\tp \\
            \beta^\tp & -u_0^\tp
          \end{bmatrix}}
      }
      &&&
      {P_0 \oplus P_1.}
      \ar@/^3mm/@{..>}[u]^-{
        s_0:=
        \footnotesize{\begin{bmatrix}
            -1 & 0\\
            0 & 0\\
            1 & 0\\
            0 & 1\\
          \end{bmatrix}}
      }
      \ar@<0.3ex>[lll]^-{
        \footnotesize{\begin{bmatrix}
            u_1^\tp & -\alpha^\tp \\
            \beta^\tp & f
          \end{bmatrix}}
      }
    }
  \end{equation*}
  Additionally we have added the dotted morphism $s$ 
  which shows
  that the surjection $\delta'_*$ splits
  in
  $Z_0(\Coh(\Spec A, -W)).$
  Note that $\rho_*$ maps $s(\mathcal{H})$ onto
  $r''^*(\mathcal{E}),$ so we can consider the commutative diagram
  \begin{equation*}
    \xymatrix{
      {\mathcal{E}} \gar[d] \ar[r]^-{r''^*}_-{\sim} 
      & {r''^*(\mathcal{E})} \ar@{}[d]|-{\cap}
      & {s(\mathcal{H})} 
      \ar@{}[d]|-{\cap} 
      \ar[l]_-{\rho_*} 
      \ar[r]^-{\delta'_*}_-{\sim}
      & {\mathcal{H}} \gar[d]
      \\
      {\mathcal{E}} 
      \ar[r]^-{r''^*}
      & {\mathcal{F}}
      & {\mathcal{G}}
      \ar[l]_-{\rho_*} 
      \ar[r]^-{\delta'_*}
      & {\mathcal{H}} 
    }
  \end{equation*}
  in $Z_0(\Coh(\Spec A, -W)).$
  Our aim is to show that the lower row defines an isomorphism
  $\delta'_* \comp (\rho_*)\inv \comp r''^*$ in
  $\DCoh(\Spec A, -W).$ 
  For this it is clearly sufficient to
  show that 
  \begin{equation*}
    \mathcal{H} \xsira{s} s(\mathcal{H}) \xra{\rho_*}
    r''^*(\mathcal{E}) 
  \end{equation*}
  becomes an isomorphism in
  $\DCoh(\Spec A, -W).$  
  This morphism occurs as the epimorphism in the short
  exact sequence 
  \begin{equation*}
    \xymatrix@R=2cm@C=1.5cm{
      {r''^*(\mathcal{E}) \colon }
      &
      {\ol{P}_1}
      \ar@<0.3ex>[rr]^-{\ol{\alpha}^\tp}
      &&
      {\ol{P}_0}
      \ar@<0.3ex>[ll]^-{-\ol{\beta}^\tp}
      \\
      {\mathcal{H} \colon }
      \ar[u]^-{\rho_* \comp s}
      &
      {P_0 \oplus P_1}
      \ar[u]^-{\can \comp
        \footnotesize{\begin{bmatrix}
            0 & 1
          \end{bmatrix}
        }
      }
      \ar@<0.3ex>[rr]^-{
        \footnotesize{
          \begin{bmatrix}
            -f & -\alpha^\tp\\
            \beta^\tp & -u_0^\tp
          \end{bmatrix}
        }
      }
      &&
      {P_0 \oplus P_1}
      \ar[u]_-{\can \comp 
        \footnotesize{
          \begin{bmatrix}
            -1 & 0
          \end{bmatrix}
        }
      }
      \ar@<0.3ex>[ll]^-{
        \footnotesize{
          \begin{bmatrix}
            u_1^\tp & -\alpha^\tp\\
            \beta^\tp & f
          \end{bmatrix}
        }
      }
      \\
      &
      {P_0 \oplus P_1}
      \ar@<0.3ex>[rr]^-{
        \footnotesize{
          \begin{bmatrix}
            1 & \alpha^\tp \\
            -\beta^\tp & fu_0^\tp
          \end{bmatrix}
        }
      }
      \ar[u]^-{
        \footnotesize{
          \begin{bmatrix}
            -1 & 0 \\
            0 & -f
          \end{bmatrix}
        }
      }
      &&
      {P_0 \oplus P_1}
      \ar@<0.3ex>[ll]^-{
        \footnotesize{
          \begin{bmatrix}
            -fu_1^\tp & \alpha^\tp \\
            -\beta^\tp & -1
          \end{bmatrix}
        }
      }
      \ar[u]_-{
        \footnotesize{
          \begin{bmatrix}
            f & 0 \\
            0 & 1
          \end{bmatrix}
        }
      }
    }
  \end{equation*}
  in $Z_0(\Coh(\Spec A, -W)).$
  The lower object in this short exact sequence becomes zero in
  $[\Coh(\Spec A,-W)]$ (use the homotopy with components 
  $h_1=\tzmat 000{-1}$
  and
  $h_0=\tzmat 1000$). This implies that $\rho_* \comp s \colon 
  \mathcal{H} \ra r''^*(\mathcal{E})$ becomes an isomorphism 
  in $\DCoh(\Spec A, -W)$  
  and finishes the proof
  of Proposition~\ref{p:j-ls-and-duality-MF-setting}.
\end{proof}

\begin{remark}
  \label{rem:duality-and-some-subcategories} 
  Lemma~\ref{l:pi-us-and-duality-MF-setting} shows that the
  subcategory $\pi^*\bfMF(X,W) \subset \bfMF(\tildew{X},W)$ is
  invariant under the duality 
  $D_{\tildew{X}} \colon  \bfMF(\tildew{X},W)\ra
  \bfMF(\tildew{X},-W)^\opp.$
  This duality takes the 
  subcategory $\bfMF(Y,W)_l$
  to the subcategory $\bfMF(Y,-W)_{-l-1},$ as follows from
  Proposition~\ref{p:j-ls-and-duality-MF-setting}
  and Lemma~\ref{l:pi-us-and-duality-MF-setting} again.
\end{remark}

\begin{proof}[Proof of \ref{enum:c:semi-orthog-2-mf}]
  The method used to prove \ref{enum:b:semi-orthog-1-mf}
  also shows that 
  $\pi^*\bfMF(X,W) \subset \bfMF(\tildew{X},W)$ is admissible.

  Let us prove that 
  $\bfMF(Y,W)_l \subset \bfMF(\tildew{X},W)$ is admissible.

  From the proof of 
  \ref{enum:b:semi-orthog-1-mf}
  it is clear that the functor
  functor 
  $\mathcal{O} (l) \otimes p^*(-) \colon \bfMF(Y,W)\ra \bfMF(E,W)$ 
  has a right and a left adjoint functor.
  Let us view $j_*$ as a functor 
  $\bfMF(E,W) \ra \bfMF(\tildew{X},W)$; it has a left adjoint
  $j^*$ by 
  Lemma~\ref{l:pi-proper-Rpi-lower-star-and-MF-and-adjunction}.
  Proposition~\ref{p:j-ls-and-duality-MF-setting}
  shows that 
  $F \mapsto D_E(j^*[1](1)D_{\tildew{X}}(F))$ 
  is right adjoint to $j_*.$
  As above
  Remark~\ref{rem:adjoints-and-semi-orthogonal-decomposition} 
  and its dual show 
  that $\bfMF(Y,W)_l \subset \bfMF(\tildew{X},W)$ is admissible.
\end{proof}  


\begin{proof}[Proof of semi-orthogonality in 
  \ref{enum:d:semi-orthog-2-mf}]
  Lemma~\ref{l:orthog-mf} shows that this
  is a direct consequence of
  Theorem~\ref{t:basic-semi-orthog-2}.\ref{enum:c:basic-semi-orthog-2}
  (and this statement is not difficult to prove using the
  local-to-global Ext spectral sequence).
\end{proof}

It remains to
prove completeness 
in \ref{enum:d:semi-orthog-2-mf}.

\begin{proposition}
  \label{p:semi-orthog-Xtilde-reformulated}
  The condition \ref{enum:d:semi-orthog-2-mf}
  is equivalent to the following condition
  \begin{enumerate}[resume*=MFXtilde]
  \item
    \label{enum:e:semi-orthog-2-mf}
    There is a semi-orthogonal decomposition 
    \begin{equation*}
      \bfMF(\tildew{X},-W)=
      \Big\langle \pi^*\bfMF(X,-W),\bfMF(Y,-W)_0,
      \dots, \bfMF(Y,-W)_{r-2} \Big\rangle.
    \end{equation*}
  \end{enumerate}
\end{proposition}

\begin{proof}
  This follows from
  Remark~\ref{rem:duality-and-some-subcategories}. 
\end{proof}

This proof of course shows that
semi-orthogonality holds in \ref{enum:e:semi-orthog-2-mf}
since it holds in 
\ref{enum:d:semi-orthog-2-mf}.
Hence we have to prove 
completeness in 
\ref{enum:e:semi-orthog-2-mf} (obviously we can replace $W$ by
$-W$ there).
Our first aim is to prove the weaker statement
of Proposition~\ref{p:left-orthog-of-pi-0-bis-r-1}
below.

\begin{lemma}
  \label{l:locality-of-being-left-orthogonal}
  Let $B\in \bfMF(\tildew{X},W).$
  \begin{enumerate}
  \item 
    \label{enum:left-orthogonal-to-pi-us}
    The condition 
    $B \in \leftidx{^\perp}{(\pi^*\bfMF(X,W))}{}$
    is equivalent to $\bR \pi_*(D_{\tildew{X}}(B))=0.$
    In particular, it is local
    on $X$ in the following sense:
    if $\mathcal{U}$ is any open covering of $X,$
    then 
    $B \in \leftidx{^\perp}{(\pi^*\bfMF(X,W))}{}$
    if and only if
    $B|_{\pi\inv(U)} \in \leftidx{^\perp}{(\pi^*\bfMF(U,W))}{}$
    for all $U \in \mathcal{U}.$
  \item 
    \label{enum:left-orthogonal-to-j-ls-twist-p-us}
    Let $l \in \DZ.$
    The condition $B\in \leftidx{^\perp}{(\bfMF(Y,W)_l)}{}$ is
    local on $X.$
  \end{enumerate}
\end{lemma}

\begin{proof}  
  \ref{enum:left-orthogonal-to-pi-us}:
  In somewhat risky
  notation we have
  \begin{align}
    \label{eq:Homs-B-pi-us}
    \Hom_{\bfMF(\tildew{X},W)}(B, \pi^*\bfMF(X,W))
    & = \Hom_{\bfMF(\tildew{X},-W)}(\pi^*\bfMF(X,-W), D_{\tildew{X}}(B))\\
    \notag
    & = \Hom_{\bfMF(X,-W)}(\bfMF(X,-W), \bR \pi_*D_{\tildew{X}}(B)).
  \end{align}
  The first equality uses the duality $D_{\tildew{X}}$ and
  Remark~\ref{rem:duality-and-some-subcategories}, the second
  equality uses the adjunction of
  Lemma~\ref{l:pi-proper-Rpi-lower-star-and-MF-and-adjunction}.
  Hence $B\in \leftidx{^\perp}{(\pi^*\bfMF(X,W))}{}$ is
  equivalent to $\bR \pi_*(D_{\tildew{X}}(B))=0,$ and this
  condition is clearly local on $X$
  (use Lemma~\ref{l:locality-of-being-zero}).

  \ref{enum:left-orthogonal-to-j-ls-twist-p-us}:
  We have
  \begin{align*}
    \Hom_{\bfMF(\tildew{X},W)}(B, \bfMF(Y,W)_l)
    & = \Hom_{\bfMF(E,W)}(j^*B, \mathcal{O}_E(l) \otimes
    p^*\bfMF(Y,W))\\ 
    & = \Hom_{\bfMF(E,W)}(\mathcal{O}_E(-l) \otimes j^*B,
    p^*\bfMF(Y,W))\\ 
    & = \Hom_{\bfMF(Y,-W)}(\bfMF(Y,-W),
    \bR p_*(D_E(\mathcal{O}_E(-l) \otimes j^*B)))
  \end{align*}
  The first equality follows from the adjunction of
  Lemma~\ref{l:pi-proper-Rpi-lower-star-and-MF-and-adjunction},
  the second equality is just the twist, and the last
  equality is obtained similarly as
  \eqref{eq:Homs-B-pi-us} (for $p$ instead of $\pi$).
  Clearly the condition 
  $\bR p_*(D_E(\mathcal{O}_E(-l) \otimes j^*B))=0$ is local on $X.$
\end{proof}

\begin{proposition}
  \label{p:left-orthog-of-pi-0-bis-r-1}
  The left orthogonal of the full triangulated subcategory
  \begin{equation*}
    \mathcal{C}:=\tria(\pi^*\bfMF(X,W), \bfMF(Y,W)_0, \dots,
    \bfMF(Y,W)_{r-1}) 
  \end{equation*}
  in $\bfMF(\tildew{X},W)$ is zero, 
  $\leftidx{^\perp}{\mathcal{C}}{}=0.$
\end{proposition}

\begin{proof}
  Let $B \in \leftidx{^\perp}{\mathcal{C}}{}.$
  From
  Lemma~\ref{l:locality-of-being-left-orthogonal}.\ref{enum:left-orthogonal-to-pi-us}
  we obtain
  $\bR \pi_*(D_{\tildew{X}}(B))=0.$ 
  Let $U\subset \tildew{X}$ be the open complement of $E \subset
  \tildew{X}.$ Then
  $\bR \pi_*(D_U(B|_U))=0,$ and 
  the restriction $B|_U$ is zero in $\bfMF(U,W).$

  Recall that the line bundle $\mathcal{I} :=
  \mathcal{O}_{\tildew{X}}(1) \subset 
  \mathcal{O}_{\tildew{X}}$ is the ideal
  sheaf defining $E.$ 
  Note that the obvious morphism
  $B(n)=\mathcal{O}_{\tildew{X}}(n) \otimes B \ra \mathcal{I}^nB$ is an
  isomorphism for all $n \geq 0,$ in particular $\mathcal{I}^n B
  \in \bfMF(\tildew{X},W).$

  Consider for $n \geq 0$ the short exact sequence
  \begin{equation*}
    0 \ra \mathcal{I}^nB \ra B \ra B/\mathcal{I}^nB \ra 0
  \end{equation*}
  in $Z_0(\Coh(\tildew{X},W)).$ 
  Since $B|_U=0,$ Proposition~\ref{p:support-mf} shows that
  $B$ is a direct summand of $B/\mathcal{I}^n B$  
  for $n\gg 0.$ Fix $n\gg 0.$ It suffices to
  prove that $B/\mathcal{I}^nB=0$ in $\DCoh(\tildew{X}, W).$

  Since $B\in \leftidx{^\perp}{\mathcal{C}}{}$ the adjunction
  of Lemma~\ref{l:pi-proper-Rpi-lower-star-and-MF-and-adjunction},
  implies that
  \begin{equation*}
    j^*B\in \leftidx{^\perp}{\tria(
      p^*\bfMF(Y,W),
      \dots,
      \mathcal{O}(r-1)
      \otimes 
      p^*\bfMF(Y,W))}{}.   
  \end{equation*}
  Hence \ref{enum:c:semi-orthog-1-mf} implies that 
  $j^*B =0$ in $\bfMF(E,W).$ But then $B/B(1) \cong
  B/\mathcal{I}B=j_* j^*B=0$ in $\bfMF(\tildew{X}, W).$
  Hence $B(1) \ra B$ becomes an isomorphism
  in $\bfMF(\tildew{X}, W);$ the same is then true for $B(n) \ra
  B,$ and hence $0 = B/B(n) \cong B/\mathcal{I}^n B$ in
  $\DCoh(\tildew{X}, W).$ 
\end{proof}

We give a local description of the inclusion $Y \subset X$ around
a closed point $y \in Y.$
Let $\Spec R$ be an affine open 
neighborhood of $y$ in $X,$
and let $I \subset R$ be the ideal defining $Y \cap \Spec R.$
By possibly
shrinking $\Spec R$ 
we can find 
$r$ elements $x_1, \dots, x_r \in R$ that can be extended to 
a system of uniformizing
parameters on $\Spec R$ such that $I=(x_1, \dots, x_r)$
(this follows for example from 
\cite[Cor.~4.2.15]{Liu} applied to $R$ localized at the maximal
ideal corresponding to $y$).

In the following subsection~\ref{sec:local-considerations}
we prove 
some results,
in particular completeness in \ref{enum:e:semi-orthog-2-mf},
for the
local situation $\Spec R/I \subset \Spec R$
just described.
In subsection~\ref{sec:back-global-setting}
we then deduce completeness in \ref{enum:e:semi-orthog-2-mf}
in the global setting.

\subsubsection{Local considerations}
\label{sec:local-considerations}

Let $R$ be a regular Noetherian $k$-algebra (with $\Spec R$ of
finite Krull dimension) and let 
$I \subset R$ be an ideal that is generated by elements 
$x_1, \dots, x_r \in R$ which are part of a system of
uniformizing parameters on $\Spec R.$ 
Abbreviate $\ol{R}:= R/I.$
We assume in this whole subsection~\ref{sec:local-considerations}
that the inclusion $i \colon Y \ra X$ 
is given by 
$\Spec \ol{R} \subset \Spec R.$
Then $\tildew{X} = \Proj(R \oplus I \oplus I^2 \oplus \dots),$
where $R \oplus I \oplus I^2 \oplus \dots$ is the Rees algebra of
$I \subset R.$ 
We define $y_a:= x_a \in I \subset R \oplus I \oplus I^2 \oplus
\dots,$ i.\,e.\ $y_a$ is $x_a$ considered as an element of
degree 1 in the Rees algebra.
Since $x_1, \dots, x_r$ is a regular sequence in $R,$
the $\ol{R}$-module $I/I^2$ is free with basis the images
$\ol{y}_a$ of the $y_a,$ and the natural map
$\ol{R}[\ol{y}_1, \dots, \ol{y}_r] = \Sym_{R/I}(I/I^2) \ra R/I
\oplus I/I^2 \oplus I^2/I^3 \oplus \dots$ is an isomorphism.
Hence $E = \Proj (R/I \oplus I/I^2 \oplus I^2/I^3
\oplus \dots)=\DP^{r-1}_{\ol{R}}.$

Let
\begin{equation*}
  K_E=\Big(0
  \ra \mathcal{O}_E(-r)
  \xra{\partial_E^{-r}}
  \mathcal{O}_{E}(-r+1)^{\oplus {{r}\choose{r-1}}}
  \xra{\partial_E^{-r+1}}
  \ldots
  \xra{\partial_E^{-2}}
  \mathcal{O}_E(-1)^{\oplus {{r}\choose {1}}}
  \xra{\partial_E^{-1}}
  \mathcal{O}_E \ra
  0\Big),
\end{equation*}
be the acyclic
Koszul complex $K_E$ on $E$
defined as the 
following tensor product of complexes,
\begin{equation*}
  K_E:=\bigotimes_{a=1}^r
  \big(\mathcal{O}_E(-1)
  \xra{\ol{y}_a}
  \mathcal{O}_E\big).
\end{equation*}

\begin{remark}
  \label{rem:diff-forms-and-koszul-komplex}
  The kernel of $\partial_E^{-s}$ is canonically isomorphic
  to the vector bundle $\Omega^s_{E/Y}$
  (for example $\mathcal{O}_E(-r) = \Omega^{r-1}_{E/Y}$). 
  Indeed, it is a nice exercise to show that the complex $K_E$
  can also be obtained as follows: the dual of the Euler sequence
  gives rise to several 
  short exact sequences (see
  \cite[I.1.1.(3)]{okonek-schneider-spindler-vector-bundles});
  combining these in the obvious manner yields a long
  exact sequence which coincides with $K_E.$
\end{remark}

\begin{corollary}
  \label{c:diff-forms-generated}
  For any $s \geq 0$ we have
  \begin{equation*}
    \Omega^s_{E/Y}(s) \otimes p^*\bfMF(Y,W) 
    \subset
    \tria\Big(
    p^*\bfMF(Y,W), \dots, 
    \mathcal{O}_E(s) \otimes p^*\bfMF(Y,W)
    \Big).    
  \end{equation*}
  in $\bfMF(E,W).$
\end{corollary}

\begin{proof}
  Remark~\ref{rem:diff-forms-and-koszul-komplex}
  provides an acyclic subcomplex of $K_E$ which vanishes in
  degrees $<-s-1,$ whose component in degree $-s-1$ is 
  isomorphic to $\Omega^s_{E/Y},$ and which coincides with $K_E$
  in degrees $\geq -s.$ Given an object $M \in \bfMF(Y,W),$ tensor
  $p^*(M)$ with this complex and twist by
  $\mathcal{O}_{E}(s).$
  Now use
  Lemma~\ref{l:totalization}.\ref{enum:brutal-truncation-and-totalization}
  and the method used to prove
  part~\ref{enum:totalization-of-finite-complex-components-null-homotopic}
  of the same lemma.
\end{proof}

\begin{proposition}
  \label{p:cohomology-of-inverse-image}
  The cohomology sheaves of
  $\mathbf{L}\pi^*(i_*\mathcal{O}_Y) 
  \in D^b(\Coh(\tildew{X}))$ are given as follows.
  \begin{enumerate}
  \item
    \label{enum:cohomology-of-inverse-image-inside}
    $H^{-s}(\mathbf{L}\pi^*(i_*\mathcal{O}_Y)=j_*\Omega^s_{E/Y}(s)$
    for $-s \in [-r+1, 0];$ 
  \item
    \label{enum:cohomology-of-inverse-image-outside}
    $H^t(\mathbf{L}\pi^*(i_*\mathcal{O}_Y)=0$ for $t \notin
    [-r+1, 0].$ 
  \end{enumerate}
  In fact, for $-s \in [-r+1, 0],$ there is an isomorphism
  \begin{equation}
    \label{eq:coho-sheaves-and-omega}
    H^{-s}(\bL\pi^*(i_*(-))) \sira j_*(\Omega^s_{E/Y}(s)
    \otimes p^*(-)) 
  \end{equation}
  of functors $\free(Y) \ra \Coh(\tildew{X}),$
  where $\free(Y) \subset \Coh(Y)$ is the full
  subcategory consisting of free $\mathcal{O}_Y$-modules 
  of finite rank, and the functor on the left is the 
  composition
  \begin{equation*}
    \free(Y) \subset \Coh(Y) \subset D^b(\Coh(Y)) \xra{i_*}
    D^b(\Coh(X) \xra{\bL\pi^*} D^b(\Coh(\tildew{X})) \xra{H^{-s}} 
    \Coh(\tildew{X}).
  \end{equation*}
\end{proposition}

\begin{proof}
  Consider the Koszul complex $K_R:=(R;x_1,\dots,x_r)$ which is
  a resolution of $\ol{R},$
  \begin{equation*}
    (K_R \ra \ol{R}) =
    (0 \ra
    R \ra
    R^{\oplus r} \ra
    \ldots \ra
    R^{\oplus {{r}\choose {2}}} \ra
    R^{\oplus r} \ra
    R \ra
    \ol{R} \ra 0)
  \end{equation*}
  Then $\mathbf{L}\pi^*(i_*\mathcal{O}_Y) = \pi^*(K_R).$ 
  This already implies that the cohomology sheaves of 
  $\mathbf{L}\pi^*(i_*\mathcal{O}_Y)$ are zero outside $[-r,0].$
  Note that $K_R$
  is the tensor product of the complexes
  $(R;x_a)=(R \xra{x_a} R),$ hence
  \begin{equation*}
    \pi^*(K_R)
    = \bigotimes_{a=1}^r \big(\mathcal{O}_{\tildew{X}}
    \xra{x_a}
    \mathcal{O}_{\tildew{X}}
    \big).
  \end{equation*}

  We will calculate the cohomology of the complex
  $\pi^*(K_R)$ by comparing it to the acyclic Koszul complex
  \begin{equation*}
    K_{\tildew{X}} =
    \big( 0\ra \mathcal{O}_{\tildew{X}}(-r)
    \ra
    \mathcal{O}_{\tildew{X}}(-r+1)^{\oplus {{r}\choose{r-1}}}
    \ra
    \ldots
    \ra
    \mathcal{O}_{\tildew{X}}(-1)^{\oplus {{r}\choose {1}}}
    \ra
    \mathcal{O}_{\tildew{X}} \ra 0\big)
  \end{equation*}
  which is defined to be the following tensor product,
  \begin{equation*}
    K_{\tildew{X}}:=\bigotimes_{a=1}^r
    \big(\mathcal{O}_{\tildew{X}}(-1)
    \xra{y_a}
    \mathcal{O}_{\tildew{X}}\big).
  \end{equation*}

  Note that the Koszul complex $K_E$ above is the restriction
  $K_{\tildew{X}}|_{E}$ of $K_{\tildew{X}}$ to the divisor $E.$

  Consider the global section $\gamma$
  of the line bundle
  $\mathcal{O}_{\tildew{X}}(-1)$
  defined by
  $\gamma|_{\{y_b\not=0\}}=\frac{x_b}{y_b}$ 
  (on the chart $\{y_b \not= 0\}$) 
  for $1 \leq b \leq r.$
  It corresponds to a morphism $\gamma \colon  \mathcal{O}_{\tildew{X}}
  \ra \mathcal{O}_{\tildew{X}}(-1)$ (which is just
  the $(-1)$-twist of
  the first morphism in \eqref{eq:ses-ideal-sheaf-E}).
  The vertical arrows in the commutative square
  \begin{equation*}
    \xymatrix{
      {\mathcal{O}_{\tildew{X}}} \ar[d]^-{\gamma} \ar[r]^-{x_a} &
      {\mathcal{O}_{\tildew{X}}} \ar[d]^-{\id}\\
      {\mathcal{O}_{\tildew{X}}(-1)} \ar[r]^-{y_a} &
      {\mathcal{O}_{\tildew{X}}}
    }
  \end{equation*}
  define an injective morphism of two-term complexes indexed by
  $1 \leq a \leq r.$ Their tensor product
  is an injective morphism
  $\pi^*(K_R) \ra K.$
  In degree $(-s)$ it is given by the ${r \choose s}$-fold coproduct
  of the map
  \begin{equation*}
    \gamma^{\otimes s} \colon  \mathcal{O}_{\tildew{X}} \ra \mathcal{O}_{\tildew{X}}(-s).
  \end{equation*}
  We denote its cokernel by $\ol{K}$ and
  obtain a short exact sequence
  \begin{equation*}
    0\ra \pi^*(K_R) \ra K_{\tildew{X}} \ra \ol{K}\ra 0
  \end{equation*}
  of complexes of sheaves on $\tildew{X}.$ Its middle term is
  acyclic, so $H^t(\pi^*(K_R))=H^{t-1}(\ol{K}).$ In
  particular, it becomes clear 
  that $H^{-r}(\mathbf{L}\pi^*(i_*\mathcal{O}_Y))$
  vanishes; hence \ref{enum:cohomology-of-inverse-image-outside}
  is proved.

  For $n\geq 0$ denote by $E^n$ the $n$-th
  infinitesimal neighborhood of $E$ in $\tildew{X},$ i.\,e.\ the
  closed subscheme of $\tildew{X}$ defined by the $(n+1)$-st
  power of the ideal sheaf
  $\mathcal{I}=\mathcal{O}_{\tildew{X}}(1)$ of $E.$ The cokernel
  of the map $\gamma^{\otimes s} \colon  \mathcal{O}_{\tildew{X}} \ra
  \mathcal{O}_{\tildew{X}}(-s)$ is $\mathcal{O}_{E^{s-1}}(-s).$
  Hence $\ol{K}^0=0$ and
  $\ol{K}^{-s}=\mathcal{O}_{E^{s-1}}(-s)^{\oplus {r \choose
      s}}={K^{-s}_{\tildew{X}}}|_{E^{s-1}}$ for $s \geq 1.$ Note
  that the complex $\ol{K}$ has the finite descending filtration
  $\ol{K} \supset \mathcal{I}\ol{K} \supset \mathcal{I}^2 \ol{K}
  \supset \dots \supset \mathcal{I}^r\ol{K}=0.$ We include a
  picture
  of
  \begin{equation*}
    \xymatrix@R=0cm{
      {\ol{K} \colon } &
      {\dots} \ar[r] &
      {\mathcal{O}_{E^2}(-3)^{\oplus {r \choose 3}}} \ar[r] &
      {\mathcal{O}_{E^1}(-2)^{\oplus {r \choose 2}}} \ar[r] &
      {\mathcal{O}_{E}(-1)^{\oplus r}} \ar[r] &
      {0,}
    }
  \end{equation*}
  and of the (non-trivial) associated graded
  complexes
  \begin{equation*}
    \xymatrix@R=0cm{
      {\gr^0(\ol{K}) \colon } &
      {\dots} \ar[r] &
      {\mathcal{O}_{E}(-3)^{\oplus {r \choose 3}}} \ar[r]^-{\partial_E^{-3}} &
      {\mathcal{O}_{E}(-2)^{\oplus {r \choose 2}}} \ar[r]^-{\partial_E^{-2}} &
      {\mathcal{O}_{E}(-1)^{\oplus r}} \ar[r] &
      {0,}\\
      {\gr^{1}(\ol{K}) \colon } &
      {\dots} \ar[r] &
      {\mathcal{O}_{E}(-2)^{\oplus {r \choose 3}}} \ar[r]^-{\partial_E^{-3}(1)} &
      {\mathcal{O}_{E}(-1)^{\oplus {r \choose 2}}} \ar[r] &
      {0} \ar[r] &
      {0,}\\
      {\gr^{2}(\ol{K}) \colon } &
      {\dots} \ar[r] &
      {\mathcal{O}_{E}(-1)^{\oplus {r \choose 3}}} \ar[r] &
      {0} \ar[r] &
      {0} \ar[r] &
      {0,}\\
    }
  \end{equation*}
  in degrees between $-3$ and $0.$
  Remark~\ref{rem:diff-forms-and-koszul-komplex} shows that the
  cohomology of $\gr^s(\ol{K})$ is concentrated in degree $-s-1$
  and canonically isomorphic to
  $\Kern(\partial_E^{-s}(s)) = \Omega_{E/Y}^s(s),$ or more
  precisely to $j_*\Omega_{E/Y}^s(s).$

  It is straightforward to see that
  spectral sequence associated to our filtered complex $\ol{K}$
  (whose $E_0$-page is $\gr^*(\ol{K})$ depicted above, up to a
  coordinate change) satisfies $E_1=E_2= \dots =E_\infty,$
  and that the induced filtration on each $H^t(\ol{K})$ has at
  most one non-trivial subquotient.   
  We hence obtain canonically
  \begin{equation}
    \label{eq:structure-sheaf-omega}
    H^{-s}(\bL\pi^*i_*\mathcal{O}_Y)=
    H^{-s}(\pi^*(K_R))=
    H^{-s-1}(\ol{K})
    =
    H^{-s-1}(\gr^s(\ol{K}))
    =
    \Kern(\partial_E^{-s}(s))
    =
    j_*\Omega_{E/Y}^s(s)
  \end{equation}
  for $-s \in [-r+1,0].$ 
  This proves \ref{enum:cohomology-of-inverse-image-inside}.

  It remains to construct the isomorphism
  \eqref{eq:coho-sheaves-and-omega}.
  Given 
  $M=\ol{R}^m$ a free $\ol{R}$-module of
  finite rank, we take the $m$-fold coproduct of the above
  construction and obtain an isomorphism
  $H^{-s}(\bL\pi^*(i_*(M))) \sira j_*(\Omega^s_{E/Y}(s)
  \otimes p^*(M))$ 
  in $\Coh(\tildew{X})$ as the $m$-fold coproduct of
  \eqref{eq:structure-sheaf-omega}.
  This defines
  \eqref{eq:coho-sheaves-and-omega} on objects. We claim that
  this is compatible with morphisms $M \ra N$ in $\free(Y).$ It
  is certainly sufficient to treat the case $M=N=\ol{R}.$
  Then any morphism $M \ra N$ is given by some $\ol{f} \in
  \ol{R}.$ Choose $f \in R$ with image $\ol{f}$ in $\ol{R}.$
  Componentwise multiplication with $f$ lifts $f \colon  M=R \ra N=R$ to
  the Koszul resolution $K_R,$ and we can use this lift to
  compute the image of $f$
  under $H^{-s}(\bL \pi^*(i_*(-)))$ (and this image does not
  depend on the choice of $f$ since all objects in
  \eqref{eq:coho-sheaves-and-omega} are supported on $E$).  The
  image of $f$ under $j_*(\Omega^s_{E/Y}(s) \otimes p^*(-))$ is
  obvious. Now note that all constructions involved in the
  definition of the isomorphism \eqref{eq:coho-sheaves-and-omega}
  are compatible with multiplication by $f.$ This proves our
  claim.
\end{proof}

\begin{corollary}
  \label{c:kuznetsov-method}
  \rule{1mm}{0mm}
  \begin{enumerate}
  \item
    \label{enum:cohomologies-pi-us-i-ls}
    Let $M\in \MF(Y,W)$ and assume that its components $M_0$ and
    $M_1$ are free $\ol{R}$-modules of finite rank. 
    Let $0 \ra Q^n \ra \dots \ra Q^0 \ra i_*(M)$ be an exact
    sequence in $Z_0(\Coh(X, W))$ with all $Q^i \in \MF(X,W),$
    cf.\ Lemma~\ref{l:resolutions}.\ref{enum:MF-reso}.
    Then the cohomologies of
    $\pi^*(Q),$ considered as a complex in $Z_0(\Coh(\tildew{X},
    W)),$   
    are given as follows. 
    \begin{enumerate}
    \item
      \label{enum:cohomology-of-inverse-image-free-comp-inside}
      $H^{-s}(\pi^*(Q)) \cong
      j_*(\Omega^s_{E/Y}(s) \otimes p^*(M))$ 
      in $\Coh(\tildew{X}, W),$
      for $-s \in [-r+1, 0];$
    \item
      \label{enum:cohomology-of-inverse-image-free-comp-outside}
      $H^t(\pi^*(Q))=0$ for $t \notin [-r+1, 0].$
    \end{enumerate}
  \item
    \label{enum:MF-minus-one-contained}
    We have
    \begin{equation*}
      \bfMF(Y,W)_{-1} \subset \tria(\pi^*\bfMF(X,W), \bfMF(Y,W)_0,
      \dots, \bfMF(Y,W)_{r-2}). 
    \end{equation*}
  \end{enumerate}
\end{corollary}

\begin{proof}
  \ref{enum:cohomologies-pi-us-i-ls}:
  The image of the morphisms $m_0 \colon  M_0 \ra M_1$
  and $m_1 \colon  M_1 \ra M_0$ 
  under the functor 
  $H^{-s}(\bL \pi^*(i_*(-))) \colon  \free(Y) \ra \Coh(\tildew{X})$ 
  can be computed using the morphisms 
  $q_0 \colon  Q_0 \ra Q_1$ 
  and $q_1 \colon  Q_1 \ra Q_0$ 
  of complexes in $\Coh(X).$ Now use
  the isomorphism of functors
  \eqref{eq:coho-sheaves-and-omega} 
  (and
  \ref{enum:cohomology-of-inverse-image-outside})
  in Proposition~\ref{p:cohomology-of-inverse-image}.

  \ref{enum:MF-minus-one-contained}
  Let $\mathcal{S}$ be the specified triangulated envelope. 

  Let $M \in \MF(Y,W)$ have free components, and let $Q \ra
  i_*(M)$ be as in 
  \ref{enum:cohomologies-pi-us-i-ls}.
  We claim that 
  $j_*(\mathcal{O}_E(-1) \otimes p^*(M)) \in \mathcal{S}.$

  Note that 
  $j_*(\mathcal{O}_E(-1) \otimes p^*(M)) =
  j_*(\Omega^{r-1}_{E/Y}(r-1) \otimes p^*(M))$ 
  by Remark~\ref{rem:diff-forms-and-koszul-komplex}.
  Hence, 
  by \ref{enum:cohomologies-pi-us-i-ls},
  $j_*(\mathcal{O}_E(-1) \otimes p^*(M))$ is the 
  $(-r+1)$-st cohomology of
  the complex $\pi^*(Q)$
  whose totalization
  $\mathbf{L}\pi^*(i_*M)$
  trivially belongs to $\pi^*\bfMF(X,W).$
  The other cohomologies of this complex are in
  the full triangulated subcategory generated by
  $\bfMF(Y,W)_0, \dots, \bfMF(Y,W)_{r-2},$
  by
  part \ref{enum:cohomologies-pi-us-i-ls} again
  and
  Corollary~\ref{c:diff-forms-generated}. 
  The claim follows
  (by the technique used in the proof of
  Lemma~\ref{l:useful-lemma}.\ref{enum:cohomologies-in-MF-then-Tot}). 

  Now let $N \in \bfMF(Y,W)$ be arbitrary.
  Certainly we find $\ol{R}$-modules $P$ and $Q$ such
  $N_0 \oplus P$ and $N_1 \oplus Q$ are free $\ol{R}$-modules of
  finite rank. Note that
  the components of
  \begin{equation*}
    M:= N \oplus [1]N 
    \oplus (\matfak{P}{1}{P}{W})
    \oplus (\matfak{Q}{1}{Q}{W}) 
  \end{equation*}
  are free $\ol{R}$-modules of finite rank.
  We already know that
  $j_*(\mathcal{O}_E(-1) \otimes p^*(M)) \in \mathcal{S}.$
  Hence 
  $j_*(\mathcal{O}_E(-1) \otimes p^*(N))$ is a direct summand 
  of an object of $\mathcal{S}.$
  But 
  $\mathcal{S}$ is an admissible subcategory of
  $\bfMF(\tildew{X}, W),$
  by  
  Lemma~\ref{l:right-admissible-orthogonals-generate-right-admissible}
  since
  \ref{enum:c:semi-orthog-2-mf} and
  semi-orthogonality in
  \ref{enum:e:semi-orthog-2-mf} are already known.
  In particular, it is a thick subcategory by 
  Corollary~\ref{c:admissibility-and-perpendicularity}.
  Hence 
  $j_*(\mathcal{O}_E(-1) \otimes p^*(N)) \in \mathcal{S}.$
\end{proof}

\begin{proof}[Proof of completeness in 
  \ref{enum:e:semi-orthog-2-mf} (in the local situation)]

  If we twist the semi-orthogonal decomposition in
  \ref{enum:c:semi-orthog-1-mf}
  by $\mathcal{O}(r-2)$
  we see that 
  \begin{equation*}
    \mathcal{O} (r-1)\otimes
    p^*\bfMF(Y,W)
    \subset \tria(\mathcal{O} (-1)\otimes
    p^*\bfMF(Y,W), \dots,
    \mathcal{O} (r-2)\otimes
    p^*\bfMF(Y,W)).
  \end{equation*}
  Apply $j_*$ to this inclusion. This yields the first inclusion
  in
  \begin{align*}
    \bfMF(Y,W)_{r-1}
    & \subset
    \tria(\bfMF(Y,W)_{-1}, \bfMF(Y,W)_{0}, \dots, \bfMF(Y,W)_{r-2})\\
    & \subset
    \mathcal{D}:= \tria(\pi^*\bfMF(X,W), \bfMF(Y,W)_{0}, \dots,
    \bfMF(Y,W)_{r-2}),
  \end{align*}
  and the second inclusion follows from
  Corollary~\ref{c:kuznetsov-method}.\ref{enum:MF-minus-one-contained}. 
  This and 
  Proposition~\ref{p:left-orthog-of-pi-0-bis-r-1} imply that
  $\leftidx{^\perp}{\mathcal{D}}{}=0.$
  Note that $\mathcal{D}$ is admissible by 
  \ref{enum:c:semi-orthog-2-mf} and  
  Lemma~\ref{l:right-admissible-orthogonals-generate-right-admissible} since we already know semi-orthogonality in
  \ref{enum:e:semi-orthog-2-mf}.
  But then Remark~\ref{rem:right-admissible-subcat-equal} shows
  that $\mathcal{D}=\bfMF(\tildew{X},W).$
\end{proof}

Now the proof of Theorem~\ref{t:semi-orthog-2-mf} is complete in
the local situation 
described at the beginning of this
subsection~\ref{sec:local-considerations}.

\subsubsection{Back to the global setting}
\label{sec:back-global-setting}

We now return to the global blowing-up setting 
described in subsection~\ref{sec:blowing-ups}.

\begin{proof}[Proof of completeness in 
  \ref{enum:e:semi-orthog-2-mf} (in the global setting)]
  If $U \subset X$ is an open subscheme, we define
  $\mathcal{S}_U$ to be the subcategory of $\bfMF(\pi\inv(U), W)$
  defined by 
  \begin{equation*}
    \mathcal{S}_U:=
    \tria(\pi^*\bfMF(U, W), \bfMF(Y \cap U, W)_0,
    \dots, \bfMF(Y \cap U, W)_{r-2}).
  \end{equation*}
  Each $\mathcal{S}_U$
  is admissible by
  Lemma~\ref{l:right-admissible-orthogonals-generate-right-admissible}
  since \ref{enum:c:semi-orthog-2-mf}
  and
  semi-orthogonality in
  \ref{enum:e:semi-orthog-2-mf} are already known.
  Let $\mathcal{S}:= \mathcal{S}_X.$
  We need to show that $\mathcal{S}=\bfMF(\tildew{X},W).$
  By Remark~\ref{rem:right-admissible-subcat-equal}
  it suffices to prove that
  the left orthogonal
  $\leftidx{^\perp}{\mathcal{S}}{}$ is zero.

  Let $B \in \leftidx{^\perp}{\mathcal{S}}{}.$
  Lemma~\ref{l:locality-of-being-left-orthogonal}
  shows that $B|_{\pi\inv(U)} \in \leftidx{^\perp}{(\mathcal{S}_U)}{}$
  for all open $U \subset X.$

  Each point of $Y$ has an open neighborhood $U$ in $X$
  such that the inclusion $Y \cap U \subset X \cap U$
  is isomorphic to $\Spec R/I \subset \Spec R$ with $I \subset R$
  as described at the beginning of
  subsection~\ref{sec:local-considerations}.
  Since we already proved
  \ref{enum:e:semi-orthog-2-mf} for this local setting
  we know that
  $\mathcal{S}_U=\bfMF(\pi\inv(U), W).$ 
  Hence $B|_{\pi\inv(U)}=0$ in $\bfMF(\pi\inv(U), W).$

  Trivially we have $\mathcal{S}_{X \setminus Y}=
  \bfMF(\tildew{X} \setminus E, W)$
  and hence $B|_{\tildew{X} \setminus E} =0$ in 
  $\bfMF(\tildew{X} \setminus E).$
  Now Lemma~\ref{l:locality-of-being-zero}
  shows that $B=0$ in $\bfMF(\tildew{X}, W).$
\end{proof}

This finishes the proof of Theorem~\ref{t:semi-orthog-2-mf}.

We will also need the following lift of this result to the dg level.

\begin{corollary}
  \label{c:semi-orthog-2-mf-dg-level}
  In the situation of Theorem~\ref{t:semi-orthog-2-mf},
  there is a
  full dg subcategory 
  $\mathcal{X}$ of
  $\Coh(\tildew{X},W)/\AcyclCoh(\tildew{X},W)$ which is
  quasi-equivalent to 
  $\Coh(X,W)/\AcyclCoh(X,W),$ and
  there are full dg subcategories $\mathcal{Y}'_{l}$ of
  $\Coh(\tildew{X},W)/\AcyclCoh(\tildew{X},W)$ (for $l \in \DZ$)
  which are
  quasi-equivalent to $\Coh(Y,W)/\AcyclCoh(Y,W),$
  such that 
  the semi-orthogonal decomposition into admissible
  subcategories from 
  Theorem~\ref{t:semi-orthog-2-mf}.\ref{enum:d:semi-orthog-2-mf}
  is given by
  \begin{equation*}
    [\Coh(\tildew{X},W)/\AcyclCoh(\tildew{X},W)]= \langle [\mathcal{Y}'_{-r+1}], \dots,
      [\mathcal{Y}'_{-1}], [\mathcal{X}] \rangle
  \end{equation*}
  if we identify $\bfMF(\tildew{X},W) \sira \DCoh(\tildew{X},W)$
  with the left-hand side as explained in
  section~\ref{sec:enhanc-dg-quot}.
\end{corollary}
 
\begin{proof}
  This is proved as Corollary~\ref{c:semi-orthog-1-mf-dg-level}.
  We could have used the dg categories 
  $\Coh(-,?)$ and $\AcyclCoh(-,?)$ instead of
  $\MF(-,?)$ and $\AcyclMF(-,?)$ there. Here we need to do this
  since we have to deal with the functor $j_*(\mathcal{O}(l)
  \otimes p^*(-)).$ 
\end{proof}

\subsection{Applications}
\label{sec:applications}

Certainly we can apply Theorem~\ref{t:semi-orthog-1-mf} to
$\DP^n_k \ra \Spec k$ and $W=0.$ We obtain a semi-orthogonal
decomposition of $\bfMF(\DP^n_k, 0)$ into admissible
subcategories. 
Let us denote the object 
$(\matfak{0}{}{\mathcal{O}_{\DP^n_k}(i)}{}) \in \bfMF(\DP^n_k,0)$
by $\mathcal{O}_{\DP^n_k}(i)$ (by abuse of notation).
Then it is not difficult to see that the objects
\begin{equation*}
  \mathcal{O}_{\DP^n_k}(-n), \dots, \mathcal{O}_{\DP^n_k}
\end{equation*}
define a strong full exceptional collection (in the
$\DZ_2$-graded sense) in $\bfMF(\DP^n_k,0).$
We will explain this in \cite{olaf-folding-derived-categories-in-prep}
using the folding functor.

We mention some corollaries of 
Theorem~\ref{t:semi-orthog-2-mf}.

\begin{corollary} 
  \label{c:semi-orthog-2-mf-zero-set-regular}
  Let $X$ be a scheme satisfying condition~\ref{enum:srNfKd}
  and let
  $\tildew{X}$ be the blowing-up of $X$ along 
  a regular equi-codimenisonal closed
  subscheme $Y$ of codimension $r \geq 2.$
  Let $W \colon X \ra \DA^1$ be a morphism.
 \begin{enumerate}
  \item 
    \label{enum:W-flat-and-Xzero-regular}
    Assume that $W$ is flat and that the scheme-theoretic zero
    fiber $X_0$ of $W \colon  X \ra \DA^1$ 
    is regular.  
    Then the
    category 
    $\bfMF(\tildew{X},W)$ has a semi-orthogonal decomposition
    into $r-1$ admissible subcategories that are all equivalent to 
    $\bfMF(Y,W).$ 
    In particular, if the codimension $r=2,$ then
    $j_*p^* \colon  \bfMF(Y,W) \ra \DCoh(\tildew{X},W)$ induces an
    equivalence 
    $\bfMF(Y,W) \sira \bfMF(\tildew{X},W).$
  \item 
    \label{enum:WresY-flat-and-Yzero-regular}
    Assume that $W|_Y \colon  Y \ra $ is flat and that its scheme
    theoretic zero fiber $Y_0$ is regular.
    Then 
    $\pi^* \colon  \bfMF(X,W) \sira \bfMF(\tildew{X}, W)$
    is an equivalence.
  \item 
    \label{enum:both-flat-and-zero-fibers-regular}
    If both $W$ and $W|_Y$ are flat and have
    regular scheme-theoretic zero fibers $X_0$ and $Y_0,$
    respectively, then $\bfMF(\tildew{X}, W)=0.$
  \end{enumerate}
\end{corollary}

\begin{proof}
  Theorem~\ref{t:factorizations=singularity} shows that
  $\bfMF(X,W)=0$ (resp.\ $\bfMF(Y, W)=0$). 
  All claims then follow from 
  Theorem~\ref{t:semi-orthog-2-mf}.
\end{proof}

\begin{example}
  \label{exam:origin-in-A2}
  Let
  $X=\DA^2_k=\Spec k[x,y],$ $W=x$ and
  $Y=\Spec k[x,y]/(x,y)=\{(0,0)\}.$
  Then 
  Corollary~\ref{c:semi-orthog-2-mf-zero-set-regular}.\ref{enum:W-flat-and-Xzero-regular}
  shows that $\bfMF(\Spec k,0) \sira \bfMF(\tildew{X},W).$
  Write
  $\tildew{X}=\Proj k[x,y][u,v]/(xv-yu)$ and let
  $U \subset \tildew{X}$ be the affine open subset defined by
  $v\not=0.$ Then $U=\Spec k[y,z]=\DA^2_k$ where $z=u/v,$ and
  $W=x=yz.$ 
  Theorem~\ref{t:factorizations=singularity}
  and \cite[Prop.~1.14]{orlov-tri-cat-of-sings-and-d-branes}
  imply that $\bfMF(\tildew{X},W) \ra \bfMF(U,W)$ is an
  equivalence.
  Altogether we obtain an equivalence
  \begin{equation*}
    \bfMF(\Spec k,0) \cong
    \bfMF(\DA^2_k,yz).
  \end{equation*}
  This is, of course, well known.
\end{example}

\begin{definition}
  \label{d:resolved-potential}
  Let $Z$ be a scheme satisfying condition~\ref{enum:srNfKd}
  and let
  let $W \colon Z \ra \DA^1$ be a regular function. We call $W$
  \define{resolved} if
  the ideal sheaf generated by $W$ is locally monomial,
  i.\,e.\ 
  $Z_0=\{W=0\}$ is a simple normal crossing divisor.
  We then also call the
  corresponding category $\bfMF(Z,W)$ \define{resolved}.
\end{definition}

In the rest of this section we assume in addition that
$\charakteristik k=0.$
Let $X$ be a separated connected smooth scheme of finite type and let
$W \colon X \ra \DA^1$ be a non-zero regular function.  
By
\cite[Thm.~3.35]{kollar-singularities}
there exists an "embedded resolution
of singularities" $\pi  \colon \tildew{X}\ra X$ of the divisor
$X_0=\{W=0\}$ such that $W \colon  \tildew{X} \ra \DA^1$ is resolved.
It is obtained by a sequence of
blowing-ups with smooth centers
$Y_1, \dots, Y_s$ which are contained
in the zero sets
of the pullbacks of $W$ (as confirmed to us by J\'anos Koll\'ar).
We can assume that the $Y_i$ are connected.

\begin{corollary} 
  \label{c:any-MF-in-resolved-MF}
  In the above setting the triangulated category
  $\bfMF(\tildew{X}, W)$ has a semi-orthogonal decomposition into
  admissible subcategories that are equivalent to
  $\bfMF(Y_i,0)$ (for $1 \leq i \leq s$) or
  $\bfMF(X,W).$ More precisely,
  the multiplicity of $\bfMF(Y_i,0)$ is equal to 
  the codimension of $Y_i$ minus 1, and 
  $\bfMF(X,W)$ appears with multiplicity one.
  In particular, the category
  $\bfMF(X,W)$ is a semi-orthogonal summand in a resolved
  category $\bfMF(\tildew{X},W).$
\end{corollary}

\begin{proof}
  This follows from the above and
  Theorem~\ref{t:semi-orthog-2-mf}.
\end{proof}

Corollary~\ref{c:any-MF-in-resolved-MF}
may allow us sometimes (depending on the problem
we are interested in) to reduce the study of
the category $\bfMF(X,W)$ to the case that the divisor $X_0$ is a
simple normal crossing divisor.
In view of this result we would like to ask the
following question.

\begin{question}
  \label{q:describe-resolved-MF}
  Can one give a "reasonable" description
  of a resolved category $\bfMF(Z,W)$? Or, at least, of its
  idempotent completion? The simplest non-trivial
  example would be that of the
  category $\bfMF(\DA^2,W=xy^2).$
\end{question}

\appendix

\section{Admissible subcategories and semi-orthogonal decompositions}
\label{sec:app:remind-admiss-subc}

We remind the reader of some definitions and facts from
\cite{bondal-kapranov-representable-functors,
  bondal-larsen-lunts-grothendieck-ring}.
Let $\mathcal{T}$ be a triangulated category.

Let $\mathcal{S} \subset \mathcal{T}$ be a subcategory. 
Recall that the
\define{right orthogonal} $\mathcal{S}^\perp$ to $\mathcal{S}$ in
$\mathcal{T}$ is the full subcategory of $\mathcal{T}$ consisting
of all objects $C \in \mathcal{T}$ such that $\mathcal{T}(S,C)=0$
for all $S\in \mathcal{S}.$
It is a triangulated subcategory of $\mathcal{T}.$ Similarly one
defines the \define{left orthogonal}
$\leftidx{^\perp}{\mathcal{S}}{}.$

\begin{definition}
  \label{d:right-left-admissible-subcats}
  A \define{right admissible} (resp.\ \define{left admissible})
  subcategory of $\mathcal{T}$ is a strict full triangulated
  subcategory $\mathcal{S}$ of $\mathcal{T}$ such that for any $A
  \in \mathcal{T}$ there is a triangle $A_\mathcal{S} \ra A \ra
  A_{\mathcal{S}^\perp} \ra [1]A_\mathcal{S}$ (resp.\
  $A_{\leftidx{^\perp}{\mathcal{S}}{}} \ra A \ra A_{\mathcal{S}}
  \ra [1]A_{\leftidx{^\perp}{\mathcal{S}}{}}$) with
  $A_\mathcal{S} \in \mathcal{S}$ and $A_{\mathcal{S}^\perp} \in
  \mathcal{S}^\perp$ (resp.\ $A_{\leftidx{^\perp}{\mathcal{S}}{}}
  \in \leftidx{^\perp}{\mathcal{S}}{}$).  An \define{admissible}
  subcategory is a subcategory which is both right and left
  admissible.
\end{definition}

\begin{remark}
  \label{rem:right-admissible-subcat-equal}
  Let $\mathcal{S}$ be a right (resp.\ left) admissible
  subcategory of $\mathcal{T}. $ If $\mathcal{S}^\perp = 0$
  (resp.\ $\leftidx{^\perp}{\mathcal{S}}{}=0$), then obviously
  $\mathcal{S} = \mathcal{T}.$
\end{remark}

\begin{lemma}
  [{\cite[Prop.~1.5]{bondal-kapranov-representable-functors}}]
  \label{l:admissibility-TFAE}
  Let $\mathcal{S}$ be a strict full triangulated subcategory of a
  triangulated category $\mathcal{T}.$ Then the following are equivalent.
  \begin{enumerate}
  \item
    $\mathcal{S} $ is right (resp.\ left) admissible.
  \item
    The inclusion functor $\mathcal{S} \hra \mathcal{T}$ has a right
    (resp.\ left) adjoint.
  \item
    $\mathcal{T}$ is the triangulated envelope of $\mathcal{S}$ and
    $\mathcal{S}^\perp$
    (resp.\ of $\leftidx{^\perp}{\mathcal{S}}{}$ and $\mathcal{S}$).
  \end{enumerate}
\end{lemma}

\begin{remark}
  If $\mathcal{S}$ is right (resp.\ left) admissible 
  and we fix for any $A \in \mathcal{T}$ a
  triangle $A_\mathcal{S} \ra A \ra A_{\mathcal{S}^\perp} \ra
  [1]A_\mathcal{S}$ (resp.\
  $A_{\leftidx{^\perp}{\mathcal{S}}{}} \ra A \ra A_{\mathcal{S}} \ra
  [1]A_{\leftidx{^\perp}{\mathcal{S}}{}}$)
  as above, then $A \mapsto A_\mathcal{S}$ 
  extends uniquely to a right (resp.\ left) adjoint functor to
  the inclusion $\mathcal{S} \hra \mathcal{T}.$ 
\end{remark}

\begin{remark}
  \label{rem:adjoints-and-semi-orthogonal-decomposition}
  Let $F \colon  \mathcal{B} \ra \mathcal{T}$ be a full and faithful
  functor of triangulated categories, and assume that $F$
  admits a right adjoint functor. Then the essential image
  of $F$ is a right admissible
  subcategory of $\mathcal{T}.$
  This is obvious from Lemma~\ref{l:admissibility-TFAE}.
\end{remark}

\begin{lemma}
  [{cf.\ \cite[Lemma~2.20]{bondal-larsen-lunts-grothendieck-ring}}]
  \label{l:right-admissible-by-triangulated-and-classical-generators}
  Let $\mathcal{T}$ be a triangulated category, and let
  $\mathcal{U},$ $\mathcal{V}$ be strict full triangulated
  subcategories of 
  $\mathcal{T}$ satisfying $\mathcal{T}(\mathcal{V}, \mathcal{U})=0.$
  Assume that there is a full subcategory $\mathcal{E} \subset
  \mathcal{T}$ such that for each $E \in \mathcal{E}$ there is a
  triangle
  \begin{equation*}
    E_\mathcal{V} \ra E \ra E_{\mathcal{U}} \ra [1]E_\mathcal{V}
  \end{equation*}
  with $E_\mathcal{V} \in \mathcal{V}$ and $E_\mathcal{U} \in
  \mathcal{U}.$
  Assume that one of the following two statements is true.
  \begin{enumerate}
  \item
    \label{enum:T-is-triang-envelope-of-E}
    We have $\tria(\mathcal{E})=\mathcal{T},$ where
    $\tria(\mathcal{E})$ is the triangulated envelope of $\mathcal{E}$ in
    $\mathcal{T}.$
  \item
    \label{enum:T-is-thick-envelope-of-E}
    The categories $\mathcal{U}$ and $\mathcal{V}$ are thick
    subcategories of $\mathcal{T},$ one of $\mathcal{U},$
    $\mathcal{V}$ is idempotent complete,
    and $\thick(\mathcal{E})=\mathcal{T},$ where $\thick(\mathcal{E})$
    is the thick envelope of $\mathcal{E}$ in $\mathcal{T},$
    i.\,e.\ the objects of $\mathcal{E}$ classically generate
    $\mathcal{T}.$
  \end{enumerate}
  Then $\mathcal{V}$ is right admissible in $\mathcal{T},$
  $\mathcal{U}$ is left admissible in $\mathcal{T},$ we have
  $\mathcal{U}=\mathcal{V}^\perp$ and
  $\mathcal{V}=\leftidx{^\perp}{\mathcal{U}}{},$ and
  $\mathcal{T}$ is the triangulated envelope of $\mathcal{U} \cup
  \mathcal{V}.$
  In the terminology of 
  Definition~\ref{d:semi-orthogonal-decomposition}
  below this says that
  $\mathcal{T}=\langle \mathcal{U}, \mathcal{V}\rangle$ is a
  semi-orthogonal decomposition of $\mathcal{T}.$
\end{lemma}

\begin{proof}
  Let $\mathcal{S}$ be the full subcategory of $\mathcal{T}$
  consisting of those objects $X \in \mathcal{T}$ such that there is a
  triangle
  \begin{equation}
    \label{eq:witness-decomposable-in-triangle}
    X_\mathcal{V} \ra X \ra X_\mathcal{U} \ra [1]X_\mathcal{V}
  \end{equation}
  with $X_V \in \mathcal{V}$ and $X_U \in \mathcal{U}.$
  We claim that $\mathcal{S} =\mathcal{T}.$

  Obviously $\mathcal{S}$ is a strict subcategory containing
  $\mathcal{E},$ $\mathcal{V}$ and $\mathcal{U},$ and it is closed under all
  shifts. Assume that $X \ra Y \ra Z \ra [1]X$ is a triangle with $X,
  Y \in \mathcal{S}.$ Assume that there is a triangle
  \eqref{eq:witness-decomposable-in-triangle}
  as above for $X,$
  and similarly for $Y.$ The morphism $X \ra Y$ extends uniquely two a
  morphism between these two triangles
  (use \cite[Prop.~1.1.9]{BBD}), and this morphism fits (since
  it is unique) into the following
  $3 \times 3$-diagram constructed using \cite[Prop.~1.1.11]{BBD}.
  \begin{equation*}
    \xymatrix{
      {[1]X_\mathcal{V}} \ar@{..>}[r] &
      {[1]Y_\mathcal{V}} \ar@{..>}[r] &
      {[1]Z'} \ar@{..>}[r] \ar@{}[rd]|{\anticomm}&
      {[2]X_\mathcal{V}} \\
      {X_\mathcal{U}} \ar[u] \ar[r] &
      {Y_\mathcal{U}} \ar[u] \ar[r] &
      {Z''} \ar[u] \ar[r] &
      {[1]X_\mathcal{U}} \ar@{..>}[u] \\
      {X} \ar[u] \ar[r] &
      {Y} \ar[u] \ar[r] &
      {Z} \ar[u] \ar[r] &
      {[1]X} \ar@{..>}[u] \\
      {X_\mathcal{V}} \ar[u] \ar[r] &
      {Y_\mathcal{V}} \ar[u] \ar[r] &
      {Z'} \ar[u] \ar[r] &
      {[1]X_\mathcal{V}} \ar@{..>}[u]
    }
  \end{equation*}
  Since $\mathcal{U}$ and $\mathcal{V}$ are strict full triangulated
  subcategories of 
  $\mathcal{T},$ we have $Z' \in
  \mathcal{V}$ and $Z'' \in \mathcal{U},$ so $Z \in \mathcal{S}.$
  This argument shows that $\mathcal{S}$ is a strict triangulated
  subcategory 
  of $\mathcal{T}.$
  If \ref{enum:T-is-triang-envelope-of-E} is satisfied this
  already shows 
  that $\mathcal{S}=\mathcal{T}.$

  Now assume that \ref{enum:T-is-thick-envelope-of-E} is satisfied.
  We claim that $\mathcal{S}$ is a thick subcategory.
  Let $X \in \mathcal{S}$ and assume that $X\cong X_1 \oplus X_2$ in
  $\mathcal{T}.$ We can even assume that $X=X_1 \oplus X_2.$
  Let $V \ra X \ra U \ra [1]V$ be a triangle with $V \in \mathcal{V}$
  and $U \in \mathcal{U}.$ Then the idempotent $e:=\tzmat 1000 \colon X \ra
  X$ can be uniquely extended to a morphism
  \begin{equation*}
    \xymatrix{
      {V} \ar[r]^f \ar[d]^v &
      {X} \ar[r]^g \ar[d]^e &
      {U} \ar[r]^h \ar[d]^u &
      {[1]V} \ar[d]^{[1]v} \\
      {V} \ar[r]^f &
      {X} \ar[r]^g &
      {U} \ar[r]^h &
      {[1]V}
    }
  \end{equation*}
  of triangles (\cite[Prop.~1.1.9]{BBD}), and both $u$ and $v$
  are idempotent. Assume that $\mathcal{V}$ is idempotent
  complete. Then we can assume that $V=V_1\oplus V_2$ with $V_1,$
  $V_2 
  \in \mathcal{V}$ and that $v=\tzmat 1000.$ We have
  $f=\tzmat{f_1}00{f_2}$ since $ef=fv.$
  Complete the morphisms $f_i \colon  V_i \ra X_i$ into triangles
  \begin{equation}
    \label{eq:summand-triangle}
    V_i \xra{f_i} X_i \ra U_i \ra [1]V_i,
  \end{equation}
  for $i=1, 2.$
  The direct sum of these two triangles is a triangle, and there
  is a 
  morphism $\varphi$ such that
  \begin{equation*}
    \xymatrix{
      {V_1\oplus V_2} \ar[r]^{f_1\oplus f_2} \gar[d] &
      {X_1\oplus X_2} \ar[r] \gar[d] &
      {U_1\oplus U_2} \ar[r] \ar@{..>}[d]^\varphi &
      {[1](V_1\oplus V_2)} \gar[d] \\
      {V} \ar[r]^f &
      {X} \ar[r]^g &
      {U} \ar[r]^h &
      {[1]V}
    }
  \end{equation*}
  is morphism of triangles; hence $\varphi$ is an isomorphism. Since $\mathcal{U}$ is a
  thick subcategory, we have $U_1,$ $U_2 \in \mathcal{U}.$
  The above triangles \eqref{eq:summand-triangle} for $i=1,$ $2$
  (and the similar argument in case $\mathcal{U}$ is idempotent complete)
  show that $\mathcal{S}$ is a thick subcategory of
  $\mathcal{T}.$ Hence $\mathcal{S}=\mathcal{T}.$

  We have proved that $\mathcal{S}=\mathcal{T}$ if 
  \ref{enum:T-is-triang-envelope-of-E} or
  \ref{enum:T-is-thick-envelope-of-E}
  is satisfied.

  By assumption we have $\mathcal{U} \subset \mathcal{V}^\perp.$
  Let $X \in \mathcal{V}^\perp.$ Since $\mathcal{S}=\mathcal{T}$ we
  have a triangle
  \begin{equation*}
    V \ra X \ra U \ra [1]V
  \end{equation*}
  with $V \in \mathcal{V}$ and $U \in \mathcal{U}.$ Since
  $X \in \mathcal{V}^\perp$ the morphism $V \ra X$ is zero and
  $\id_{[1]V}$ factors through $U$ (in fact $U \cong X \oplus
  [1]V$). But $\mathcal{T}(\mathcal{V}, \mathcal{U})=0$ and hence
  $[1]V =0.$ Hence $X \ra U$ is an isomorphism, and $X \in
  \mathcal{U}$ by strictness.
  This shows $\mathcal{U} = \mathcal{V}^\perp.$
  Similarly we obtain
  $\mathcal{V}=\leftidx{^\perp}{\mathcal{U}}{}.$

  Right admissibility of $\mathcal{V},$ left admissibility of
  $\mathcal{U},$ and the fact that $\mathcal{T}$ is the triangulated
  envelope of $\mathcal{U} \cup \mathcal{V}$ follow directly from the
  definition of $\mathcal{S}$ (cf.\
  \eqref{eq:witness-decomposable-in-triangle})
  and the fact that $\mathcal{S}=\mathcal{T}.$
\end{proof}

\begin{corollary}
  [{\cite[Lemma~1.7]{bondal-kapranov-representable-functors}}]
  \label{c:admissibility-and-perpendicularity}
  If $\mathcal{S}$ is a right admissible subcategory of a $\mathcal{T},$
  then $\mathcal{S}=\leftidx{^\perp}{(\mathcal{S}^\perp)}{},$ so in
  particular $\mathcal{S}$ is a thick subcategory of
  $\mathcal{T}.$
  Similarly, if $\mathcal{S}$ is left admissible,
  then $\mathcal{S}=(\leftidx{^\perp}{\mathcal{S}}{})^\perp$ is thick.
\end{corollary}

\begin{proof}
  The first statement follows from
  Lemma~\ref{l:right-admissible-by-triangulated-and-classical-generators}
  by taking $\mathcal{U}= \mathcal{S}^\perp,$ $\mathcal{V}=
  \mathcal{S}$ and $\mathcal{E}=\mathcal{T}.$
  For the second statement take
  $\mathcal{U}= \mathcal{S},$
  $\mathcal{V}=\leftidx{^\perp}{\mathcal{S}}{}$ and
  $\mathcal{E}=\mathcal{T}.$
\end{proof}

\begin{lemma}
  \label{l:quotients-by-admissibles}
  If $\mathcal{S}$ is right admissible, 
  the functors
  $\mathcal{S}^\perp \ra \mathcal{T}/\mathcal{S}$ and
  $\mathcal{S} \ra \mathcal{T}/\mathcal{S}^\perp$ are
  equivalences.
  If $\mathcal{S}$ is left admissible, 
  the functors
  $\mathcal{S} \ra \mathcal{T}/\leftidx{^\perp}{\mathcal{S}}{}$
  and 
  $\leftidx{^\perp}{\mathcal{S}}{} \ra \mathcal{T}/\mathcal{S}$ 
  are equivalences.
\end{lemma}

\begin{proof}
  By
  parts \ref{enum:verdier-loc-CW-factors-CVW}
  and \ref{enum:verdier-loc-WC-factors-WVC}
  of Proposition~\ref{p:verdier-localization-induced-functor-ff},
  all these functors
  are full and faithful, and it is clear that they are
  essentially surjective.
\end{proof}

\begin{lemma}
  \label{l:right-admissible-orthogonals-generate-right-admissible}
  Let $\mathcal{S}_1,$ $\mathcal{S}_2$ be right admissible
  subcategories of a triangulated category $\mathcal{T}$ and assume
  that
  $\mathcal{T}(\mathcal{S}_2, \mathcal{S}_1)=0.$
  Then the triangulated envelope $\mathcal{D}:=\tria(\mathcal{S}_1, \mathcal{S}_2)$
  in $\mathcal{T}$ of the full subcategory $\mathcal{S}_1 \cup
  \mathcal{S}_2$ is a right admissible subcategory of $\mathcal{T}.$

  Similarly, if $\mathcal{S}_1$ and $\mathcal{S}_2$ are left
  admissible subcategories of $\mathcal{T}$ satisfying
  $\mathcal{T}(\mathcal{S}_2, \mathcal{S}_1)=0,$
  then
  $\tria(\mathcal{S}_1, \mathcal{S}_2)$ is left admissible in $\mathcal{T}.$
\end{lemma}

\begin{proof}
  Let $T \in \mathcal{T}$ be given. By right admissibility of
  $\mathcal{S}_2$ there is a triangle
  \begin{equation*}
    S_2 \ra T \xra{g_2} Q_2 \ra [1]S_2
  \end{equation*}
  with $S_2 \in \mathcal{S}_2$ and $Q_2 \in \mathcal{S}_2^\perp,$
  and right admissibility of $\mathcal{S}_1$ yields a triangle
  \begin{equation*}
    S_1 \ra Q_2 \xra{g_1} Q_1 \ra [1]S_1
  \end{equation*}
  with $S_1 \in \mathcal{S}_1$ and $Q_1 \in \mathcal{S}_1^\perp.$
  Note that $S_1 \in \mathcal{S}_1 \subset \mathcal{S}_2^\perp$ and
  $Q_2 \in \mathcal{S}_2^\perp$ imply that $Q_1 \in
  \mathcal{S}_2^\perp.$
  Hence $Q_1 \in \mathcal{D}^\perp.$
  Fit the composition $g_1g_2$ into a triangle
  \begin{equation}
    \label{eq:triangle-showing-T-right-admissible}
    U \ra T \xra{g_1g_2} Q_1 \ra [1]U
  \end{equation}
  The octahedral axiom applied to the morphisms $g_2$ and $g_1$
  provides a triangle
  \begin{equation*}
    S_2 \ra U \ra S_1 \ra [1]S_2.
  \end{equation*}
  This shows that $U \in
  \mathcal{D}.$
  Hence we see from \eqref{eq:triangle-showing-T-right-admissible}
  that $\mathcal{D}$ is right admissible.
\end{proof}

\begin{definition}
  \label{d:semi-orthogonal-decomposition}
  A sequence $(\mathcal{S}_1, \mathcal{S}_2, \dots
  \mathcal{S}_n)$ of subcategories of $\mathcal{T}$ is called
  \define{semi-orthogonal} if $\mathcal{T}(\mathcal{S}_j,
  \mathcal{S}_i)=0$ for all $j>i,$ and
  \define{complete}
  (in $\mathcal{T}$) if $\mathcal{T}$ is the triangulated
  envelope of 
  $\mathcal{S}_1 \cup \mathcal{S}_2 \cup \dots \cup
  \mathcal{S}_n.$ 
  A \define{semi-orthogonal decomposition}
  of $\mathcal{T}$ is a 
  complete semi-orthogonal sequence
  $(\mathcal{S}_1, \mathcal{S}_2, \dots \mathcal{S}_n)$
  of strict full triangulated subcategories,
  and
  is denoted by
  \begin{equation*}
    \mathcal{T} =\big\langle \mathcal{S}_1, \dots, \mathcal{S}_n
    \big\rangle. 
  \end{equation*}
  A \define{semi-orthogonal decomposition into admissible
    subcategories} 
  is a semi-orthogonal decomposition whose components are 
  admissible subcategories.
\end{definition}

\begin{lemma}
  \label{l:first-properties-semi-orthog-decomp}
  \rule{1mm}{0mm}
  \begin{enumerate}
  \item 
    \label{enum:admissible-and-semi-orth}
    If $\mathcal{S}$ is
    a right admissible subcategory of $\mathcal{T},$ then $\mathcal{T}=\langle \mathcal{S}^\perp,
    \mathcal{S}\rangle$ is a semi-orthogonal decomposition of $\mathcal{T}.$
    Similarly, if $\mathcal{S}$ is left admissible, then 
    $\langle \mathcal{S},
    \leftidx{^\perp}{\mathcal{S}}\rangle$ is a semi-orthogonal
    decomposition. 
  \item
    \label{enum:semi-orth-and-perps}
    If $\mathcal{T}=\langle \mathcal{U}, \mathcal{V} \rangle$ is
    a semi-orthogonal decomposition, then $\mathcal{V}$ is
    right admissible, $\mathcal{U}$ is left admissible,
    $\mathcal{U}=\mathcal{V}^\perp$ and
    $\mathcal{V}=\leftidx{^\perp}{\mathcal{U}}{}.$ 
  \item
    \label{enum:split-semi-orth-decomp}
    Let $\mathcal{T} =\langle \mathcal{S}_1, \dots, \mathcal{S}_n
    \rangle $ be a semi-orthogonal decomposition (into admissible
    subcategories), and let $1 \leq
    a < n.$ Let $\mathcal{D}_1:=\tria(\mathcal{S}_1 \cup \dots
    \cup \mathcal{S}_a)$ and $\mathcal{D}_2:=
    \tria(\mathcal{S}_{a+1} \cup \dots \cup \mathcal{S}_n)$
    denote the indicated triangulated envelopes.  Then
    $\mathcal{T}=\langle \mathcal{D}_1, \mathcal{D}_2\rangle$ and
    $\mathcal{D}_1= \langle \mathcal{S}_1, \dots,
    \mathcal{S}_a\rangle$ and $\mathcal{D}_2= \langle
    \mathcal{S}_{a+1}, \dots, \mathcal{S}_n\rangle$ are
    semi-orthogonal decompositions (into admissible
    subcategories).  In particular, 
    $\mathcal{D}_1= \mathcal{D}_2^\perp$ and
    $\mathcal{D}_2=\leftidx{^\perp}{\mathcal{D}_1}{}.$
  \end{enumerate}
\end{lemma}

\begin{proof}
  \ref{enum:admissible-and-semi-orth}:
  Use Lemma~\ref{l:admissibility-TFAE}.

  \ref{enum:semi-orth-and-perps}: This is a consequence
  Lemma~\ref{l:right-admissible-by-triangulated-and-classical-generators}:
  take 
  $\mathcal{E}=\mathcal{S}_1 \cup \mathcal{S}_2.$

  \ref{enum:split-semi-orth-decomp}:
  If
  $\mathcal{T} =\langle \mathcal{S}_1, \dots, \mathcal{S}_n
  \rangle $ is a semi-orthogonal decomposition, all statements
  are trivial (the last one
  follows directly from
  \ref{enum:semi-orth-and-perps}).
  So let us assume that all components $\mathcal{S}_i$ are
  admissible in 
  $\mathcal{T}.$ 
  Then Lemma~\ref{l:right-admissible-orthogonals-generate-right-admissible}
  implies that $\mathcal{D}_1$ and $\mathcal{D}_2$ are admissible
  subcategories of $\mathcal{T}.$
  Moreover, each $\mathcal{S}_j,$ for $1
  \leq j \leq a$ (resp.~$a+1 \leq j \leq n$), is obviously
  admissible 
  in $\mathcal{D}_1$ (resp.~$\mathcal{D}_2$).
\end{proof}

\begin{corollary}
  \label{c:semi-od-and-karoubi-envelope}
  A semi-orthogonal decomposition $\mathcal{T}=\langle
  \mathcal{U}, \mathcal{V} \rangle$ (into admissibles) induces a
  semi-orthogonal decomposition (into admissibles)
  of the Karoubi envelope $\mathcal{T}^\natural$ of
  $\mathcal{T},$ namely $\mathcal{T}^\natural=\langle
  \mathcal{U}^\natural, \mathcal{V}^\natural \rangle.$
\end{corollary}

\begin{proof}
  Use
  Lemmata~\ref{l:right-admissible-by-triangulated-and-classical-generators}.\ref{enum:T-is-thick-envelope-of-E} and
  \ref{l:first-properties-semi-orthog-decomp}.
\end{proof}

\section{Embeddings of Verdier quotients}
\label{sec:app:embedd-verd-quot}

Verdier localization is described beautifully in
\cite[2.1]{neeman-tricat}. We give here some additional results.
In contrast to \cite{neeman-tricat} we do not assume that
triangulated subcategories are strict (= closed under
isomorphisms).

Let $\mathcal{D}$ be a triangulated category and $\mathcal{C}
\subset \mathcal{D}$ a full triangulated subcategory (not
necessarily thick). 
Let $F \colon \mathcal{D} \ra \mathcal{D}/\mathcal{C}$ be the Verdier 
localization functor (\cite[Theorem~2.1.8]{neeman-tricat}).
We denote by $\Mor_{\mathcal{C}}$ the subclass of morphisms (in
$\mathcal{D}$) that fit into a triangle with cone in
$\mathcal{C}.$

\begin{lemma}
  \label{l:verdier-loc-morphisms-equal}
  Let $f,g \colon X \ra Y$ be two morphisms in $\mathcal{D}.$
  The following conditions are equivalent:
  \begin{enumerate}
  \item
    $F(f)=F(g)$;
  \item
    there is a morphism $\alpha \colon X' \ra X$ in $\Mor_{\mathcal{C}}$ such that 
    $f \alpha= g \alpha \colon  X' \ra Y$;
  \item
    there is a morphism $\beta \colon Y \ra Y'$ in $\Mor_{\mathcal{C}}$ such that 
    $\beta f = \beta g \colon X \ra Y'$;
  \item
    the morphism $f-g \colon X \ra Y$ factors as
    $X \ra C \ra Y$
    with $C$ in $\mathcal{C}.$
  \end{enumerate}
\end{lemma}

\begin{proof}
  This is a slightly extended version of
  \cite[Lemma~2.1.26]{neeman-tricat} using the description of
  morphisms in $\mathcal{D}/\mathcal{C}$ via "coroofs". The proof
  is easily generalized.
\end{proof}

\begin{proposition}
  \label{p:verdier-localization-induced-functor-ff}
  Let $\mathcal{D}$ be a triangulated category with full
  triangulated subcategories $\mathcal{C},$ $\mathcal{W},$
  $\mathcal{V}$ such that $\mathcal{V}$ is contained in both
  $\mathcal{W}$ and $\mathcal{C},$ i.\,e.\ pictorially
  \begin{equation*}
    \xymatrix{
      {\mathcal{V}} \ar@{}[r]|{\subset} \ar@{}[d]|{\cap} &
      {\mathcal{C}}\ar@{}[d]|{\cap}\\
      {\mathcal{W}} \ar@{}[r]|{\subset} &
      {\mathcal{D}.}\\
    }
  \end{equation*}
  Let $i$ be the inclusion $\mathcal{W} \subset \mathcal{D}.$ Then $i$ 
  factors to a triangulated functor $\ol{i} \colon 
  \mathcal{W}/\mathcal{V} \ra  
  \mathcal{D}/\mathcal{C},$ i.\,e.\ pictorially
  \begin{equation*}
    \xymatrix{
      {\mathcal{W}} \ar@{}[r]|{\subset}^i \ar[d]^G &
      {\mathcal{D}} \ar[d]^F \\
      {\mathcal{W}/\mathcal{V}} \ar[r]^{\ol{i}} &
      {\mathcal{D}/\mathcal{C}.}
    }
  \end{equation*}
  where $F$ and $G$ are the Verdier localization functors.

  \begin{enumerate}[label=(\Roman*)]
  \item 
    \label{enum:conditions-for-full-and-faithful}
    The following three conditions are equivalent, and if they
    hold, the functor 
    $\ol{i}$ is full and faithful.
    \begin{enumerate}[label=(ff\arabic*)]
    \item
      \label{enum:verdier-loc-postcompose-to-Mor-V}
      For all morphisms $s \colon  W \ra D$ in $\Mor_{\mathcal{C}}$ 
      with $W$ in $\mathcal{W}$ and $D$ in $\mathcal{D}$
      there is an object $W'$ in $\mathcal{W}$ and a 
      morphism $t \colon D \ra W'$ such that the morphism $ts \colon W \ra W'$ 
      in $\mathcal{W}$ is in $\Mor_{\mathcal{V}}.$
    \item 
      \label{enum:verdier-loc-CW-factors-CVW}
      Any morphism $C \ra W$ with $C \in \mathcal{C}$ and $W \in
      \mathcal{W}$ factors as $C \ra V \ra W$ with $V \in
      \mathcal{V}.$ 

      (Equivalently: For any morphism $s \colon C \ra W$ with $C \in
      \mathcal{C}$ and $W \in \mathcal{W}$ there is an object $W' \in
      \mathcal{W}$ and a morphism $t \colon  W \ra W'$ in $\Mor_\mathcal{V}$
      such that $ts=0.$)
    \item 
      \label{enum:verdier-loc-Hom-DW-in-DmodV-equal-those-in-DmodC}
      For all $D \in \mathcal{D}$ and $W \in \mathcal{W}$ the
      obvious 
      morphism
      \begin{equation}
        \label{eq:verdier-loc-Hom-DW-in-DmodV-equal-those-in-DmodC}
        j: \Hom_{\mathcal{D}/\mathcal{V}}(D,W) \ra
        \Hom_{\mathcal{D}/\mathcal{C}}(D,W)
      \end{equation}
      is bijective.
    \end{enumerate}
    These three conditions hold if the following condition
    \ref{enum:verdier-loc-CW-factors-CVW-via-classical-generation} 
    is satisfied.
    \begin{enumerate}[resume,label=(ff\arabic*)]
    \item 
      \label{enum:verdier-loc-CW-factors-CVW-via-classical-generation}
      $\mathcal{C}$ is classically generated by a collection
      $\mathcal{E}$ of objects in $\mathcal{D},$ i.\,e.\
      $\mathcal{C}=\thick(\mathcal{E}),$ and any morphism $E \ra W$ with
      $E \in \mathcal{E}$ and
      $W \in \mathcal{W}$ factors through an object of $\mathcal{V}.$
    \end{enumerate}

  \item 
    \label{enum:dual-conditions-for-full-and-faithful}
    Dually, 
    the following three conditions are equivalent, and if they hold, the functor
    $\ol{i}$ is full and faithful.
    \begin{enumerate}[label=(ff\arabic*)$^\opp$]
    \item
      \label{enum:verdier-loc-precompose-to-Mor-V}
      For all morphisms $s \colon  D \ra W$ in $\Mor_{\mathcal{C}}$ 
      with $D$ in $\mathcal{D}$ and $W$ in $\mathcal{W}$ 
      there is an object $W'$ in $\mathcal{W}$ and a 
      morphisms $t \colon W' \ra D$ such that the morphism $st \colon W' \ra W$ in 
      $\mathcal{W}$ is in $\Mor_{\mathcal{V}}.$
    \item 
      \label{enum:verdier-loc-WC-factors-WVC}
      Any morphism $W \ra C$ with $W \in \mathcal{W}$ and $C \in
      \mathcal{C}$ factors as $W \ra V \ra C$ with $V \in
      \mathcal{V}.$

      (Equivalently: For any morphism $s \colon W \ra C$ with $W \in \mathcal{W}$ and $C \in
      \mathcal{C}$ there is an object $W' \in \mathcal{W}$ and a
      morphism $t \colon  W' \ra W$ in $\Mor_\mathcal{V}$ such that $st=0.$)
    \item 
      \label{enum:verdier-loc-Hom-WD-in-DmodV-equal-those-in-DmodC}
      For all $W \in \mathcal{W}$ and $D \in \mathcal{D}$ the obvious
      morphism
      \begin{equation*}
        \Hom_{\mathcal{D}/\mathcal{V}}(W,D) \ra
        \Hom_{\mathcal{D}/\mathcal{C}}(W,D)
      \end{equation*}
      is bijective.

    \end{enumerate}
    Moreover, these three conditions hold if
    the following condition
    \ref{enum:verdier-loc-WC-factors-WVC-via-classical-generation} 
    is satisfied.
    \begin{enumerate}[resume, label=(ff\arabic*)$^\opp$]
    \item 
      \label{enum:verdier-loc-WC-factors-WVC-via-classical-generation}
      $\mathcal{C}$ is classically generated by a collection
      $\mathcal{E}$ of objects in $\mathcal{D},$ i.\,e.\
      $\mathcal{C}=\thick(\mathcal{E}),$ and any morphism $W \ra E$ with
      $W \in \mathcal{W}$ and $E \in 
      \mathcal{E}$ factors through an object of $\mathcal{V}.$
    \end{enumerate}
  \end{enumerate}
\end{proposition}

\begin{proof}
  We use implicitly some results of \cite{neeman-tricat}, e.\,g.\ Remark 2.1.23.
  Let $F' \colon  \mathcal{D} \ra \mathcal{D}/\mathcal{V}$ be the Verdier
  localization functor and 
  $j \colon  \mathcal{D}/\mathcal{V} \ra \mathcal{D} /\mathcal{C}$ the
  functor such that $jF'=F.$

  We start with the proof of \ref{enum:conditions-for-full-and-faithful}.

  \ul{\ref{enum:verdier-loc-postcompose-to-Mor-V} implies
    \ref{enum:verdier-loc-Hom-DW-in-DmodV-equal-those-in-DmodC}:}
  Let $D \in \mathcal{D}$ and $W \in \mathcal{W}.$
  We have to prove that
  \eqref{eq:verdier-loc-Hom-DW-in-DmodV-equal-those-in-DmodC} 
  is bijective. 

  Injectivity: Let $h \colon D \ra W$ be a morphism in
  $\mathcal{D}/\mathcal{V}.$  Then $h=F'(f)F'(g)\inv$ for some
  $D'$ in $\mathcal{D}$ and morphisms $D \xla{g} D' \xra{f} W$ 
  (a "roof") in $\mathcal{D}$ with $g \in
  \Mor_{\mathcal{V}}.$

  Assume that $j(h)=0.$ Then $F(f)F(g)\inv=0$ and hence $F(f)=0$;
  it is sufficient to show that $F'(f)=0.$
  Lemma~\ref{l:verdier-loc-morphisms-equal} shows that there is $s \colon  W \ra D''$
  in $\Mor_{\mathcal{C}}$ such that $sf=0 \colon  D' \ra D''.$
  Assumption~\ref{enum:verdier-loc-postcompose-to-Mor-V} applied to $s$
  yields $W'$ in $\mathcal{W}$ and $t \colon D'' \ra W'$ such that $ts \colon W \ra W'$ 
  is in $\Mor_{\mathcal{V}}.$ We obtain that
  $0 =tsf \colon  D' \xra{f} W \xra{ts} W'.$
  This implies $0=F'((ts)f)=F'(ts)F'(f).$ Note that $F'(ts)$ is
  invertible since $ts \in \Mor_{\mathcal{V}}.$ Hence $F'(f)=0.$

  Surjectivity: Let a morphism $a \colon D \ra W$ in
  $\mathcal{D}/\mathcal{C}$ be represented by a "coroof"
  \begin{equation*}
    D \xra{f} D' \xla{s} W
  \end{equation*}
  with $s \in \Mor_{\mathcal{C}}.$
  Assumption~\ref{enum:verdier-loc-postcompose-to-Mor-V} applied to $s$ 
  yields $W'$ in $\mathcal{W}$ and $t \colon D' \ra W'$ such that $ts \in 
  \Mor_{\mathcal{V}}.$
  Our coroof is 
  equivalent to the coroof
  \begin{equation*}
    D \xra{tf} W' \xla{ts} W
  \end{equation*}
  which represents a morphism $b \colon D \ra W$ in
  $\mathcal{D}/\mathcal{V},$ namely $b=F'(ts)\inv F'(tf).$ 
  Since $s$ and $ts$ are in $\Mor_{\mathcal{C}}$ the same is true
  for $t$ by the octahedral axiom.
  Hence
  \begin{equation*}
    j(b)=F(ts)\inv F(tf) = (F(t)F(s))\inv F(t)F(f)=F(s)\inv F(f)=a.
  \end{equation*}

  \ul{\ref{enum:verdier-loc-Hom-DW-in-DmodV-equal-those-in-DmodC}
    implies \ref{enum:verdier-loc-CW-factors-CVW}:} Let a
  morphism $C \ra W$ with $C \in \mathcal{C}$ and $W \in
  \mathcal{W}$ be given. It becomes zero in
  $\mathcal{D}/\mathcal{C}$ by
  Lemma~\ref{l:verdier-loc-morphisms-equal}.  By assumption it
  then becomes already zero in $\mathcal{D}/\mathcal{V}.$
  Lemma~\ref{l:verdier-loc-morphisms-equal} implies that $C \ra
  W$ factors through $\mathcal{V}.$

  \ul{\ref{enum:verdier-loc-CW-factors-CVW} implies
    \ref{enum:verdier-loc-postcompose-to-Mor-V}:}
  Let a morphism $s \colon  W \ra D$ in $\Mor_{\mathcal{C}}$ 
  with $W$ in $\mathcal{W}$ and $D$ in $\mathcal{D}$ be given.
  Fit $s$ into a triangle $W \xra{s} D \ra C \ra [1]W$ with $C \in
  \mathcal{C}.$ By assumption $C \ra [1]W$ factors as $C \ra V
  \ra[1]W$ with $V \in \mathcal{V}.$
  We fit the morphism $V \ra [1]W$ 
  into a triangle $W \ra W' \ra V \ra [1]W$ with $W' \in \mathcal{W}.$
  The partial morphism
  \begin{equation*}
    \xymatrix{
      {W} \ar[r]^s \gar[d] &
      {D} \ar[r] \ar@{..>}[d]^t &
      {C} \ar[r] \ar[d] &
      {[1]W} \gar[d] \\
      {W} \ar[r] &
      {W'} \ar[r] &
      {V} \ar[r] &
      {[1]W}
    }
  \end{equation*}
  can be completed by a morphism $t$ to a morphism of triangles, and
  the morphism $ts$ is the first morphism in the lower triangle and
  hence lies in $\Mor_\mathcal{V}.$

  \ul{\ref{enum:verdier-loc-CW-factors-CVW-via-classical-generation}
    implies \ref{enum:verdier-loc-CW-factors-CVW}:} A morphism $C
  \ra W$ with $C \in \mathcal{C}$ and $W \in \mathcal{W}$ factors
  through an object of $\mathcal{V}$ if and only if $C \ra W$
  becomes the zero morphism in $\mathcal{D}/\mathcal{V},$ by
  Lemma~\ref{l:verdier-loc-morphisms-equal}.  Using this one
  proves that the class of all objects $E'$ such that each
  morphism from $E'$ to an arbitrary object of $\mathcal{W}$
  factors through an object of $\mathcal{V}$ is closed under
  shifts, extensions and direct summands.  This implies the
  claim.

  \ul{\ref{enum:verdier-loc-Hom-DW-in-DmodV-equal-those-in-DmodC} implies
    that $\ol{i}$ is full and faithful:}
  Let $W',$ $W \in \mathcal{W}.$ Since $\ol{i}$ factors as
  $\mathcal{W}/\mathcal{V} \ra \mathcal{D}/\mathcal{V} \ra
  \mathcal{D}/\mathcal{C}$
  it is enough to show that
  \begin{equation*}
    \Hom_{\mathcal{W}/\mathcal{V}}(W',W) \ra
    \Hom_{\mathcal{D}/\mathcal{V}}(W',W)
  \end{equation*}
  is bijective. If $W' \xla{s} D \xra{f} W$ is a roof with $D \in
  \mathcal{D}$ and $s \in \Mor_\mathcal{V}$ representing a morphism in
  $\Hom_{\mathcal{D}/\mathcal{V}}(W',W),$ then $D \xra{s} W'$ fits
  into a triangle with cone in $\mathcal{V} \subset \mathcal{W}.$
  The second and third object of this triangle are in
  $\mathcal{W},$ so the first object $D$ is isomorphic to an
  object $W''$ of $\mathcal{W}.$
  Let $t \colon W'' \sira D$ be an isomorphism.
  Then our roof is isomorphic to the roof $W'
  \xla{st} W'' \xra{ft} W.$ This argument shows that the above
  map is 
  surjective as well as injective.

  We leave the proof of the "dual" statements in
  \ref{enum:dual-conditions-for-full-and-faithful} to the reader.
\end{proof}


\def\cprime{$'$} \def\cprime{$'$} \def\cprime{$'$} \def\cprime{$'$}
  \def\Dbar{\leavevmode\lower.6ex\hbox to 0pt{\hskip-.23ex \accent"16\hss}D}
  \def\cprime{$'$} \def\cprime{$'$}
\providecommand{\bysame}{\leavevmode\hbox to3em{\hrulefill}\thinspace}
\providecommand{\MR}{\relax\ifhmode\unskip\space\fi MR }
\providecommand{\MRhref}[2]{%
  \href{http://www.ams.org/mathscinet-getitem?mr=#1}{#2}
}
\providecommand{\href}[2]{#2}

\end{document}